\documentclass[twoside,11pt]{article}

\usepackage{blindtext}

\usepackage[preprint]{jmlr2e}
\usepackage{algorithm}
\usepackage{algorithmic}

\usepackage{tikz}
\usepackage{amsmath}
\usepackage{subcaption}

\DeclareMathOperator*{\E}{\mathbb{E}}
\DeclareMathOperator*{\R}{\mathbb{R}}

\DeclareMathOperator*{\cov}{\operatorname{Cov}}
\DeclareMathOperator*{\sym}{\operatorname{sym}}

\newcommand{\OUTPUT}{\item[\textbf{Output:}] \hspace{\algorithmicindent}}

\newcommand{\norm}[1]{\left\lVert#1\right\rVert}
\newcommand{\di}{\mathrm{d}}

\newcommand{\N}{\mathbb{N}}

\newcommand*\circled[1]{\tikz[baseline=(char.base)]{
            \node[shape=circle,draw,inner sep=2pt] (char) {#1};}}
\newcommand{\inn}[2]{\left\langle #1, #2 \right\rangle}

\usepackage{lastpage}

\jmlrheading{23}{2024}{1-\pageref{LastPage}}{8/24; Revised 5/22}{9/22}{21-0000}{Linfeng Wang and Nikolas N\"usken}

\ShortHeadings{Measure transport with kernel mean embeddings}{L. Wang and N. N\"usken}
\firstpageno{1}

\begin{document}

\title{Measure transport with kernel mean embeddings}

\author{\name Linfeng Wang \email linfeng.2.wang@kcl.ac.uk \\
       \addr Department of Mathematics\\
       King's College London\\
       Strand, London, WC2R 2LS, UK
       \AND
       \name Nikolas N\"usken\email nikolas.nusken@kcl.ac.uk \\
       \addr Department of Mathematics\\\
       King's College London\\
       Strand, London, WC2R 2LS, UK}

\editor{My editor}

\maketitle

\begin{abstract}
Kalman filters constitute a scalable and robust methodology for approximate Bayesian inference, matching first and second order moments of the target posterior. To improve the accuracy in nonlinear and non-Gaussian settings, we extend this principle to include more or different characteristics, based on kernel mean embeddings (KMEs) of probability measures into reproducing kernel Hilbert spaces. Focusing on the continuous-time setting, we develop a family of interacting particle systems (termed \emph{KME-dynamics}) that bridge between prior and posterior, and that include the Kalman-Bucy filter as a special case. KME-dynamics does not require the score of the target, but rather estimates the score implicitly and intrinsically, and we develop links to score-based generative modeling and importance reweighting. A variant of KME-dynamics has recently been derived from an optimal transport and Fisher-Rao gradient flow perspective by Maurais and Marzouk, and we expose further connections to (kernelised) diffusion maps, leading to a variational formulation of regression type. Finally, we conduct numerical experiments on toy examples and the Lorenz 63 and 96 models, comparing our results against the ensemble Kalman filter and the mapping particle filter (Pulido and van Leeuwen, 2019, J. Comput. Phys.). Our experiments show particular promise for a hybrid modification (called Kalman-adjusted KME-dynamics).
\end{abstract}

\begin{keywords}
  Kernel mean embeddings, measure transport, interacting particle systems, data assimilation, Kalman filtering
\end{keywords}

\section{Introduction}
The task of replicating a time-discrete transition or a time-continuous flow of probability measures by a system of interacting particles is a ubiquitous challenge in computational statistics and machine learning. In a  Bayesian context, we might be presented with samples $(X_0^i)_{i=1}^N$ from the prior $\pi_0$, challenged to transform those into samples $(X^i_1)_{i=1}^N$ from the posterior $\pi_1$. In this paper, we will focus on a variant of this problem in (artificial, algorithmic) continuous time, and consider the curve $(\pi_t)_{t \in [0,1]}$ of distributions
\begin{equation}
\label{eq:interpolation intro}
\pi_t = \frac{e^{-th} \pi_0}{Z_t}, \qquad t \in [0,1],
\end{equation}
see, for example, \cite{heng2021gibbs,reich2011dynamical,vargas2024transport}. In \eqref{eq:interpolation intro}, the function $h:\mathbb{R}^d \rightarrow \mathbb{R}$ stands for the negative log-likelihood (we suppress the dependence on observed data in the notation), and the normalising constants are given by 
\begin{equation}
    \label{eq:definition of normalising constant}
    Z_t = \int_{\mathbb{R}^d} e^{-th} \,  \mathrm{d}\pi_0.
\end{equation} 
For $t=1$, we thus recover Bayes' formula, and \eqref{eq:interpolation intro} provides an interpolation between the prior and the posterior, where the likelihood is introduced in an incremental fashion. 

The central problem of Bayesian posterior approximation may be approached by constructing a dynamical evolution governed by
\begin{equation}
\label{eq:MF intro}
\mathrm{d}X_t = a_t(X_t,\rho_t) \, \mathrm{d}t + \sigma_t(X_t, \rho_t) \, \mathrm{d}W_t, \qquad X_0 \sim \rho_0 = \pi_0,    
\end{equation}
choosing the coefficients $a_t$ and $\sigma_t$ in such a way that the corresponding distributions ${\rho_t = \mathrm{Law} X_t}$ match $\pi_t$ as closely as possible, for all $t \in [0,1]$. We stress the fact that $a_t$ and $\sigma_t$ are allowed to depend on the law $\rho_t$, lending additional flexibility to the framework. Given appropriate coefficients, an interacting-particle and discrete-time approximation of \eqref{eq:MF intro} yields an implementable algorithm that produces approximate samples from $\pi_1$.   

\textbf{The Kalman paradigm.} Enforcing strict equality between \eqref{eq:interpolation intro} and \eqref{eq:MF intro}, that is, $\rho_t = \pi_t$ for all $t \in [0,1]$, is ambitious, as $a_t$ and $\sigma_t$ typically depend on $\rho_t$ and $h$ through high-dimensional PDEs, see Section \ref{sec:diffusion maps} below or,  e.g., eq. (11) in \cite{reich2011dynamical} or eq. (9) in \cite{heng2021gibbs}. To increase scalability and reduce computational burden, it is thus reasonable to require equivalence between \eqref{eq:interpolation intro} and \eqref{eq:MF intro} only according to certain characteristics. The prototypical example is to match first and second moments between $\rho_t$ and $\pi_t$ (see, for example, \citet[Chapter 3]{calvello2022ensemble}), leading to closed-form expressions for $a_t$ and $\sigma_t$, and to the celebrated family of Kalman filters \citep{evensen2022data,reich2015probabilistic,law2015data}. Empirically, Kalman-type methodologies perform robustly and with a moderate computational budget in high-dimensional settings, but the restriction to first and second moments renders them approximate when nonlinearity and/or non-Gaussianity is present, and they may thus turn out unreliable in highly complex data assimilation scenarios.   

\textbf{Embedding probability measures.} In this work, we follow an abstract version of the Kalman principle and consider mappings of the form 
\begin{equation}
\label{eq:abstract embedding}
\Phi: \mathcal{P}(\mathbb{R}^d) \rightarrow \mathcal{H},
\end{equation}
from the set of probability measures $\mathcal{P}(\mathbb{R}^d)$ to an appropriate Hilbert space $\mathcal{H}$. The space $\mathcal{H}$ should be thought of as containing the information $\Phi(\rho)$ that we would like to retain from a probability measure $\rho \in \mathcal{P}(\mathbb{R}^d)$. Accordingly, we then impose $\Phi(\rho_t) = \Phi(\pi_t)$, a requirement often considerably weaker than $\rho_t = \pi_t$. 

\textbf{Kernel mean embeddings.} Throughout this work, we focus on the case when $\mathcal{H}$ is the reproducing kernel Hilbert space (RKHS) $\mathcal{H}_k$ associated to a positive definite kernel $k: \mathbb{R}^d \times \mathbb{R}^d \rightarrow \mathbb{R}$, and $\Phi$ is the corresponding kernel mean embedding (KME) $\Phi_k$, see \cite{muandet2017kernel} for an overview (we summarise relevant background in Section \ref{sec:rkhs}). For the specific choice \begin{equation}
\label{eq:quadratic_kernel}
k_2(x,y) = (x^\top y + 1)^2,
\end{equation}
the corresponding RKHS $\mathcal{H}_{k_2}$ is finite-dimensional, containing information precisely about first and second order moments encoded in the KME $\Phi_{k_2}$ (see Observation \ref{obs:Kalman} below). Unsurprisingly, therefore, our approach based on \eqref{eq:abstract embedding} recovers specific forms of Kalman filters for the choice $\mathcal{H} = \mathcal{H}_{k_2}$, as we demonstrate below in Section \ref{sec:quadratic Kalman}. 

Leveraging the generality of \eqref{eq:abstract embedding}, it is plausible to remedy the limitations of standard Kalman filters by considering larger target spaces than $\mathcal{H}_{k_2}$, enhancing the expressiveness of the map $\Phi$. Indeed, \emph{characteristic} kernels (see Definition \ref{def:characteristic} below) allow full recovery of $\rho \in \mathcal{P}(\mathbb{R}^d)$ from $\Phi_k(\rho) \in \mathcal{H}_k$, the RKHS in this case necessarily being infinite-dimensional. Those kernels represent a full nonparametric description of probability measures through their KMEs, and therefore in principle lead to exact Bayesian inference schemes. A judicious kernel choice (possibly ``in between'' $k_2$ and a characteristic kernel) therefore holds the promise of striking a balance between accuracy, robustness and computational tractability. 

\subsection{Structure and background}
\noindent \textbf{Outline and contributions.} To develop our framework around \eqref{eq:abstract embedding}, we proceed as follows. In Section \ref{sec:rkhs} we survey relevant background on positive definite kernels, reproducing kernel Hilbert spaces, and kernel mean embeddings. In Section \ref{sec:construction}, we detail the construction of our proposed KME-dynamics, culminating in Algorithm \ref{alg:KME dynamics}. In Section \ref{sec:quadratic Kalman}, we establish the connection between KME-dynamics and the continuous-time Kalman-Bucy filter through the quadratic kernel $k_2$ defined in \eqref{eq:quadratic_kernel}. In Sections \ref{sec:Score estimation} and Section \ref{sec:importance sampling scheme}, we link KME-dynamics to score estimators in generative modeling, and leverage this connection to develop an importance-weighting scheme for KME-dynamics. Section \ref{sec:diffusion maps} exposes connections  to related work, in particular to (kernelised) diffusion maps, the Poisson equation, and Tikhonov regularised regression functionals. We test various aspects of KME-dynamics in numerical experiments in Section \ref{sec:experiments}, including one-dimensional toy examples, normalising constant estimation, and a data assimilation task for the Lorenz 63 and 96 model systems.

\textbf{Concurrent and related work.} Close to finishing the writing of this paper, we became aware of the very recent preprint \cite{maurais2024sampling}, extension of the workshop paper \cite{maurais2023adaptive}, which suggests a very similar algorithm,  based on a different derivation centred around optimal transport and the Fisher-Rao gradient flow; the KME-dynamics in Algorithm \ref{alg:KME dynamics} reduces to the KFR-flow proposed by \cite{maurais2024sampling} for $C_t = I_{d \times d}$ and $v^0_t \equiv 0$.  

Prior work has considered Bayesian inference through kernel mean embeddings and the \emph{kernel Bayes' rule}, see, for instance, \cite{fukumizu2013kernel, song2009hilbert, song2013kernel, SongSGS10, xu2022importance} or 
\cite{boots2013hilbert, NishiyamaBGF12} for a control-theoretic context. The kernel Bayes' rule requires a scenario slightly different from ours, in particular the availability of data from the joint parameter-observation distribution to approximate the corresponding covariance operators. Closer to our setting (where we base our inference on the negative log-likelihood $h$) is the \emph{kernel Kalman rule}, see \cite{gebhardt2017kernel,gebhardt2019kernel,sun2023adaptive}, and our approach can be viewed as a continuous-time variant tailored to the interpolation \eqref{eq:interpolation intro}; see Remark  \ref{rem:kernel Bayes} below. We would also like to point the reader to \citet[Section 6]{klebanov2020rigorous}, providing a discussion about Gaussian and non-Gaussian Kalman update rules for KMEs, and to \citet[Section 4.1]{kuang2019sample}, discussing Gaussian optimal transport with quadratic kernels. Both complement our result from Proposition \ref{pro:ReproducingKF}, stating the equivalence of KME-dynamics and the Kalman-Bucy filter for quadratic kernels. In Section \ref{sec:diffusion maps}, we build on the recently introduced \emph{kernelised diffusion maps} \citep{pillaud2023kernelized}, and mention in passing that (unkernelised) diffusion maps \citep{coifman2006diffusion} have been used in Bayesian inference following a similar line of reasoning \citep{taghvaei2020diffusion,pathiraja2021analysis}.

\textbf{Score estimation.} The methodology that we develop in this paper can be employed when the prior $\pi_0$ is only available through a sample-based approximation, $\pi_0 \approx \tfrac{1}{N}\sum_{i=1}^N \delta_{X_0^i}$, as is common in stochastic filtering and data assimilation applications \citep{law2015data,reich2015probabilistic}. This stands in contrast to gradient-based MCMC-type schemes such as Langevin or Hamiltonian Monte Carlo \citep[Chapter 9]{barbu2020monte}, or Stein variational gradient descent \citep{liu2016stein}, which require the score $\nabla \log \pi_0$. At a conceptual level, KME-dynamics approximates the score implicitly and inherently via the identity
\begin{equation}
\label{eq:IBP_ KEM for score estimation}
-\int_{\R^{d}} \nabla_{x}k(\cdot,x) \cdot \nabla_{x} \log \pi(x) \pi(\mathrm{d} x) = \int_{\R^{d}} \Delta_x k(\cdot,x) \pi(\mathrm{d} x) \approx \frac{1}{N} \sum_{i=1}^N \Delta_{X^i} k(\cdot, X^i),
\end{equation}
bypassing the direct computation of $\nabla \log \pi_0$. Leveraging  \eqref{eq:IBP_ KEM for score estimation}, we link KME-dynamics to score estimators used in generative modeling \citep{vincent2011connection,song2021score}, put forward a kernelised version, and develop an importance-weighted KMED scheme (see Sections \ref{sec:Score estimation} and  \ref{sec:importance sampling scheme}).

To put into perspective  such implicit score estimation, it is of interest to numerically compare KME-dynamics to methods that rely on exterior inputs to estimate the score. One  proposal of this kind is the Mapping Particle Filter (MPF) by \citet{pulido2019sequential} that combines Stein Variational Gradient Descent \citep{liu2016stein} either with a Monte-Carlo or a Gaussian estimate of the score (see \cite{stordal2021p} for a comparison). In Section \ref{sec:Lorenz} we observe favourable performance and robustness of KME-dynamics in Lorenz 63 and 96 data assimilation scenarios, supporting the idea that inherent score estimation might be preferable.

\textbf{Other reference dynamics.} The interpolation \eqref{eq:interpolation intro} may be replaced by other reference dynamics according to the application. For example, $\pi_t$ might solve the Kushner-Stratonovich SPDE from stochastic nonlinear filtering with continuous-time observations \citep[Chapter 3]{bain2009fundamentals}. Matching the mean-field dynamics \eqref{eq:MF intro} leads to the \emph{feedback particle filter}, see \cite{yang2013feedback,taghvaei2018kalman,pathiraja2021mckean,coghi2023rough}. For a comprehensive recent overview we refer the reader to \cite{calvello2022ensemble} and references therein.

\noindent \textbf{Notation.} 
Throughout the paper, we denote by $\mathcal{P}(\mathbb{R}^d)$ the set of probability measures on $\mathbb{R}^d$ (equipped with the Borel $\sigma$-algebra). Probability measures and their densities with respect to the Lebesgue measure will be denoted by the same symbol. The Hilbert space of (equivalence classes of) $\pi$-square-integrable functions will be denoted by $L^2(\pi) = \{f:\mathbb{R}^d \rightarrow \mathbb{R}: \,\, \int_{\mathbb{R}^d} f^2 \, \mathrm{d} \pi < \infty \}$, and the corresponding inner product by $\langle \cdot, \cdot \rangle_{L^2(\pi)}$. The space of smooth compactly supported functions on $\mathbb{R}^d$ will be denoted by $C_c^\infty(\mathbb{R}^d)$. The set of symmetric positive definite matrices will be referred to as $\R^{d\times d}_{\text{sym}, >0}$. For $\rho \in \mathcal{P}(\mathbb{R}^d)$, we denote its mean by $\mu(\rho) = \int_{\mathbb{R}^d} x \, \mathrm{d}\rho \in \mathbb{R}^d$, and its covariance matrix by $\cov(\rho) = \int_{\mathbb{R}^d}(x - \mu)(x - \mu)^\top \mathrm{d}\rho \in \mathbb{R}^{d \times d}$.

\section{Reproducing kernel Hilbert spaces and mean embeddings}
\label{sec:rkhs}
In this section we recall relevant notions corresponding to  positive definite kernels; we refer the reader to \cite{smola1998learning,steinwart2008support} for background on reproducing kernel Hilbert spaces, and to \cite{muandet2017kernel} for a survey on kernel mean embeddings.

\textbf{Positive definite kernels and feature maps.} A bivariate function $k:\mathbb{R}^d \times \mathbb{R}^d \rightarrow \mathbb{R}$ is called a \emph{positive definite kernel} if 
\begin{enumerate}
    \item it is \emph{
    symmetric}, that is, $k(x,y) = k(y,x)$ for all $x,y\in \mathbb{R}^d$,
    \item the associated Gram matrices are \emph{positive semi-definite}, that is, for all integers ${N \in \N}$, coefficients $(\alpha_{i})_{i=1}^{N} \subset \R$ and points $(x_{i})_{i=1}^{N}\subset \R^{d}$, we have ${\sum_{i,j = 1}^{N}\alpha_{i}\alpha_{j}k(x_{i},x_{j}) \geq 0}$.
\end{enumerate}

For every positive definite kernel $k$, there exists an  (up to isomorphism) unique associated \emph{reproducing kernel Hilbert space} (RKHS) of functions, denoted  $(\mathcal{H}_{k}, \langle \cdot, \cdot \rangle_{\mathcal{H}_k})$, characterised  by the fact that ${k(\cdot,x) \in \mathcal{H}_k}$ for all $x \in \mathbb{R}^d$, as well as by the \emph{reproducing property} \citep[Theorem 12.11]{wainwright2019high}:
\begin{equation}
\label{eq:rep property}
\nonumber
\inn{f}{k(\cdot ,x)}_{\mathcal{H}_{k}} = f(x), \qquad \text{for all } x \in \mathbb{R}^d, \,\, f \in \mathcal{H}_k.  
\end{equation}
A more constructive approach to understanding  $\mathcal{H}_k$ is as follows \citep[Section 2.3]{kanagawa2018gaussian}: 
For a fixed positive definite kernel $k$, we may consider the \emph{feature map}
\begin{align}
    \phi_k:{\R}^{d} &\rightarrow  \mathcal{H}_{k}, \notag\\
    x & \mapsto  k(\cdot, x),\label{eq:feature map}
\end{align}
which embeds the points from $\mathbb{R}^d$ into the function space $\mathcal{H}_k$, typically encoding a similarity-like notion between $x,y \in \mathbb{R}^d$ in the form of $k(x,y) = \langle k(\cdot, x), k(\cdot, y) \rangle_{\mathcal{H}_k}$. Building on this intuition, the RKHS $\mathcal{H}_k$ can be constructed as an appropriate closure\footnote{The closure is taken with respect to $\langle \cdot, \cdot \rangle_{\mathcal{H}_k} = \Vert \cdot \Vert^2_{\mathcal{H}_k} $, defined in such a way that the reproducing property \eqref{eq:rep property} holds. See \citet[Section 2.3]{kanagawa2018gaussian} for further details.} of the linear span of all features:
\begin{equation}
\nonumber
    \mathcal{H}_k := \overline{\left\{ f = \sum_{i=1}^N \alpha_i k(\cdot, x_i): \quad N \in \mathbb{N}, \quad (\alpha_i)_{i=1}^N \subset \mathbb{R}, \quad (x_i)_{i=1}^N \subset \mathbb{R}^d \right\}}^{\Vert \cdot \Vert_{\mathcal{H}_k}},
\end{equation}
that is, $\mathcal{H}_k$ contains (up to limits) linear combinations of the features $k(\cdot,x)$ that encode the points $x \in \mathbb{R}^d$.

\textbf{Kernel mean embeddings.} Augmenting \eqref{eq:feature map}, the \emph{kernel mean embedding} (KME) proposed by \citet{smola2007hilbert} provides a feature map that embeds elements from the space of probability distributions into $\mathcal{H}_k$:
\begin{align}
    \Phi_{k}:\mathcal{P}({\R}^{d}) &\rightarrow  \mathcal{H}_{k}, \notag\\
    \rho & \mapsto  \int_{\R^{d}} k(\cdot, x) \rho(\mathrm{d}x).\label{eq:def_KME}
\end{align} 
Clearly, $\Phi_k$ extends $\phi_k$, in the sense that $\Phi_k(\delta_x) = \phi_k(x)$, where $\delta_x$ denotes the Dirac measure centred at $x \in \mathbb{R}^d$. Furthermore, expectations with respect to $\rho \in \mathcal{P}(\mathbb{R}^d)$ can be computed using the KME and the inner product in $\mathcal{H}_k$, that is, $\mathbb{E}_{X \sim \rho}[f(X)] = \langle f, \Phi_k(\rho)\rangle_{\mathcal{H}_k}$, for $f \in \mathcal{H}_k$, assuming that the expectation is well defined.

A natural question with particular relevance for the present paper is the extend to which $\rho \in \mathcal{P}(\mathbb{R}^d)$ can be recovered from $\Phi_k(\rho) \in \mathcal{H}_k$. The amount of information preserved by the KME depends on the chosen kernel, and kernels that allow perfect reconstruction are referred to as \emph{characteristic} \citep{fukumizu2004dimensionality,fukumizu2007kernel,sriperumbudur2011universality}:
\begin{definition}[Characteristic kernels]
\label{def:characteristic}
A positive definite kernel $k: \mathbb{R}^d \times \mathbb{R}^d \rightarrow \mathbb{R}$ is called \emph{characteristic} if and only if the associated KME is injective, that is, if ${\Phi_k(\rho) = \Phi_k(\widetilde{\rho})}$ implies $\rho = \widetilde{\rho}$, for all $\rho,\widetilde{\rho} \in \mathcal{P}(\mathbb{R}^d)$.
\end{definition}
According to \cite{fukumizu2007kernel}, the Gaussian (or RBF) kernel
\begin{equation}
\label{eq:Gaussian kernel}
    k(x,y) = \exp \left(- \frac{\Vert x- y\Vert^2 }{2 \sigma^2}\right), \qquad x,y \in \mathbb{R}^d,
\end{equation}
is characteristic. On the other hand, the KME associated to the quadratic kernel $k_2$ defined in \eqref{eq:quadratic_kernel} retains information precisely about first and second order moments, and is thus closely connected to the Kalman methodology:
\begin{observation}[First and second moment matching]
\label{obs:Kalman}
For two probability measures $\rho,\widetilde{\rho} \in \mathcal{P}(\mathbb{R}^d)$, the following are equivalent: 
\begin{enumerate}
\item The KMEs induced by $k_2$ agree: $\Phi_{k_2}(\rho) = \Phi_{k_2}(\widetilde{\rho})$,
\item The measures $\rho$ and $\widetilde{\rho}$ have the same mean vectors and covariance matrices.
\end{enumerate}
In particular, $k_2$ is not characteristic.
\end{observation}
\begin{proof}
Following \citet[Example 3]{sriperumbudur2010hilbert}, this can be shown by directly computing the (squared) \emph{maximum mean discrepancy} 
\begin{align}
\label{eq:k2 calculation}
\nonumber
\mathrm{MMD}^2_{k_2}(\rho,\widetilde{\rho}) & := \Vert \Phi_{k_2}(\rho) - \Phi_{k_2}(\widetilde{\rho}) \Vert^2_{\mathcal{H}_{k_2}}
 = 2 \Vert \mu(\rho) - \mu(\widetilde{\rho}) \Vert^2 
\\
& + \Vert \textrm{Cov}(\rho) - \textrm{Cov}(\widetilde{\rho}) + \mu(\rho)\mu(\rho)^\top - \mu(\widetilde{\rho})\mu(\widetilde{\rho})^\top \Vert^2_{F},
\end{align}
where $\Vert \cdot \Vert_F$ denotes the Frobenius norm, $\mu(\rho)$, $\mu(\widetilde{\rho}) \in \mathbb{R}^d$ denote the means of $\rho$ and $\widetilde{\rho}$, and similarly $\textrm{Cov}(\rho)$ and $\textrm{Cov}(\widetilde{\rho})$ refer to the respective covariance matrices. Clearly, $\mathrm{MMD}^2_{k_2}(\rho,\widetilde{\rho}) = 0$ if and only if  $\Phi_{k_2}(\rho) = \Phi_{k_2}(\widetilde{\rho})$, which by \eqref{eq:k2 calculation} is also equivalent to $\mu(\rho) = \mu(\widetilde{\rho})$ together with $\textrm{Cov}(\rho) = \textrm{Cov}(\widetilde{\rho})$. The fact that different probability measures can have the same mean vectors and covariance matrices shows that $k_2$ is not characteristic.
\end{proof}

\section{RKHS embeddings for mean-field dynamics}
\label{sec:derivation}

\subsection{Constructing KME-dynamics}
\label{sec:construction}

As alluded to in the Introduction, we consider the task of replicating the curve of tempered distributions
\begin{equation}
\label{eq:dist curve}
\pi_t = \frac{e^{-th} \pi_0}{Z_t}, \qquad t \in [0,1],
\end{equation}
by the dynamics of a mean-field ordinary differential equation,
\begin{equation}
\label{eq:ODE}
\frac{\mathrm{d}X_t}{\mathrm{d}t} = v_t(X_t) + v^0_t(X_t), \qquad \qquad X_0 \sim \rho_0 = \pi_0.
\end{equation}
We think of \eqref{eq:ODE} as a parameterisation of the time-marginal distributions $\rho_t := \mathrm{Law}(X_t)$ through the (time-dependent) vector field $v_t$ (possibly itself depending on $\rho_t$). The additional offset $v_t^0$ will be beneficial when prior information on the target dynamics is available, for example when the prior and the likelihood are close to Gaussian (see Section \ref{sec:Lorenz} for an example), but $v^0_t \equiv 0$ is a viable choice. Notice that we have dropped the noise contribution from \eqref{eq:MF intro} by setting $\sigma_t \equiv 0$; allowing for a stochastic component is an interesting direction for future work (and has recently been considered by \citet{maurais2024sampling}).  

Building on Section \ref{sec:rkhs}, we aim to match $(\rho_t)_{t \in [0,1]}$ to $(\pi_t)_{t \in [0,1]}$ by equating their kernel mean embeddings $(\Phi_k(\rho_t))_{t \in [0,1]} \subset \mathcal{H}_k$ and $(\Phi_k(\pi_t))_{t \in [0,1]} \subset \mathcal{H}_k$, for an appropriate choice of positive definite kernel $k: \mathbb{R}^d \times \mathbb{R}^d \rightarrow \mathbb{R}$. Throughout, we work under the following assumption on the compatibility between the distributions defined by \eqref{eq:dist curve}-\eqref{eq:ODE} and the choice of kernel:
\begin{assumption}
\label{ass:k and pi}
The negative log-likelihood $h:\mathbb{R}^d \rightarrow \mathbb{R}$ is measurable, and the likelihood $e^{-h}$ is  $\pi_0$-integrable. The positive definite kernel $k: \mathbb{R}^d \times \mathbb{R}^d \rightarrow \mathbb{R}$ is twice continuously differentiable. Furthermore, $\mathbb{R}^d \ni x \mapsto k(\cdot, x) \in \mathcal{H}_k$ is $\pi_t$-Bochner-integrable for all $t \in [0,1]$, so that the KMEs $(\Phi_k(\pi_t))_{t \in [0,1]}$ are well defined. Finally, we assume that $t \mapsto \Phi_k(\pi_t)$ is differentiable, for all $t \in (0,1)$.  
\end{assumption}
Assumption \ref{ass:k and pi} is relatively mild; the integrability assumptions are satisfied for instance if $k$ has at most polynomial growth, and $\pi_t$ decays exponentially at infinity. 
As a first step towards matching \eqref{eq:ODE} to \eqref{eq:dist curve}, we compute  the evolution equations for the KMEs $\Phi_k(\pi_t)$ and $\Phi_k(\rho_t)$: 
\begin{lemma}[Evolution equations for KMEs]
\label{lem:KME ODE}
The kernel mean embeddings $\Phi_k(\pi_t)$ and $\Phi_k(\rho_t)$ satisfy the $\mathcal{H}_k$-valued ODEs
\begin{subequations}
\begin{align}
\label{eq:KME cov}
\partial_t \Phi_k(\pi_t) & = -{\cov}_{x \sim \pi_t} \left[ \phi_k(x), h(x) \right],
\\
\label{eq:KME dyn}
\partial_t \Phi_k(\rho_t) & = \mathbb{E}_{x \sim \rho_t} \left[ \nabla_x \phi_k(x) \cdot (v_t + v^0_t)(x)\right],
\end{align}
\end{subequations}
where $\phi_k$ denotes the feature map defined in \eqref{eq:feature map}, and the covariance is given by
\begin{equation}
\nonumber
{\cov}_{x \sim \pi_t} \left[ \phi_k(x), h(x) \right] =  \int_{\mathbb{R}^d} k(\cdot, x)  h(x) \pi_t(\mathrm{d}x) - \left( \int_{\mathbb{R}^d} k(\cdot,x) \pi_t( \mathrm{d}x)\right) \left(\int_{\mathbb{R}^d} h \, \mathrm{d}\pi_t \right).
\end{equation}
\end{lemma}
\begin{proof}
Equation \eqref{eq:KME cov} follows by differentiating \eqref{eq:dist curve} and applying the KME defined in  \eqref{eq:def_KME}. Similarly, \eqref{eq:KME dyn} can be obtained using the continuity equation associated to \eqref{eq:ODE}. Detailed calculations can be found in Appendix \ref{app:calc}.  Note that we make use of the fact that derivatives of the feature map $\phi$ are contained in $\mathcal{H}_k$ according to \citet[Corollary 4.36]{steinwart2008support}.  
\end{proof}
\begin{remark}[Kernel Bayes' and Kalman rule]
\label{rem:kernel Bayes}
A discrete-time version of \eqref{eq:KME cov} reads
\begin{equation*}
\Phi_k(\pi_{t+ \Delta t}) = \Phi_k(\pi_t) - \Delta t \cdot {\cov}_{x \sim \pi_t} \left[ \phi_k(x), h(x) \right] + \mathcal{O}(\Delta t^2),    
\end{equation*}
which should be interpreted as an infinitesimal counterpart of the kernel Kalman rule \citep{gebhardt2017kernel,gebhardt2019kernel,sun2023adaptive}, in the context where the non-negative log-likelihood $h$ is available. It is also instructive to compare to the (uncentred) kernel Bayes' rule, see eq. (1.3) in \cite{klebanov2020rigorous}. 
\end{remark}
To proceed, we equate \eqref{eq:KME cov} and \eqref{eq:KME dyn}, obtaining an integral equation for $v_t$:
\begin{equation}
\label{eq:int eq}
-\int_{\mathbb{R}^d} \nabla_x k(\cdot,x) \cdot v_t(x) \rho_t(\mathrm{d}x) = \underbrace{ {\cov}_{x \sim \pi_t} \left[ \phi_k(x), h(x) \right] + \mathbb{E}_{x \sim \rho_t} \left[ \nabla_x \phi_k(x) \cdot  v^0_t(x)\right]}_{=: f_{t} \in \mathcal{H}_k}, 
\end{equation}
where we have introduced the short-hand notation $f_t$ for the right-hand side, which does not depend on $v_t$.
\begin{remark}[Kernelised continuity equation]
\label{rem:kernelised CE}
We refer to \eqref{eq:int eq} as the kernelised continuity equation, because it can be derived by treating the kernel as a test function integrated against the continuity equation $\partial_{t} \pi_{t} + \nabla \cdot (\pi_{t}v_{t})  = 0$, see also \cite{maurais2023adaptive,maurais2024sampling}. Specifically, by computing the time derivative for the tempered path \eqref{eq:dist curve} as $\partial_t \pi_{t} = -\pi_{t} \left(h - \int_{\R^{d}}h \, \mathrm{d} \pi_{t} \right)$ and substituting this time derivative into the continuity equation, we obtain 
\begin{equation}
\label{eq:continuity equation for linear path}
\nabla \cdot (\pi_{t}v_{t}) = \pi_{t} \left(h - \int_{\R^{d}}h \, \mathrm{d} \pi_{t} \right),
\end{equation}
see \cite{daum2010exact,reich2011dynamical,heng2021gibbs}.
The kernelised continuity equation emerges as a weak form of the continuity equation:
\begin{align}
& \nabla \cdot (\pi_{t} v_{t}) = \pi_{t}\left (h - \int_{\mathbb{R}^{d}}h \hspace{1mm} \mathrm{d}\pi_{t} \right ) \notag\\
\implies & \int_{\mathbb{R}^{d}}k(\cdot, x) \, \nabla \cdot (\pi_{t}(x) v_{t}(x)) \, \mathrm{d} x = \int_{\mathbb{R}^{d}}k(\cdot, x) , \pi_{t}(x) \left (h(x) - \int_{\mathbb{R}^{d}}h \hspace{1mm} \mathrm{d}\pi_{t} \right ) \mathrm{d} x \notag\\
\implies & -\int_{\mathbb{R}^d} \nabla_x k(\cdot,x) \cdot v_t(x) \pi_t(\mathrm{d}x) = {\cov}_{x \sim \pi_t} \left[ \phi_k(x), h(x) \right], \label{eq: IBP int eq}
\end{align}
where the last line is obtained from integration by parts.
From \eqref{eq: IBP int eq} to \eqref{eq:int eq}, we simply adjust $v_{t}$ to $v_{t} + v_{t}^{0}$, incorporating a correction term.
\end{remark}
Clearly, solutions to \eqref{eq:int eq} will not be unique: Given a solution $v_t$, the modification $v_t + w_t$ will solve \eqref{eq:int eq} as well, provided that $w_t$ is orthogonal to $\nabla_x k(\cdot,x)$ in $L^2(\rho_t)$.
To resolve this issue (and to obtain a symmetric form on the left-hand side in \eqref{eq:int eq}), we reparameterise $v_t$ through a scalar field $\alpha_t:\mathbb{R}^d \rightarrow \mathbb{R}$, 
\begin{equation}
\label{eq:ansatz v}
v_t(x)  = - C_t  \int_{\mathbb{R}^d} \nabla_x k(x ,y) \alpha_t(y) \rho_t(\mathrm{d}y),
\end{equation}
where $C_t \in \mathbb{R}^{d \times d}_{\mathrm{sym},>0}$ is a (possibly time-dependent) symmetric, positive definite matrix. The point of choosing the ansatz \eqref{eq:ansatz v} is the fact that \eqref{eq:int eq} then takes the form 
\begin{equation}
\label{eq:int G}
\int_{\mathbb{R}^d} G_{\rho_t,C_t} (\cdot,y)   \alpha_t(y)  \rho_t(\mathrm{d}y) = f_{t},
\end{equation}
as an integral equation for $\alpha_t$, 
where we have defined the new ($\rho$ and $C$-dependent) kernel
\begin{equation}
\label{eq:kernel_G}
G_{\rho,C}(x,y):= \int_{\R^{d}}\nabla_{z}k(x,z) \cdot C \nabla_{z}k(z,y)\rho(\di z),
\end{equation}
which is positive definite (see Lemma \ref{lem:G properties} below) due to the judicious (symmetrising) choice of $v_t$ in \eqref{eq:ansatz v}. To obtain a well-posed equation for $\alpha_t$ (and hence for $v_t$), we need the following modification:

\textbf{Tikhonov regularisation.}
Equation \eqref{eq:int G} is a prototypical example of an ill-posed inverse problem, see  \cite{engl1996regularization,kirsch2011introduction}. Indeed, since the operator ${\alpha \mapsto \int_{\mathbb{R}^d} G_{\rho,C} (\cdot,y)   \alpha(y)  \rho(\mathrm{d}y)}$ is compact in $L^2(\rho)$, solutions to \eqref{eq:int G} need not exist, and generically will be unstable with respect to small perturbations of the right-hand side. Therefore, we will need to consider the regularised version 
\begin{equation}
\label{eq:int G reg}
\int_{\mathbb{R}^d} G_{\rho_t,C_t} (\cdot,y)   \alpha_t(y)  \rho_t(\mathrm{d}y) + \varepsilon \alpha_t = f_{t},
\end{equation}
for a small parameter $\varepsilon > 0$. In the broader context of kernel mean embeddings for Bayesian inference, the occurrence of inverse problems is not surprising; typically the solution of an integral equation is required to recover the posterior from its KME (the ``outbedding'' or ``pre-image'' problem; see, for instance, \citet[Section 5.1.2]{gebhardt2019kernel}).

The following lemma justifies the ansatz in \eqref{eq:ansatz v} as well as the regularisation in \eqref{eq:int G reg}:
\begin{lemma}[Properties of $G_{\rho,C}$]
\label{lem:G properties}
Given $\rho \in \mathcal{P}(\mathbb{R}^d)$ and  $C \in \mathbb{R}^{d\times d}_{\mathrm{sym},>0}$, assume that the function ${z \mapsto \nabla_{z}k(x,z) \cdot C \nabla_{z}k(z,y)}$ is $\rho$-integrable, for all $x,y \in \mathbb{R}^d$, so that \eqref{eq:int G} is well defined.
Then the following holds:
\begin{enumerate}
\item  The kernel $G_{\rho,C}:\mathbb{R}^d \times \mathbb{R}^d \rightarrow \mathbb{R}$ is symmetric and positive definite.
\item For any $\varepsilon > 0$ and $f_t \in L^2(\rho_t)$, there exists a unique solution $\alpha_t \in L^2(\rho_t)$ to \eqref{eq:int G reg}.
\end{enumerate}
\end{lemma}
\begin{proof}
See Appendix \ref{app:calc}.
\end{proof}
\begin{remark}[Other types of regularisation]
Instead of \eqref{eq:int G reg}, we may consider the regularised equation  \begin{equation}
\label{eq:new reg}
\int_{\mathbb{R}^d} \left( G_{\rho_t,C_t} (\cdot,y) + \varepsilon k(\cdot,y)\right) \alpha_t(y)  \rho_t(\mathrm{d}y)  = f_{t},
\end{equation}
replacing the identity operator in \eqref{eq:int G reg} by the integral operator $\alpha \mapsto \int k(\cdot, y) \alpha(y) \rho(\mathrm{d}y)$. This form of regularisation is inspired by a variational formulation of regression type (see Section \ref{sec:diffusion maps}), and we have observed that numerically it behaves on par or slightly better than \eqref{eq:int G reg}. Remarkably, while \eqref{eq:new reg} on its own is still ill-posed (because the sum of two compact operators is compact), we show in Section \ref{sec:diffusion maps} that the combined system composed of \eqref{eq:ansatz v}  and \eqref{eq:new reg} is well posed for $v_t$. More systematic studies of regularisation schemes for \eqref{eq:int G} are deferred to future work. 
\end{remark}
Combining equations \eqref{eq:ansatz v} and \eqref{eq:int G reg}, we obtain a two-step procedure for finding an appropriate vector field $v_t$ in the dynamics \eqref{eq:ODE}: We first solve \eqref{eq:int G reg} for $\alpha_t$, and then compute $v_t$ from \eqref{eq:ansatz v}. Notice that the choice of $C_t$ is still arbitrary; in Section \ref{sec:quadratic Kalman} below, we suggest the covariance matrix $C_t = \cov{\rho_t}$, based on a comparison to the Kalman filter.

\textbf{Mean-field dynamics.} To close equation \eqref{eq:int G reg}, we replace $\pi_t$ by $\rho_t$ on the right-hand side, and collect the dynamical equations and those for $v_t$ as follows:
\begin{subequations}
\label{eq:MF dynamics}
\begin{align}
\label{eq:MF ODE}
\frac{\mathrm{d}X_t}{\mathrm{d}t} & = - C_t \int_{\mathbb{R}^d} \alpha_t(y) \nabla_{X_t} k(X_t,y) \, \mathrm{d}\rho_t(y) + v_t^0(X_t), \qquad X_0 \sim \rho_0,
\\
\label{eq:int MF}
\int_{\mathbb{R}^d} & G_{\rho_t,C_t}  (\cdot,y)   \alpha_t(y)  \rho_t(\mathrm{d}y) + \varepsilon \alpha_t   
= f_{t},
\end{align}
\end{subequations}
where the right-hand side of \eqref{eq:int MF} is given by
\begin{subequations}
\begin{align}
 f_{t}  & = {\cov}_{x \sim \rho_t} \left[ \phi_k(x), h(x) \right] + \mathbb{E}_{x \sim \rho_t} \left[ \nabla_x \phi_k(x) \cdot  v^0_t(x)\right]
 \\
 \label{eq:RHS linear system}
 & = \int_{\mathbb{R}^d} k(\cdot,x) \left( h(x) - \int_{\mathbb{R}^d} h \, \mathrm{d}\rho_t \right) \rho(\mathrm{d}x) + \int_{\mathbb{R}^d} \nabla_x k(\cdot,x) \cdot v_t^0(x) \rho_t(\mathrm{d}x).
 \end{align}
\end{subequations}
Notice that \eqref{eq:MF ODE} is a mean-field ODE for $X_t$, since the driving vector field on the right-hand side depends on the distribution $\rho_t = \mathrm{Law}(X_t)$. In particular, $\alpha_t$ depends on $\rho_t$ through \eqref{eq:int MF} since both the kernel $G_{\rho_t,C_t}$ and the right-hand side $f_{t}$ are $\rho_t$-dependent.  

\textbf{Interacting particle system.} It is straightforward to translate \eqref{eq:MF dynamics} into a coupled system of ODEs for an interacting particle system $(X_t^i)_{i=1}^N \subset \mathbb{R}^d$, making the approximation ${\rho_t \approx \tfrac{1}{N} \sum_{i=1}^N \delta_{X_t^i}}$:
\begin{subequations}
\label{eq:IPS}
\begin{align}
\label{eq:IPS ODE}
\frac{\mathrm{d}X_t^i}{\mathrm{d}t} & = -\tfrac{1}{N} C_t \sum_{j=1}^N \alpha^j_t \nabla_{X_t^i} k(X_t^i, X_t^j) + v_t^0(X_t^i), \\
\label{eq:IPS linear_system}
\tfrac{1}{N}&( \mathbf{G}_t  + \varepsilon I_{N\times N}) \boldsymbol{\alpha}_t  = \underbrace{  \mathbf{h}_t^k + \mathbf{v}_t^{0,k}}_{=:\mathbf{f}_t},
\end{align}
\end{subequations}
where $\boldsymbol{\alpha}_t = (\alpha_t^1,\ldots,\alpha_t^N) \in \mathbb{R}^N$ is a weight vector, $I_{N \times N} \in \mathbb{R}^{N \times N}$ refers to the identity matrix, and the entries of the matrix $\mathbf{G}_{t}$ are given by 
\begin{equation}
\label{eq: G mat}
(\mathbf{G}_t)_{ij} =  \tfrac{1}{N} \sum_{l=1}^N \nabla_{X_t^l}  k(X_t^i, X_t^l) \cdot C_t \nabla_{X_t^l} k(X_t^l,X_t^j), \qquad  \qquad i,j = 1, \ldots, N.
\end{equation}
The modified regularisation scheme from \eqref{eq:new reg} can be implemented by replacing $I_{N \times N}$ in \eqref{eq:IPS linear_system} with $\mathbf{K}_t \in \mathbb{R}^{N \times N}$, where $(\mathbf{K}_t)_{ij} = k(X_t^i,X_t^j)$.
Note that in passing from \eqref{eq:int MF} to \eqref{eq:IPS linear_system}, we have also applied the rescaling $\varepsilon \mapsto \varepsilon/N$, which is harmless since $\varepsilon$ is arbitrary at this point and will be specified by the user.
The right-hand side $\mathbf{f}_t = \mathbf{h}_t^k + \mathbf{v}_t^{0,k} \in \mathbb{R}^N$ of \eqref{eq:IPS linear_system} splits into the likelihood term
\begin{equation}
\label{eq:h vec}
(\mathbf{h}^k_t)_i = 
\tfrac{1}{N} \sum_{j=1}^N k(X_t^i,X_t^j) h(X_t^j) - \left(\tfrac{1}{N}\sum_{j=1}^N h(X_t^j)\right) \left( \tfrac{1}{N} \sum_{l=1}^N k(X_t^i,X_t^l)\right), \qquad i = 1,\ldots,N,
\end{equation}
and the correction term
\begin{equation}
\label{eq:cor vec}
(\mathbf{v}_t^{0,k})_i = \tfrac{1}{N}\sum_{j=1}^N \nabla_{X_t^j} k(X_t^i,X_t^j) \cdot v_t^0(X_t^j).
\end{equation}
We expect the dynamics of \eqref{eq:IPS} to be close to those of \eqref{eq:MF dynamics} in an appropriate sense, but a rigorous proof is beyond the scope of the current paper. Discretising \eqref{eq:IPS} in time using the Euler scheme leads to Algorithm \ref{alg:KME dynamics}, summarising the KME-dynamics approach to Bayesian inference.

\begin{algorithm}
\caption{KME-dynamics}
\label{alg:KME dynamics}
\begin{algorithmic}[1] 
\REQUIRE samples $(X^i_0)_{i=1}^N$ from $\pi_0$, negative log-likelihood $h:\mathbb{R}^d \rightarrow \mathbb{R}$, baseline vector field $v_t^0$, positive definite kernel $k: \mathbb{R}^d \times \mathbb{R}^d \rightarrow \mathbb{R}$, number of steps $N_{\mathrm{steps}}$, time-dependent matrix ${C_t \in \mathbb{R}^{d \times d}_{\mathrm{sym},>0}}$, regularisation parameter $\varepsilon$.
\OUTPUT \hspace{-0.5cm} approximate samples $(X^i_1)_{i=1}^N$ from $\pi_1$.
\STATE Set the step size $\Delta t \leftarrow 1/N_{\mathrm{steps}}$.
\FOR{$n=0,\ldots, N_{\mathrm{steps}}-1$}
\STATE Compute $\mathbf{G}_{n\Delta t}$, $\mathbf{h}_{n\Delta t}^k$ and  $\mathbf{v}_{n\Delta t}^{0,k}$ using \eqref{eq: G mat}, \eqref{eq:h vec} and \eqref{eq:cor vec}.
\STATE Solve the linear system \eqref{eq:IPS linear_system} for $\boldsymbol{\alpha}_{n\Delta t}$.
\FOR{$i=1,\ldots, N$}
\STATE $X_{(n+1)\Delta t}^i \leftarrow X_{n\Delta t}^i - \tfrac{\Delta t}{N} \cdot C_{n\Delta t} \sum_{j=1}^N \alpha_{n \Delta t}^j \nabla_{X_{n\Delta t}^i} k(X_{n \Delta t}^i, X_{n \Delta t}^j) + \Delta t \cdot v^0_{n\Delta t}(X_{n\Delta t}^i)$.
\ENDFOR
\ENDFOR
\end{algorithmic}
\end{algorithm}

\subsection{Quadratic kernels and the Kalman update}
\label{sec:quadratic Kalman}
In this section, we discuss the particular case when the Bayesian prior and the likelihood are both Gaussian, 
\begin{equation}
\label{eq:Kalman setting}
\pi_0 = \mathcal{N}(\mu_0, \Sigma_0), \qquad h(x) = \tfrac{1}{2} (Hx-\beta) \cdot R^{-1} (Hx-\beta),
\end{equation}
with prior mean $\mu_0 \in \mathbb{R}^d$ and prior covariance $\Sigma_0 \in \mathbb{R}^{d \times d}$, observation $\beta \in \mathbb{R}^{d'}$, observation operator $H \in \mathbb{R}^{d' \times d}$ and observation noise covariance $R \in \mathbb{R}^{d' \times d'}$, see, for example, \citet[Section 6.1]{reich2015probabilistic}. In this scenario, the posterior $\pi_1$ remains Gaussian, with mean and covariance available in closed form. As shown in \cite{bergemann2010localization,bergemann2010mollified,reich2011dynamical}, the Bayesian update can be implemented in an algorithmically robust way via the mean field Kalman-Bucy ODE
\begin{equation}
\label{eq:KB ODE}
\frac{\mathrm{d}X_t}{\mathrm{d}t} = -\tfrac{1}{2} \cov(\rho_t) H^\top R^{-1}(H X_t + \mu(\rho_t) - 2\beta), \qquad X_0 \sim \pi_0;
\end{equation}
in particular, the solution to \eqref{eq:KB ODE} satisfies $X_1 \sim \pi_1$. In \eqref{eq:KB ODE}, we use $\rho_t = \mathrm{Law} (X_t)$ for the associated distribution, and $\mu(\rho_t)$ for the corresponding mean.  

As suggested by Observation \ref{obs:Kalman} and the construction in Section \ref{sec:construction} based on \eqref{eq:abstract embedding}, there should be a close connection between the KME-dynamics \eqref{eq:MF dynamics} and the Kalman-Bucy ODE \eqref{eq:KB ODE} for the quadratic kernel $k_2$ defined in \eqref{eq:quadratic_kernel}. The following result shows that  indeed the KME-dynamics \eqref{eq:MF dynamics} reduces to \eqref{eq:KB ODE} if we set $C_t$ to be the covariance matrix of $\rho_t$.
\begin{proposition}[From KME-dynamics to Kalman-Bucy]\label{pro:ReproducingKF}
Assume that $\pi_0$ and $h$ are given by \eqref{eq:Kalman setting}, and furthermore that
\begin{equation*}
k(x,y) = (x^\top y + 1)^2, \qquad C_t = \cov(\rho_t).    
\end{equation*}
Then the evolution equations \eqref{eq:MF dynamics} and \eqref{eq:KB ODE}, and hence their solutions, coincide, for $\varepsilon \rightarrow 0$. 
\end{proposition}
\begin{proof}
The proof proceeds by direct (although somewhat lengthy) calculation, see Appendix \ref{app:calc}.    
\end{proof}
\begin{remark}[$\varepsilon \rightarrow 0$]
Following the discussion in Section \ref{sec:construction}, the integral equation \eqref{eq:int G reg} is ill-posed for $\varepsilon = 0$. However, for the quadratic kernel $k_2$, the range of the integral operator ${\alpha \mapsto \int_{\mathbb{R}^d} G_{\rho,C} (\cdot,y)   \alpha(y)  \rho(\mathrm{d}y)}$ is finite dimensional, and so the solution to \eqref{eq:int G} can be interpreted in terms of the generalised Moore-Penrose inverse \citep[Section 2.1]{engl1996regularization}. As $\varepsilon \rightarrow 0$, the solution to \eqref{eq:int G reg} converges towards this limit, and the statement of Proposition \ref{pro:ReproducingKF} should be understood in this sense.
\end{remark}
\begin{remark}
It is worth stressing that the relationship between KME-dynamics and Kalman-Bucy dynamics does not extend to the respective interacting particle systems  (given by \eqref{eq:IPS} for the KME-dynamics and by replacing the mean and covariance in  \eqref{eq:KB ODE} by their estimator versions). Indeed, the proof of Proposition \ref{pro:ReproducingKF} rests on identities between higher-order  moments of Gaussian measures (see Lemma \ref{le:moments_normal}), which are satisfied only approximately for the empirical measures. A more fine-grained exploration of the finite-sample properties of KME-dynamics is left for future work.
\end{remark}
Based on Proposition \ref{pro:ReproducingKF} and the well-documented good performance of Kalman-type methods, we suggest implementing the KME interacting particle system \eqref{eq:IPS} with the choice
\begin{equation}
\label{eq:ensemble covariance}
C_t = \frac{1}{N-1} \sum_{i=1}^N (X_t^i - \overline{X}_t)(X_t^i - \overline{X}_t)^\top,
\end{equation}
the empirical covariance matrix for the ensemble $(X_t^i)_{i=1}^n$. For a numerical comparison between \eqref{eq:ensemble covariance} and $C_t = I_{d \times d}$, we refer to Section \ref{sec:experiment_estimate_NC}, in particular Figure \ref{fig:Bench_d}. 

In settings with Gaussian likelihood, it is furthermore reasonable to use \eqref{eq:KB ODE} as a baseline dynamics; we hence advertise
\begin{equation}
\label{eq:Kalman baseline}
v_t^0(x) = v_{\mathrm{Kalman}}(x) := -\tfrac{1}{2} \cov(\rho_t) H^\top R^{-1}(H x + \mu(\rho_t) - 2\beta),    
\end{equation}
refer to the corresponding interacting particle system as \emph{Kalman-adjusted KME-dynamics},
and evaluate its empirical performance in Section \ref{sec:experiments}. Similar hybrid schemes (combining the Kalman filter with an asymptotically exact method, such as the bootstrap particle filter) have been suggested before \citep{chustagulprom2016hybrid,frei2013bridging,rammelmuller2024adaptive}, but it is generally challenging to retain asymptotic exactness of the hybrid method. In contrast, Kalman-adjusted KME-dynamics is expected to recover the exact posterior when $N \rightarrow \infty$ and $\varepsilon \rightarrow 0$.
\begin{remark}[Affine invariance] The choice $C_t = \cov(\rho_t)$ is not sufficient to make the KME-dynamics \eqref{eq:MF dynamics} affine-invariant in the sense of \cite{goodman2010ensemble,chen2023gradient,garbuno2020affine}; this would require appropriate scaling of the kernel $k$. 
\end{remark}
\subsection{Estimating the score}\label{sec:Score estimation}

In recent years, encoding probability measures $\pi \in \mathcal{P}(\mathbb{R}^d)$ in terms of their scores $\nabla \log \pi$ has become a central theme in computational statistics and machine learning \citep{hyvarinen2005estimation,song2021score,ho2020denoising}. Here, our motivation is twofold: Firstly, we would like to expand on the idea that KME-dynamics estimates the score implicitly, as already alluded to in the Introduction (see the discussion around equation \eqref{eq:IBP_ KEM for score estimation}). Secondly, we propose a KME-based estimator for $\nabla \log \pi$ given representative samples. Apart from the general interest in score estimators, this is also a prerequisite for Section \ref{sec:importance sampling scheme}, where we develop an importance weighting scheme to enhance the accuracy of KME-dynamics.

To begin, we recall the kernel embedding of the score from equation \eqref{eq:IBP_ KEM for score estimation}:
\begin{equation}
\label{eq:IBP_ KME for score estimation _main body}
-\int_{\R^{d}} \nabla_{x}k(\cdot,x) \cdot \nabla_{x} \log \pi(x) \pi(\mathrm{d} x) = \int_{\R^{d}} \Delta_x k(\cdot,x) \pi(\mathrm{d} x).
\end{equation}
The relation \eqref{eq:IBP_ KME for score estimation _main body}  is a kernelised version 
of
\begin{equation}
\label{eq:stationary FKP}
\nabla \cdot (\pi \nabla \log \pi) = \Delta \pi,   
\end{equation}
which itself can be interpreted in two ways: First, it is the stationary
Fokker-Planck equation for the overdamped Langevin diffusion 
\begin{equation}
\label{eq:stochastic ODE}
\mathrm{d}X_t = \nabla \log\pi(X_t)\,\mathrm{d}t + \sqrt{2}\,\mathrm{d}W_{t},
\end{equation}
governing the equilibrium steady state (set $\partial_t \pi = 0$, cf. \citet[Chapter 4]{pavliotis2016stochastic}). Secondly, \eqref{eq:stationary FKP} describes the equivalence in law of the Brownian motion $\mathrm{d}X_t =  \sqrt{2}  \, \mathrm{d}W_t$ and the probability flow ODE $\mathrm{d}X_t =  -\nabla \log \pi_t(X_t) \, \mathrm{d}t$, see  \cite{song2021score,chen2024probability}. Both viewpoints are  naturally connected to the main thread of this work, matching a dynamical evolution to a fixed curve of distributions: Indeed, the task of estimating the score $\nabla \log \pi$ can be rephrased as determining the drift of an overdamped Langevin diffusion at equilibrium, or as finding an ODE that reproduces the law of Brownian motion with a given initial distribution. Hence, we can follow Section \ref{sec:construction} and choose the approximation $\nabla \log \pi(x) \approx -C\int_{\R^{d}}\nabla_{x}k(x, y) \alpha(y) \pi(\mathrm{d}y)$  as in \eqref{eq:ansatz v}, with $C \in \mathbb{R}_{\mathrm{sym},>0}^{d \times d}$ an arbitrary symmetric positive definite matrix. Substituting this into \eqref{eq:IBP_ KME for score estimation _main body} and following the reasoning from Section \ref{sec:construction}, we obtain the estimator
\begin{subequations}
\label{eq:score_estimation}
\begin{align}
\label{eq:score_representation}
\widehat{\nabla \log \pi}(x) & = - C \int_{\mathbb{R}^d} \alpha(y) \nabla_{x} k(x,y) \,\pi(\mathrm{d} y),\\
\label{eq:int score estimation}
\int_{\mathbb{R}^d} G_{\pi,C} (\cdot,y) \alpha(y) \pi(\mathrm{d}y) + \varepsilon \alpha &= \int_{\R^{d}} \Delta_y k(\cdot,y) \pi( \mathrm{d} y),
\end{align}
\end{subequations}
where we recall the definition of $G_{\pi,C}$ from \eqref{eq:kernel_G}.

\textbf{Finite sample approximation.} Practically (for instance in Section \ref{sec:importance sampling scheme} below), $\pi$ is often approximated by weighted samples 
\begin{equation}
\label{eq:weighted pi}
\pi \approx \frac{\sum_{i=1}^{N}w^{i}\delta_{X^{i}}}{\sum_{j=1}^{N}w^{j}},
\end{equation}
where $w^{i} \geq 0$ are unnormalised weights assigned to the samples $(X^{i})_{i=1}^N$. For notational convenience, we denote by $W^{i}:= \frac{w^{i}}{\sum_{j=1}^{N}w^{j}}$ the normalised weights, ensuring that ${\sum_{i=1}^{N}W^{i} = 1}$. Using \eqref{eq:weighted pi}, 
it is straightforward to translate \eqref{eq:score_estimation} into a finite sample estimator,
\begin{equation}
\label{eq:score estimator}
\widehat{\nabla \log \pi}(x)  = - \tfrac{1}{N}C\sum_{j=1}^{N} \alpha^{j} \nabla_{x} k(x,X^{j}), \qquad \qquad 
\tfrac{1}{N}(\mathbf{G} + \varepsilon I_{N\times N})\boldsymbol{\alpha} = \boldsymbol{\psi},
\end{equation}
where the entries of the matrix $\mathbf{G}$ are given by 
\begin{equation}
\label{eq: G for score estimation}
\mathbf{G}_{ij} =  \sum_{l=1}^N W^{l} \nabla_{X^l}  k(X^i, X^l) \cdot C \nabla_{X^l} k(X^l,X^j), \qquad  \qquad i,j = 1, \ldots, N,
\end{equation}
and the entries of the vector $\boldsymbol{\psi}$ are given by
\begin{equation}
\label{eq: psi for score estimation}
    \boldsymbol{\psi}_{i} = \sum_{l=1}^{N}W^{l}\Delta_{X^{l}}k(X^{i}, X^{l}), \qquad \qquad i = 1,\ldots, N.
\end{equation}
\begin{remark}[Other score estimators]
As shown by \cite{hyvarinen2005estimation,vincent2011connection}, the Fisher information between two measures $\rho$ and $\pi$ can be expressed as
\begin{equation}
\label{eq:Fisher}
 \frac{1}{2} \int_{\mathbb{R}^d} |\nabla \log \rho |^2 \, \mathrm{d}\pi + \int_{\mathbb{R}^d} \nabla \cdot \nabla \log \rho \, \mathrm{d}\pi + \text{const.},\end{equation}
 and thus the score $\nabla \log \pi$ can be approximated by (i) parameterising $\nabla \log \rho$ via a neural network, (ii) using a sample based approximation of \eqref{eq:Fisher}, and (iii) minimising the ensuing objective by gradient descent type methods. Our kernel-based estimator \eqref{eq:score estimator} can be derived from \eqref{eq:Fisher} by approximating $\nabla \log \pi(x) \approx -C\int_{\R^{d}}\nabla_{x}k(x, y) \alpha(y) \pi(\mathrm{d}y)$ and analytically computing the minimiser in $\alpha$. Different kernel-based score estimators have been proposed by
\cite{batz2016variational,li2018gradient,maoutsa2020interacting,sriperumbudur2017density,shi2018spectral}, and we leave a detailed comparison for future work.
\end{remark}
\begin{remark}[Kernel mean embeddings for score estimation] In the case when $k$ is translation invariant, i.e., $k(x, y) = g(x - y)$ for some real-valued function $g$, we have that $\Delta_x k$ is a positive definite kernel function. In this case the right hand side of \eqref{eq:IBP_ KME for score estimation _main body} is the kernel mean embedding of $\pi$ with respect to the kernel $\Delta_x k$.
\end{remark}
\subsection{Weighted KME-dynamics}\label{sec:importance sampling scheme}
Numerically, the finite sample size and the regularisation introduced in \eqref{eq:int G reg} inevitably lead to errors in computing the velocity field, resulting in at least minor discrepancies between the transported measures $\rho_{t}$ and the targets $\pi_{t}$ as defined in \eqref{eq:dist curve}. In this section, we propose an importance sampling scheme for KME-dynamics that assigns weights to the samples, attempting to correct said discrepancies. We refer the reader to 
\cite{bunch2016approximations} and \citet[Section 4.2]{heng2021gibbs} for similar ideas in the context of different reference dynamics.

To capture the mismatch between $\rho_t$ and $\pi_t$, we introduce
\begin{equation*}
  w_{t} := \frac{e^{-th}\pi_{0}}{\rho_{t}} = \frac{Z_{t}\pi_{t}}{\rho_{t}} \qquad \text{and} \qquad w_t^i:= w_t(X_t^i),  
\end{equation*}
the unnormalised weight function and corresponding individual particle weights, respectively. Expectations with respect to $\pi_t$ can then be estimated via
\begin{equation}
\label{eq:reweighting}
\int_{\mathbb{R}^d} f \, \mathrm{d} \pi_t = \int_{\mathbb{R}^d} f \frac{w_t}{Z_t} \, \mathrm{d}\rho_t \approx \sum_{i=1}^N  \frac{w_t^i f(X_t^i)}{\sum_{j=1}^N w_t^j}, 
\end{equation}
using samples $(X_t^i)_{i=1}^N$ from $\rho_t$. To shorten the notation, we denote the velocity field on the right-hand side of \eqref{eq:MF ODE} by $u_t$, that is, we assume that the particle system $(X^i_t)_{i=1}^N$ is governed by 
\begin{equation}
\label{eq:ODE weighting}
\frac{\mathrm{d}X_t^i}{\mathrm{d}t} = u_t(X_t^i), \qquad \qquad i=1,\ldots,N,  
\end{equation}
and the density 
$\rho_t$ satisfies the continuity equation
\begin{equation}
\label{eq:continuity equation for flow}
\partial_{t}\rho_{t} + \nabla \cdot (\rho_{t} u_{t}) = 0.
\end{equation}
The evolution of the particle weights can now be computed as follows:
\begin{lemma}[Flow of the unnormalised weights] \label{lem:flow of weight}
The total derivative of $\log w_{t}^{i}$ with respect to $t$ is given by
\begin{equation}
    \label{eq:flow of weight}
    \frac{\mathrm{d}\log w_{t}^{i}}{\mathrm{d}t} = \nabla \log \pi_{t}(X_{t}^{i}) \cdot u_{t}(X_{t}^{i}) + \nabla \cdot u_{t}(X_{t}^{i}) - h(X_{t}^{i}).
\end{equation}
\end{lemma}
\begin{proof}
    The proof proceeds by straightforward calculation. Details can be found in Appendix \ref{app:calc}. 
\end{proof}
Following Lemma \ref{lem:flow of weight} and using an Euler approximation, we can update the particle weights as follows,
\begin{equation}
    \label{eq:Euler scheme for logw}
    \log w_{t+\Delta t}^{i} = \log w_{t}^{i} + \Delta t \left (\nabla \log \pi_{t}(X_{t}^{i}) \cdot u_{t}(X_{t}^{i}) + \nabla \cdot u_{t}(X_{t}^{i}) - h(X_{t}^{i}) \right),
\end{equation}
or, equivalently, 
\begin{equation}
\label{eq:Sequential computation for weight flow}
    w_{t+\Delta t}(X_{t+\Delta t}^{i}) = w_{t}(X_{t}^{i}) \, \exp\left\{\Delta t \left(\nabla \log \pi_{t}(X_{t}^{i}) \cdot u_{t}(X_{t}^{i}) + \nabla \cdot u_{t}(X_{t}^{i}) - h(X_{t}^{i})\right)\right\},
\end{equation}
with initialisation $w_{0}^{i} = 1$, for all $i = 1,\ldots,N$. To complete the scheme, notice that $\nabla \cdot u_t$ can be computed in closed form thanks to the kernel ansatz in \eqref{eq:ansatz v},
    \begin{equation}
    \label{eq:divergence of v}
    \nabla_{x} \cdot u_{t}(x) = - \int_{\mathbb{R}^d} \alpha_t(y) \sum_{m,n=1}^{d}(C_t)_{mn} \left(\mathrm{Hess}_{x} k(x, y)\right)_{mn} \, \mathrm{d}\rho_t(y) + \nabla_{x}\cdot v_t^0(x),
\end{equation}
with finite sample estimator
\begin{equation}
\label{eq:estimator of divergence of v}
    \widehat{\nabla_{x} \cdot u_{t}} (x) = -\frac{1}{N} \sum_{m,n=1}^{d} \sum_{j=1}^N \alpha^j_t (C_t)_{mn} \left(\mathrm{Hess}_{x} k(x, X_t^j)\right)_{mn} + \nabla_{x}\cdot v_t^0(x),
\end{equation}
and that the score $\nabla \log \pi_t$ can be estimated as in Section \ref{sec:Score estimation} (where we use the available weights in \eqref{eq: G for score estimation} and \eqref{eq: psi for score estimation}). We summarise the method derived in this way in Algorithm \ref{alg:weighted KME dynamics}.
\begin{remark}[Estimating normalising constants]
From $Z_t = \int_{\mathbb{R}^d} \exp(-t h) \, \mathrm{d}\pi_t$, it is straightforward to see that the normalising constant satisfies
\begin{equation*}
\partial_t \log Z_t = - \int_{\mathbb{R}^d} h \, \mathrm{d}\pi_t,
\end{equation*}
leading to the estimator
\begin{equation}
\label{eq:estimator of normalising constant outside IS scheme}
{Z}_{1} = \exp\left(- \int_{0}^{1}\int_{\R^{d}} h \, \mathrm{d} \pi_{t} \, \mathrm{d} t \right) \approx \exp\left(- \frac{1}{N} \int_{0}^{1}  \sum_{i=1}^N h(X_{t}^i)\, \mathrm{d} t \right) =: \widehat{Z}^{\mathrm{plain}}_1,
\end{equation}
based on the unweighted samples $(X_t^i)_{i=1}^N$. On the other hand, \eqref{eq:reweighting} suggests using
\begin{equation}
\label{eq:estimator of normalising constant within IS scheme}
\widehat{Z}^{\mathrm{weighted}}_{1} := \frac{1}{N}\sum_{i=1}^{N}w_{1}^{i}.
\end{equation}
Both estimator are related as follows: In the case of perfect matching, $\rho_t = \pi_t$, the continuity equation \eqref{eq:continuity equation for linear path}, with $v$ replaced by $u$, implies 
\begin{equation*}
\nabla \log \pi_t \cdot u_t + \nabla \cdot u_t = h - \int h \, \mathrm{d}\pi_t, 
\end{equation*}
and so Lemma \ref{lem:flow of weight} leads to the weight update
\begin{equation}
\label{eq:flow of weight outside IS scheme}
\frac{\mathrm{d} \log w_{t}^{i}}{\mathrm{d}t} = -\int_{\R^{d}}h \,\mathrm{d}\pi_{t}.
\end{equation}
Combining \eqref{eq:flow of weight outside IS scheme} with \eqref{eq:estimator of normalising constant within IS scheme} leads back to the plain estimator \eqref{eq:estimator of normalising constant outside IS scheme}. A numerical comparison between $\widehat{Z}_1^{\mathrm{plain}}$ and $\widehat{Z}_1^{\mathrm{weighted}}$ is performed in Section \ref{sec:experiment_estimate_NC}.
\end{remark}
\begin{algorithm}
\caption{Weighted KME-dynamics}
\label{alg:weighted KME dynamics}
\begin{algorithmic}[1] 
\REQUIRE samples $(X^i_{n\Delta t})_{i=1}^N$ and the corresponding vector field $(u^i_{n\Delta t})_{i=1}^N$ for $n = 0,\ldots, N_{\mathrm{steps}}$ computed from KME-dynamics, negative log-likelihood $h:\mathbb{R}^d \rightarrow \mathbb{R}$, positive definite kernel $k: \mathbb{R}^d \times \mathbb{R}^d \rightarrow \mathbb{R}$, time-dependent matrix ${C_t \in \mathbb{R}^{d \times d}_{\mathrm{sym},>0}}$, regularisation parameter $\varepsilon > 0$.
\OUTPUT \hspace{-0.5cm} unnormalised weights for samples $(w^i_{n\Delta t})_{i=1}^N$ and the corresponding normalised weights $(W^i_{n\Delta t})_{i=1}^N$, for all $n = 0,\ldots, N_{\mathrm{steps}}$.
\STATE Initialise the unnormalised weights as $w_{0}^{i} = 1$, and the normalised weights as $W_{0}^{i} = 1/N$ for all $i=1,\ldots N$.
\FOR{$n=0,\ldots, N_{\mathrm{steps}}-1$}
\STATE Compute $\mathbf{G}_{n\Delta t}$ and $\boldsymbol{\psi}_{n\Delta t}$ using \eqref{eq: G for score estimation} and \eqref{eq: psi for score estimation}.
\STATE Solve the linear system in \eqref{eq:score estimator} for $\boldsymbol{\alpha}_{n\Delta t}$.
\FOR{$i=1,\ldots, N$}
\STATE Compute the score estimators $\widehat{\nabla \log \pi_{n\Delta t}}(X_{n\Delta t}^{i})$ using the estimator in \eqref{eq:score estimator}.
\STATE Compute the estimator of the divergence of the vector field $\widehat{\nabla_{x} \cdot u_{t}}(X_{n\Delta t}^{i})$ using \eqref{eq:estimator of divergence of v}.
\STATE $w_{(n+1)\Delta t}^{i} \leftarrow w_{n\Delta t}^{i} \, \exp\left\{\Delta t \left(\widehat{\nabla \log \pi_{n\Delta t}}(X_{n\Delta t}^{i}) \cdot u^i_{n\Delta t} + \widehat{\nabla_{x} \cdot u_{t}}(X_{n\Delta t}^{i}) - h(X_{n\Delta t}^{i})\right)\right\}$.
\ENDFOR
\STATE Normalise the weights as $W_{(n+1)\Delta t}^{i} = \frac{w_{(n+1)\Delta t}^{i}}{\sum_{i=1}^{N}w_{(n+1)\Delta t}^{i}}$.
\ENDFOR
\end{algorithmic}
\end{algorithm}
\section{Variational formulation and kernelised diffusion maps}
\label{sec:diffusion maps}
In this section, we provide further insight into the KME mean-field dynamics \eqref{eq:MF dynamics} by giving a variational principle for the vector field on the right-hand side of \eqref{eq:MF ODE}, where \eqref{eq:int MF} needs to be replaced by the alternative regularisation in \eqref{eq:new reg}. Perhaps this theoretical insight speaks in favour of \eqref{eq:new reg}, but we leave a more extensive and systematic exploration for future work. For simplicity, we will only consider the case $v^0_t \equiv 0$. To set the stage, we define the norm
\begin{equation*}
 \Vert v \Vert^2_{(L^2(\rho))^d,C} := \int_{\mathbb{R}^d} v \cdot C^{-1} v \, \mathrm{d} \rho  
\end{equation*}
on the space of vector fields $(L^2(\rho))^d$, where $C \in \mathbb{R}_{\mathrm{sym},>0}^{d \times d}$ is a fixed symmetric positive definite matrix. We also recall the \emph{maximum mean discrepancy} (see \citet[Section 3.5]{muandet2017kernel} or \citet{smola2007hilbert})
\begin{equation*}
\mathrm{MMD}_k^2 (\rho,\widetilde{\rho})  := \Vert \Phi_k(\rho) - \Phi_k(\widetilde{\rho}) \Vert^2_{\mathcal{H}_k} = \int_{\mathbb{R}^d} \int_{\mathbb{R}^d} k(x,y) (\rho(\mathrm{d}x) - \widetilde{\rho}(\mathrm{d}x)) (\rho(\mathrm{d}y) - \widetilde{\rho}(\mathrm{d}y)
)     
\end{equation*}
between two signed measures $\rho$ and $\widetilde{\rho}$, which should be interpreted as measuring the distance between $\rho$ and $\widetilde{\rho}$ through the embedding $\Phi_k$ and the RKHS-norm $\Vert \cdot \Vert_{\mathcal{H}_k}$. In the following result, we interpret the ODE $\tfrac{\mathrm{d}X_t}{\mathrm{d}t} = v_t(X_t)$ as giving rise to the signed measure $\partial_t \rho_t^v$: The change of the density describing the law of $X_t$ is induced by the vector field $v_t$, and we have $\partial_t \rho_t^ v = -\nabla \cdot(\rho_t v_t)$ according to the continuity equation.\footnote{Notice that $\rho_t^v$ and $v_t$ need to be regular enough so that $-\nabla \cdot(\rho_t^v v_t)$ can be interpreted as a signed measure.} 
\begin{proposition}
\label{prop:regression}
Assume that $k:\mathbb{R}^d \times \mathbb{R}^d \rightarrow \mathbb{R}$ is bounded and twice continuously differentiable, with bounded derivatives. Fix $\varepsilon > 0$ and $t \in (0,1)$, and assume that $\rho_t \in \mathcal{P}(\mathbb{R}^d)$ admits a continuously differentiable density with respect to the Lebesgue measure (so that $\partial_t \rho_t^v = -\nabla \cdot (\rho^v_t v)$ can be interpreted as a signed measure, for smooth $v$), and let $h \in L^2(\rho_t)$. Then the functional 
\begin{equation}
\label{eq:MMD functional}
J_{\varepsilon,t}(v) := \mathrm{MMD}_k^2(\partial_t \rho^v_t, \partial_t \pi_t) + \varepsilon \Vert v \Vert^2_{L^2(\rho_t),C_t}
\end{equation}  admits a unique minimiser $v^* \in (L^2(\rho_t))^d$ which coincides with the KME vector field with regularisation \eqref{eq:new reg}: for 
\begin{equation}
\label{eq:v star}
v^* = 
-C_t \nabla \int_{\mathbb{R}^d} \alpha_t(y) k(\cdot,y) \, \mathrm{d}\rho_t(y),
\end{equation}
with $\alpha_t$ solving \eqref{eq:new reg}, we have that $J_{\varepsilon,t}(v^*) \le J_{\varepsilon,t}(v)$, for all $v \in (L^2(\rho_t))^d$.
\end{proposition}
\begin{proof}
See Appendix \ref{app:calc}. The proof gives a regularised regression type interpretation to \eqref{eq:MMD functional} and identifies \eqref{eq:MF ODE} together with \eqref{eq:new reg} as the corresponding normal equations.
\end{proof}
Proposition \ref{prop:regression} gives precise meaning to the claim that KME-dynamics matches the curves $(\rho_t)_{t \in [0,1]}$ and $(\pi_t)_{t \in [0,1]}$ according to their kernel mean embeddings. More generally, it draws a connection between statistical estimation and regression, transport and interacting particle systems. For further directions in this recently emerging field, we would like to point the reader towards \cite{chewi2024statistical}.
\begin{remark}[Discrete-time interpretation] Assuming  perfect matching $\rho_t = \pi_t$ at fixed $t \in (0,1)$,  and approximating $\partial_t \rho_t^v \approx \tfrac{1}{\Delta t}(\rho^v_{t + \Delta t} - \rho_t)$ and $\partial_t \pi_t \approx \tfrac{1}{\Delta t}(\pi_{t + \Delta t} - \pi_t)$, the functional \eqref{eq:MMD functional} takes the form
\begin{equation*}
J_{\varepsilon,t}(v) \approx \tfrac{1}{\Delta t^2}\mathrm{MMD}_k^2(\rho^v_{t + \Delta t}, \pi_{t + \Delta t}) + \varepsilon \Vert v \Vert^2_{L^2(\rho_t),C_t}. 
\end{equation*}
Therefore, KME-dynamics \eqref{eq:MF dynamics} can be thought of as matching $(\rho_t)_{t \in [0,1]}$ to $(\pi_t)_{t \in [0,1]}$ according to maximum mean discrepancy, with an $L^2$-regularisation on the velocity. For ${C_t = I_{d \times d}}$, this regularisation is in the spirit of the Benamou-Brenier dynamical reformulation of (quadratic) optimal transport \citep{benamou2000computational}, linking to the viewpoint on KME-dynamics (or kernelised Fisher-Rao flow) put forward by \cite{maurais2023adaptive,maurais2024sampling}. It is worth pointing out that similar variational formulations are available for the Kalman-Bucy vector field \eqref{eq:Kalman baseline}, see \citet[Section 3.6]{reich2013ensemble}, and for the conditional mean embedding in the kernel Bayes' rule \citep{GrunewalderLGBPP12}.  
\end{remark}
\begin{remark}
The integral operator 
\begin{equation}
\label{eq:new reg integral op}
\alpha 
\mapsto \int_{\mathbb{R}^d} \left( G_{\rho_t,C_t} (\cdot,y) + \varepsilon k(\cdot,y)\right) \alpha_t(y)  \rho_t(\mathrm{d}y)
\end{equation}
is compact in $L^2(\rho_t)$, and therefore the equation \eqref{eq:new reg} is ill-posed \citep{engl1996regularization,kirsch2011introduction}. However, composing the (unbounded) inverse of  \eqref{eq:new reg integral op} with the integral operator in \eqref{eq:v star} overall results in a bounded operator (see the proof of Proposition \ref{prop:regression} in Appendix \ref{app:calc}), and therefore $v^*$ in Proposition \ref{prop:regression} is well defined.
\end{remark}

\textbf{The Poisson equation and infinitesimal generators of diffusions.}  In the remainder of this section, we connect the variational formulation from Proposition \ref{prop:regression} to other related work, in particular to the kernelised diffusion maps introduced by \cite{pillaud2023kernelized}. We would like to highlight the fact that kernelised diffusion maps aim at solving the same problem as conventional diffusion maps \citep{coifman2006diffusion}, but the construction is vastly different (despite the name). Conventional diffusion maps have been used to construct interacting particle systems for filtering, generative modeling and sampling, see, for example, \cite{li2023diffusion,pathiraja2021analysis,taghvaei2018kalman}.

In a nutshell, the kernelised diffusion maps perspective shifts the focus from embeddings of probability measures to embeddings of differential operators. 
We start by recalling a fundamental connection between the interpolation \eqref{eq:dist curve} and second-order elliptic partial differential equations:

To any symmetric positive definite matrix $C \in \mathbb{R}^{d \times d}_{\mathrm{\sym},>0}$ and probability measure $\pi\in \mathcal{P}(\mathbb{R}^d)$ with strictly positive and differentiable Lebesgue density, we associate the operator
\begin{equation*}
\mathcal{L}_{\pi,C}  \phi := \frac{1}{\pi} \nabla \cdot (\pi C \nabla \phi), \qquad \phi \in C_c^{\infty}(\mathbb{R}^d), 
\end{equation*}
extended to an unbounded operator acting on $L^2(\pi)$. It is straightforward to check that $-\mathcal{L}_{\pi,C}$ is positive semi-definite and essentially self-adjoint, and we furthermore assume that $\mathcal{L}_{\pi,C}$ admits a spectral gap (see \citet[Chapter 4]{pavliotis2016stochastic} for an in-depth discussion). 
\begin{remark} We may allow $C$ to be position dependent (also in \eqref{eq:MF dynamics} and \eqref{eq:IPS}), but will not pursue this direction in this paper. The operator $\mathcal{L}_{\pi,C}$ is the infinitesimal generator of the diffusion
\begin{equation}
\label{eq:aldi}
\mathrm{d}X_t =  C(X_t) \nabla \log \pi (X_t) \, \mathrm{d}t + (\nabla \cdot C)(X_t) \, \mathrm{d}t + \sqrt{2C(X_t)} \, \mathrm{d}W_t, 
\end{equation}
see \citet[Section 2.3]{pavliotis2016stochastic}, which may be interpreted as overdamped Langevin diffusion on a Riemannian manifold with metric tensor $C^{-1}$; see, for instance, \cite{livingstone2014information}. In the case when $C$ is the (empirical) covariance matrix, \eqref{eq:aldi} is connected to affine-invariant approaches to Langevin sampling \citep{garbuno2020interacting,garbuno2020affine}.
\end{remark}
The following is a well-known result: In order to match the interpolation \eqref{eq:dist curve}, it it sufficient to solve linear elliptic PDEs involving $\mathcal{L}_{\pi,C}$ and $h$.
\begin{lemma}[The Poisson equation for replicating mean-field dynamics]
\label{lem:Poisson}
Assume \\that ${\phi \in C^2([0,1] \times \mathbb{R}^d; \mathbb{R})}$ solves the Poisson equation 
\begin{equation}
\label{eq:Poisson}
\mathcal{L}_{\pi_t,C_t} \phi_t = h - \int_{\mathbb{R}^d} h \, \mathrm{d}\pi_t,
\end{equation}
for all $t \in [0,1]$, and that the mean-field ODE
\begin{equation}
\label{eq:Poisson ODE}
\frac{\mathrm{d}X_t}{\mathrm{d}t} = C_t \nabla \phi_t(X_t), \qquad X_0 \sim \pi_0     
\end{equation}
admits a unique solution. Then the law of $X_t$ reproduces the interpolation $(\pi_t)_{t \in [0,1]}$, that is, $\mathrm{Law}(X_t) = \pi_t$, for all $t \in [0,1]$. 
\end{lemma}
\begin{proof}
The proof can for instance be found in \cite{daum2010exact},\cite{reich2011dynamical} or \cite{heng2021gibbs}, and proceeds by differentiating \eqref{eq:dist curve} and using the continuity equation for \eqref{eq:Poisson ODE}. See also Remark \ref{rem:kernelised CE}.
\end{proof}
In \cite{maurais2023adaptive,maurais2024sampling}, the authors recently derived the KME-dynamics \eqref{eq:MF dynamics} (with $C_t = I_{N \times N}$ and $v_t^0 = 0$) starting from the Poisson equation \eqref{eq:Poisson} and assuming that $\phi \in \mathcal{H}_k$. In what follows, we present a related approach, based on \cite{pillaud2023kernelized}, that connects directly to the formulation in \eqref{eq:MMD functional}.

\textbf{Kernelising $\mathcal{L_{\pi,C}}$.} In keeping with the general construction principle from Section \ref{sec:construction}, and motivated by Lemma \ref{lem:Poisson}, it is plausible to ``embed'' the operator $\mathcal{L}_{\pi,C}$ into an RKHS, and solve an embedded version of the Poisson equation \eqref{eq:Poisson}. To spell this out, assume that the positive definite kernel $k$ is bounded, with bounded first and second order derivatives, $k \in C_b^2(\mathbb{R}^d \times \mathbb{R}^d; \mathbb{R})$. Under these conditions, $\mathcal{H}_k$ is continuously embedded in $L^2(\pi)$, see \citet[Theorem 4.26]{steinwart2008support}, and we denote the canonical inclusion by  $i: \mathcal{H}_k \hookrightarrow L^2(\pi)$.\footnote{This inclusion is essentially the identity, $i(f) = f$. However, note that $f \in \mathcal{H}_k$ is indeed a function, whereas $i(f) \in L^2(\pi)$ is an equivalence class (under equality up to a $\pi$-null set.)} Following the same reference, the adjoint  $i^* : L^2(\pi) \rightarrow \mathcal{H}_k$, characterised by $$\langle f, i g \rangle_{L^2(\pi)} = \langle i^*f, g \rangle_{\mathcal{H}_k}, \qquad \text{for all} \qquad  f \in L^2(\pi), g \in \mathcal{H}_k,$$ is given by the integral operator
\begin{equation*}
i^* f = \int_{\mathbb{R}^d} k(\cdot, x) f(x) \pi(\mathrm{d}x).
\end{equation*}
Following \citet{pillaud2023kernelized}, we define the \emph{kernelisation} of $\mathcal{L}_{\pi,C}$ by 
\begin{equation*}
\mathcal{L}^k_{\pi,C} := i^* \mathcal{L}_{\pi,C} i,
\end{equation*}
which, using integration by parts, can be expressed as
\begin{equation*}
\mathcal{L}^k_{\pi,C} f = - \int_{\mathbb{R}^d} \nabla_x k(\cdot, x) \cdot C \nabla f(x) \pi(\mathrm{d}x), \qquad f \in \mathcal{H}_k,
\end{equation*}
or, for $f = i^* \alpha = \int_{\mathbb{R}^d}k(\cdot, y) \alpha (y) \pi(\mathrm{d}y)$ as
\begin{equation}
\label{eq:diffusion map KME}
\mathcal{L}_{\pi,C}^k i^* \alpha = - \int_{\mathbb{R}^d} G_{\pi,C}(\cdot,y) \alpha(y) \pi(\mathrm{d}y). 
\end{equation}
The last identity, relating the kernelised diffusion operator $\mathcal{L}^k_{\pi,C}$ to the positive definite kernel $G_{\pi,C}$ defined in \eqref{eq:kernel_G}, provides the main connection between kernelised diffusion maps and KME-dynamics. 

The operator $\mathcal{L}^k_{\pi,C}$ acts on $\mathcal{H}_k$ and is self-adjoint, since $\mathcal{L}_{\pi,C}$ is self-adjoint in $(L^2(\pi),\langle \cdot,\cdot\rangle_{L^2(\pi)})$, see \citet[Section 4.6]{pavliotis2016stochastic}. The following key lemma, due to \citet[Proposition 3]{pillaud2023kernelized}, shows that the inverse of  $\mathcal{L}_{\pi,C}$ (and hence solutions to the Poisson equation \eqref{eq:Poisson}) can be obtained from the inverse of its kernelisation $\mathcal{L}^k_{\pi,C}$.  
\begin{lemma}[\cite{pillaud2023kernelized}, Proposition 3.3]
\label{lem:inversion}
Under the assumption that  ${k \in C_b^2(\mathbb{R}^d \times \mathbb{R}^d; \mathbb{R})}$ and that $\mathcal{L}_{\pi,C}$ admits an $L^2(\pi)$-spectral gap, we have that  
\begin{equation}
\label{eq:inversion formula}
\mathcal{L}_{\pi,C}^{-1} = i (\mathcal{L}^k_{\pi,C})^{-1}i^*, \end{equation}
where \eqref{eq:inversion formula} holds on the centred Sobolev space 
\[
H_0^1(\pi) = \left\{f \in L^2(\pi): \int_{\mathbb{R}^d} \Vert \nabla f\Vert^2 < \infty, \,\, \int_{\mathbb{R}^d} f \, \mathrm{d}\pi = 0 \right\}.
\]
\end{lemma}
Combining Lemma \ref{lem:Poisson} and Lemma \ref{lem:inversion}, we see that an appropriate vector field can be obtained by setting 
\begin{equation}
\label{eq:kernelised diffusion maps}
v_t = C_t \nabla i (\mathcal{L}^k_{\pi,C})^{-1}i^* \overline{h}_t,   
\end{equation}
with the shorthand notation $\overline{h}_t = h - \int_{\mathbb{R}^d} h \, \mathrm{d}\pi_t$. Indeed, $i^* \overline{h}_t$ corresponds to the right-hand side in \eqref{eq:int MF}, for $v_t^0 = 0$, and $\mathcal{L}^k_{\pi,C}$ implements the integral operator on the left-hand side.

\textbf{Regression interpretation.}
Let us introduce the scaled gradient operator $A := C \nabla$, viewed as an unbounded operator mapping (a dense subspace of) $L^2(\pi)$ into $(L^2(\pi))^d$. It is straightforward to verify that its formal adjoint with respect to $\langle \cdot ,\cdot \rangle_{(L^2(\pi))^d,C}$ is given by $A^* v = -\tfrac{1}{\pi}\nabla \cdot(\pi v)$. Indeed, 
\begin{equation*}
\langle v,A  \phi\rangle_{(L^2(\pi))^d,C} = \langle A^* v, \phi \rangle_{L^2(\pi)}, \qquad \text{for all }\phi \in C_c^{\infty}(\mathbb{R}^d), \,\, v \in (C_c^{\infty}(\mathbb{R}^d))^d.
\end{equation*}
Therefore, we can write $\mathcal{L}_{\pi,C} = -A^* A$, and, using \eqref{eq:kernelised diffusion maps}, we can express the KME vector field in an abstract form as
\begin{subequations}
\label{eq:regression}
\begin{align}
\label{eq:regression1}
v_t^* & = A i (i^*A^* A i +  \varepsilon I_{\mathcal{H}_k} )^{-1} i^* \overline{h}
\\
\label{eq:regression2}
& = Ai i^* \left( (Aii^*)^* (Ai i^*) + \varepsilon ii^* \right)^{-1} i i^* \overline{h},
\end{align}
\end{subequations}
where we have dropped the time-indices for notational efficiency, but note that the adjoints depend on time through the underlying measure $\pi_t$. We would also like to stress that in \eqref{eq:regression1}, we have $i^*: L^2(\pi) \rightarrow \mathcal{H}_k$ and $A^* : (L^2(\pi))^d \rightarrow L^2(\pi)$, that is, the adjoints are taken with respect to different spaces and inner products (implied by the domains and target spaces of $i$ and $A$). To translate \eqref{eq:regression2} back into the context from the previous sections, note that $i i^* \bar{h}$ corresponds to the right-hand side of \eqref{eq:RHS linear system} for $v_t^0 \equiv 0$, the operator $\left( (Aii^*)^* (Ai i^*) + \varepsilon ii^* \right)^{-1}$ is the solution operator for \eqref{eq:new reg}, and $Aii^*$ maps $\alpha$ to $v$ as in \eqref{eq:v star}.

Inspection of \eqref{eq:regression2} shows that $v^*$ is obtained from a normal-form expression for a regularised regression-type problem of the form 
\begin{equation*}
(Ai i^*)^* v  = i i^* \overline{h},
\end{equation*}
see \citet[Section 2.2]{kirsch2011introduction}. More specifically, $v^*$ minimises an appropriate Tikhonov functional which turns out to be equivalent to \eqref{eq:MMD functional}. For details, we refer to the proof of Proposition \ref{prop:regression} in Appendix \ref{app:calc}. The regression-type formulation in \eqref{eq:regression} holds the promise of connecting statistics and optimal transport \citep{chewi2024statistical}, and of establishing deeper relationships to other dynamical inference methods that rely on least-squares objectives: score-based generative modeling \citep{song2021score,yang2023diffusion}, flow matching \citep{lipman2023flow}, or backward stochastic differential equation approaches for optimal control \citep{chessari2023numerical,richter2023continuous} share similar principles.   

\section{Numerical experiments}
\label{sec:experiments}
In this section, we evaluate KME-dynamics from Algorithm \ref{alg:KME dynamics} through toy and filtering experiments, and demonstrate its performance as the dimension $d$ of the underlying space increases.

In our experiments, we employ two positive definite kernel functions: the quadratic kernel, as defined in \eqref{eq:quadratic_kernel}, and the Gaussian characteristic kernel \eqref{eq:Gaussian kernel}, also called radial basis function (RBF) kernel. To improve reproducibility, we generate samples for Gaussian prior distributions by creating uniformly distributed Sobol sequences (using the \emph{Sobol} Julia package\footnote{\href{https://github.com/JuliaMath/Sobol.jl}{https://github.com/JuliaMath/Sobol.jl}}), which are then transformed into normally distributed pseudo-random numbers using the Gaussian inverse c.d.f.. Furthermore, recall that the time-dependent matrix $C_t$ is chosen as the ensemble covariance matrix at each iteration, see \eqref{eq:ensemble covariance}.
\begin{figure}
\centering
\begin{subfigure}{0.32\textwidth}
    \centering
    \includegraphics[width=\textwidth]{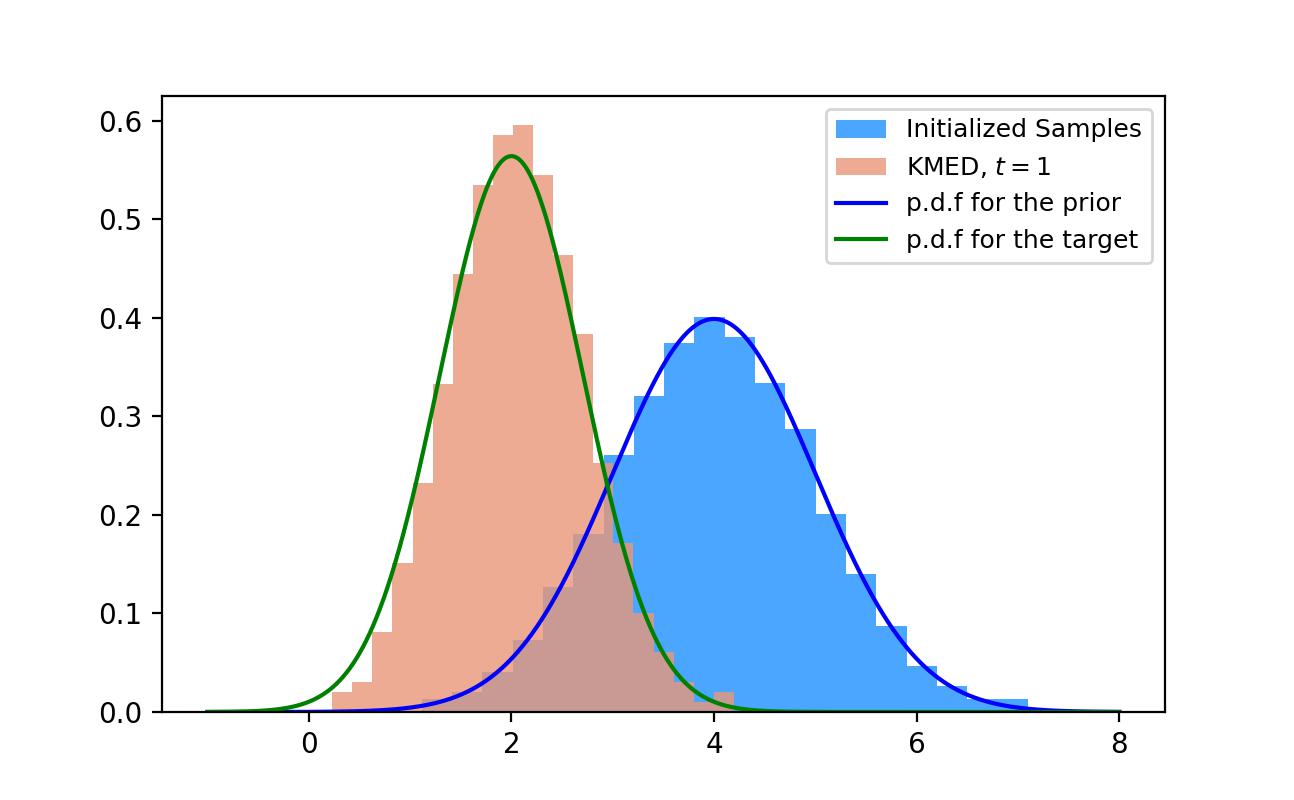}
    \caption{Gaussian to Gaussian.}
    \label{fig:toyGtoG}
\end{subfigure}
\hfill 
\begin{subfigure}{0.32\textwidth}
    \centering
    \includegraphics[width=\textwidth]{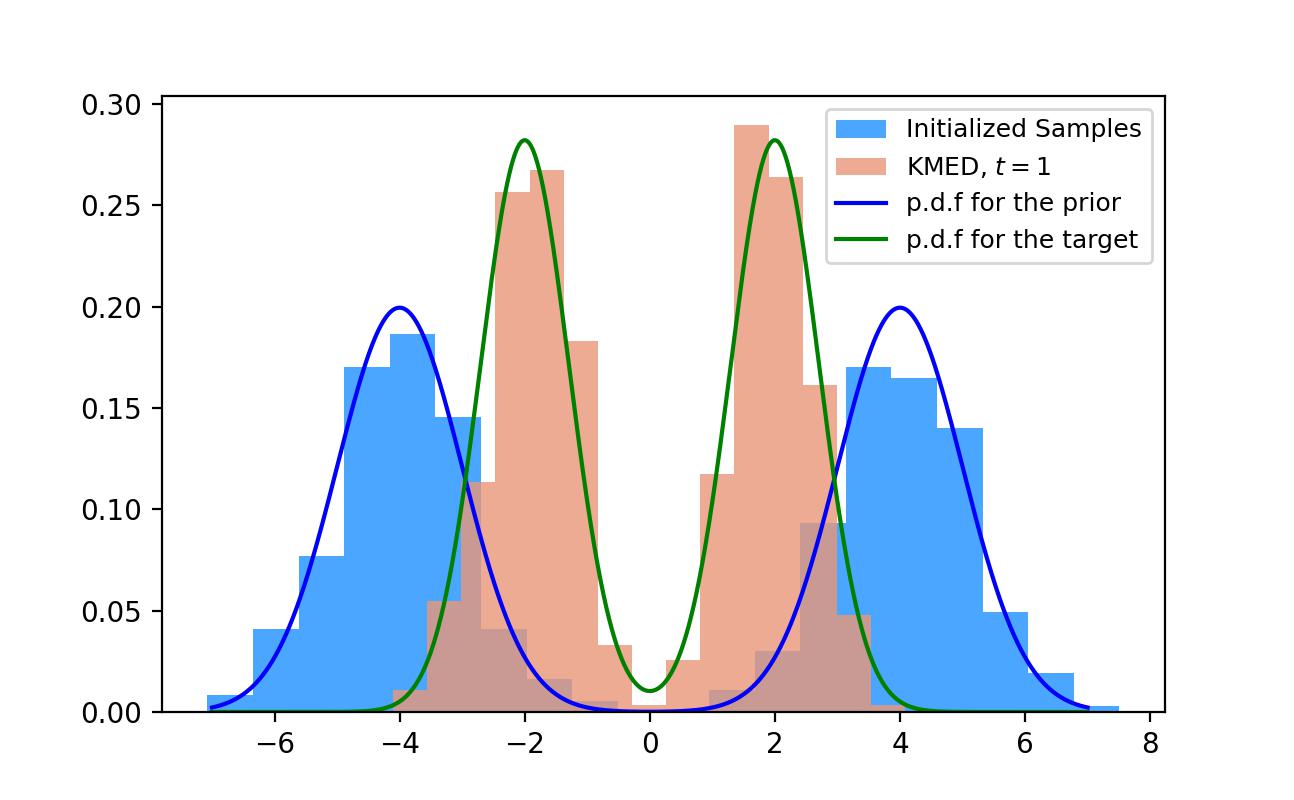}
    \caption{Mixture to Mixture.}
    \label{fig:toyGMtoGM}
\end{subfigure}
\hfill
\begin{subfigure}{0.32\textwidth}
    \centering
    \includegraphics[width=\textwidth]{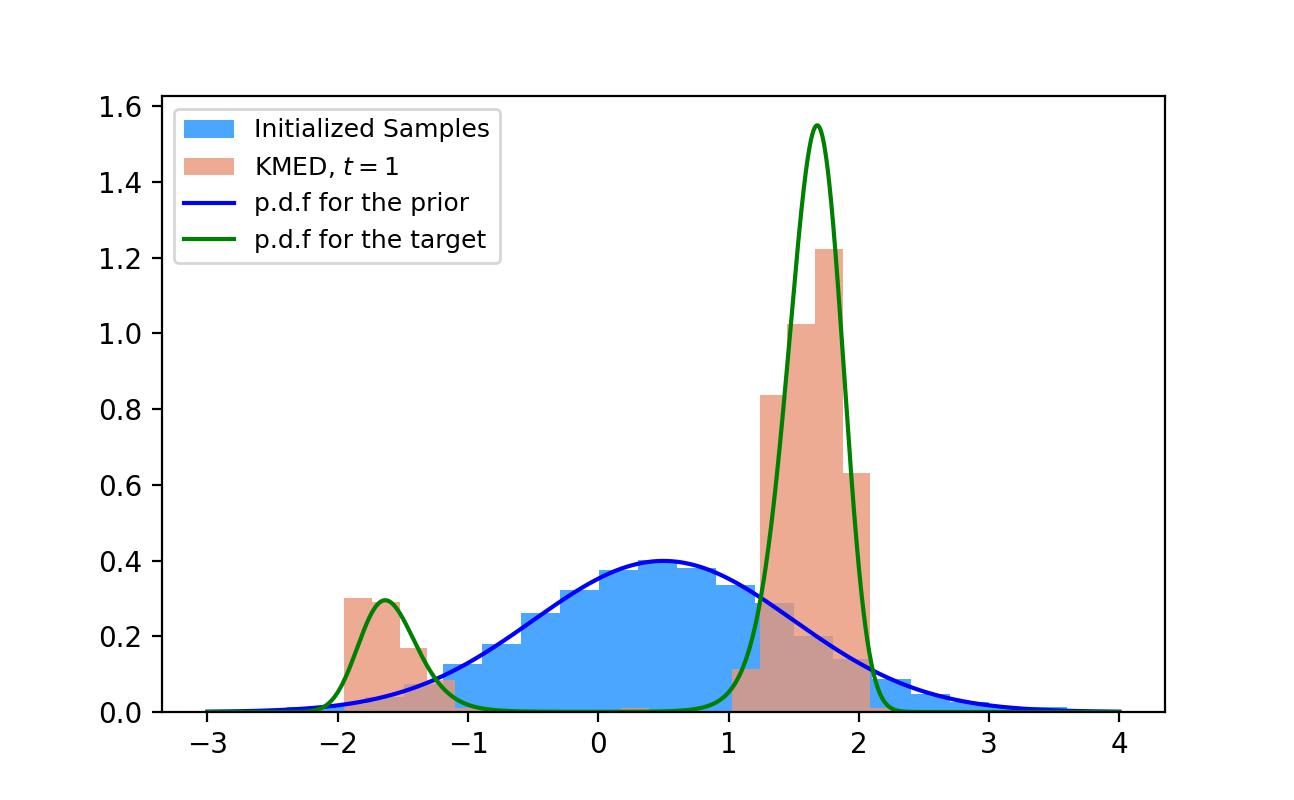}
    \caption{Gaussian to Mixture.}
    \label{fig:toypulido}
\end{subfigure}
\caption{Three Gaussian toy examples.}
\label{fig:toyexample}
\end{figure}
\subsection{Toy examples}
In this section, we test KME-dynamics with five $1$-dimensional toy examples. All of these experiments are performed with ensemble size $N = 500$, number of time steps $N_{\text{steps}} = 50$, and baseline vector field $v_{t}^{0} \equiv 0$, for all $t \in [0,1]$. Based on the non-Gaussian experiments, we explore the question of kernel choice for KME-dynamics; in particular, we contrast the performance of the quadratic Kalman-like kernel \eqref{eq:quadratic_kernel} with the characteristic RBF kernel \eqref{eq:Gaussian kernel}.

Three of the toy examples are constructed by using either Gaussian distributions or Gaussian mixture distributions for the priors and posteriors:

\textbf{Gaussian to Gaussian.} We define the negative log-likelihood to be ${h(x) = \frac{1}{2}x^2}$, and initialise samples from $\pi_{0} = \mathcal{N}(4,1)$. A straightforward computation shows that the target posterior is given by $\mathcal{N}(2,0.5)$. We apply the RBF kernel with bandwidth $\sigma = 5$, and fix the regularisation as $\varepsilon = 10^{-9}$.

\textbf{Mixture to Mixture.} We define the negative log-likelihood to be $h(x) = \frac{1}{2}x^2$, and initialise samples from $\pi_{0} = \frac{1}{2}\mathcal{N}(4,1) + \frac{1}{2}\mathcal{N}(-4,1)$. A straightforward computation shows that the target is $\frac{1}{2}\mathcal{N}(2,0.5) + \frac{1}{2}\mathcal{N}(-2,0.5)$. Again, we apply the RBF kernel with bandwidth $\sigma = 5$, and fix the regularisation as $\varepsilon = 10^{-9}$.

\textbf{Gaussian to Mixture.} We define the negative log-likelihood to be ${h(x) = (3-x^2)^2}$, and initialise samples from $\pi_{0} = \mathcal{N}(0.5,1)$. We compute the p.d.f for the target as ${\pi_{1}(x) = \frac{e^{-h(x)} \pi_{0}(x)}{Z}}$, where we compute the normalising constant ${Z = \int_{-\infty}^{\infty} e^{-h(x)} \pi_{0}(x)\, \mathrm{d} x}$ numerically using the Julia package \textit{QuadGK}.\footnote{\href{https://juliamath.github.io/QuadGK.jl/stable/}{https://juliamath.github.io/QuadGK.jl/stable/}} We apply the RBF kernel with bandwidth $\sigma = 0.95$, and fix the regularisation as $\varepsilon = 10^{-8}$.
The results of the Gaussian toy experiments are visualised in Figure \ref{fig:toyexample}, where we plot the normalised histograms for $X_{0}$ and $X_{1}$, and compare them to the exact p.d.f.'s for the priors and targets. Figure \ref{fig:toyexample} shows that KME-dynamics can produce high-quality samples, although some challenges remain in the `Gaussian to Mixture' (splitting of probability mass) example. We conjecture that including a noise term as in \eqref{eq:MF intro} might mitigate this problem, and leave further exploration for future work (see also \cite{maurais2024sampling}).
\begin{figure}
\centering
\begin{subfigure}{0.49\textwidth}
    \centering
    \includegraphics[scale = 0.4]{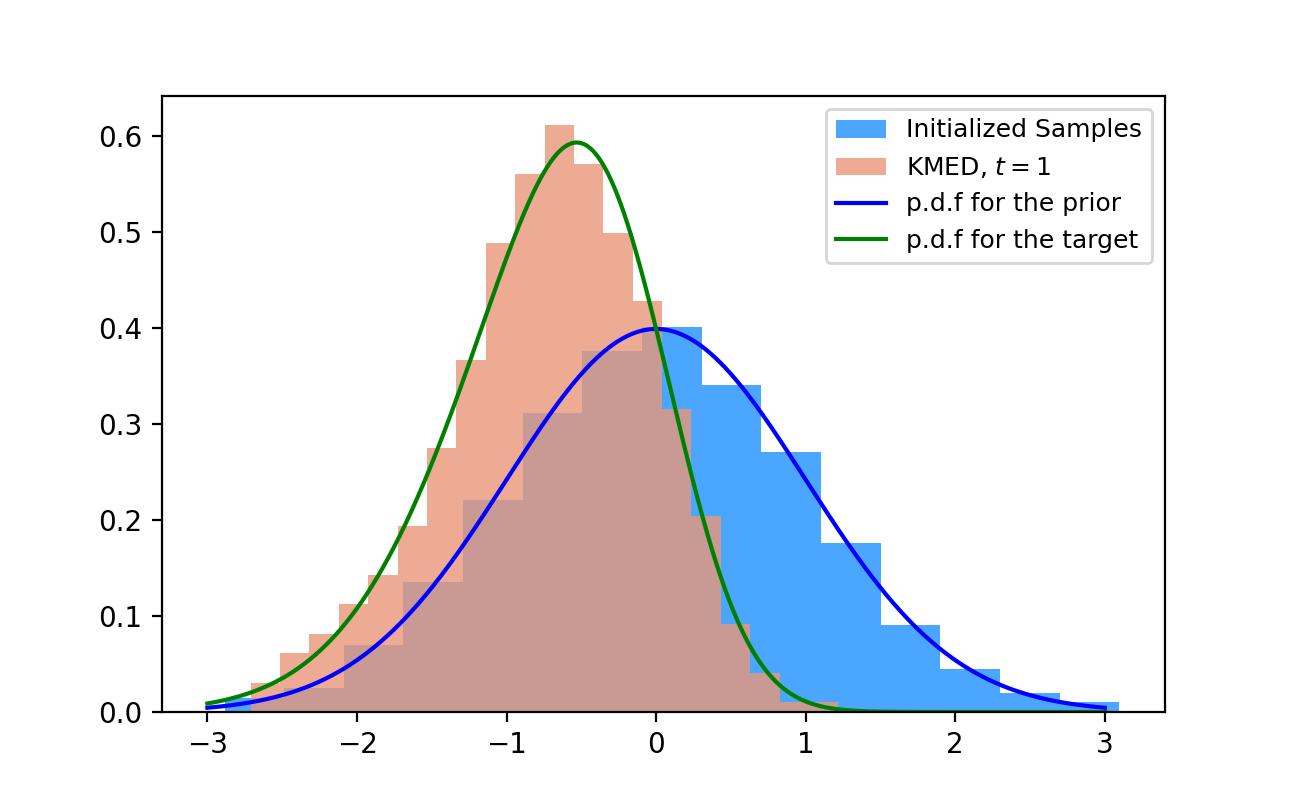}
    \caption{$d = 1$, $k_{\text{RBF}}$}
    \label{fig:skew_example}
\end{subfigure}
\hfill
\begin{subfigure}{0.49\textwidth}
    \centering
    \includegraphics[scale = 0.4]{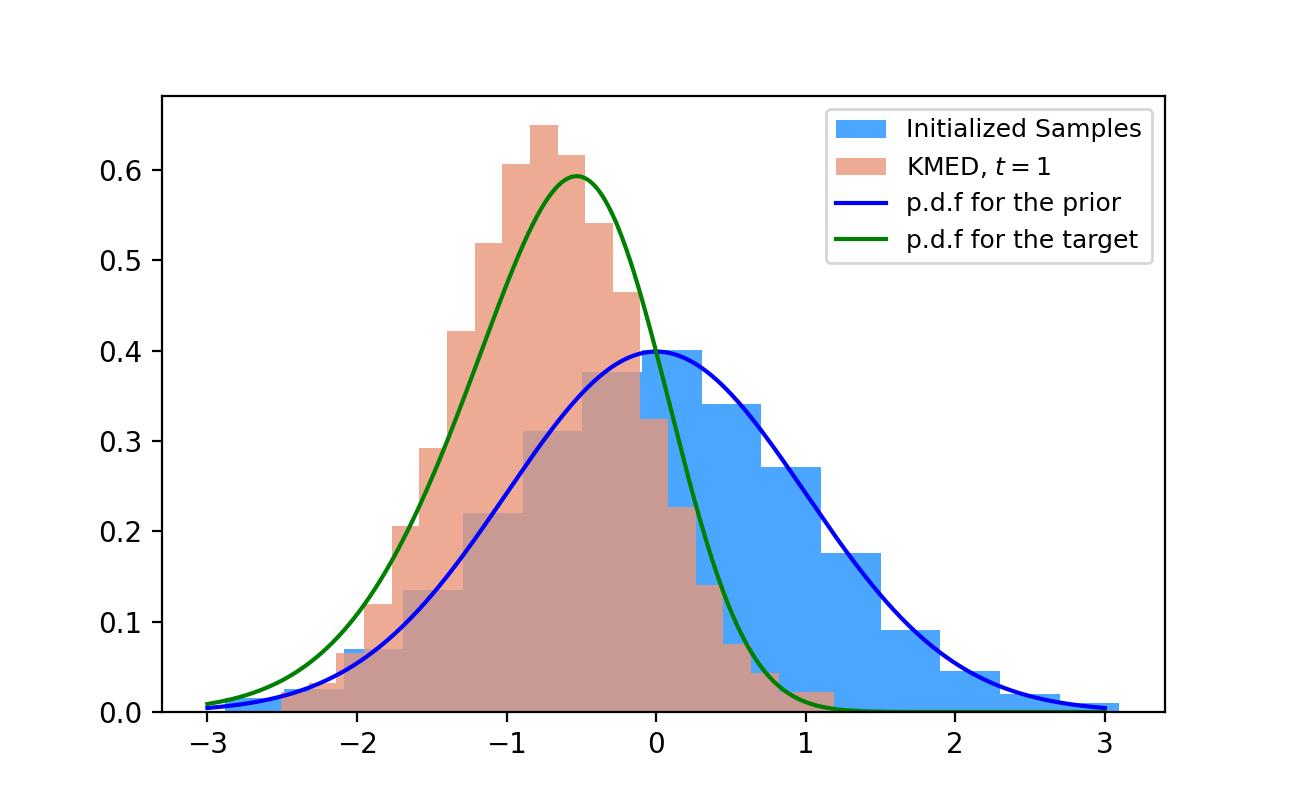}
    \caption{$d = 1$, $k_2$}
    \label{fig:skew_example_q}
\end{subfigure}
\caption{Posterior approximation and target p.d.f. for the one-dimensional skew-normal experiment.}
\label{fig:skew_Normal}
\end{figure}

Another two toy experiments are constructed with Gaussian priors and skew Normal posteriors:

\textbf{Skewing the Normal distribution and the kernel choice.} In this case we define the negative log-likelihood to be $h(x) = -\log(2\Phi(-2x))$, where $\Phi$ is the Gaussian c.d.f.. We initialise samples from $\pi_{0} = \mathcal{N}(0,1)$. Again, the p.d.f for the target is computed with the help of \textit{QuadGK}. We apply the RBF kernel with bandwidth $\sigma = 1.5$ and the quadratic kernel, respectively. The regularisation is fixed as $10^{-5}$. 

In Figure \ref{fig:skew_Normal}, we show normalised histograms of the posterior approximation, obtained by KME-dynamics with RBF-kernel (Figure \ref{fig:skew_example}) and quadratic kernel (Figure \ref{fig:skew_example_q}). Comparing with the target p.d.f., we see that the RBF kernel outperforms the quadratic kernel in this case, capturing the nonlinearity in the dynamical evolution. As suggested by Observation \ref{obs:Kalman} and Proposition \ref{pro:ReproducingKF}, the quadratic kernel provides a Gaussian approximation. 

\subsection{Estimating normalising constants}\label{sec:experiment_estimate_NC}
In this section, we further compare KME-dynamics and its extensions in the context of estimating the normalising constant for Gaussian posteriors.
All experiments are performed on the $d$-dimensional `Gaussian to Gaussian' toy example. Specifically, the prior is $\pi_{0} = \mathcal{N}(\mu_{0},\Sigma_{0})$, with mean  $\mu_{0} = \left(1, 1, \ldots, 1 \right) \in \mathbb{R}^d$ and $\Sigma_{0} = I_{d\times d}$ is the $d\times d$ identity matrix, and the negative log-likelihood is $h(x) = \tfrac{1}{2}x^{\top} x$.

\textbf{Benchmarks with normalising constants.} We will evaluate the performance of the algorithms considering various features: the dimensionality ($d$), the sample size ($N$), the choice of kernel function (either RBF kernel with different bandwidths or quadratic kernel), and the subdivision of the time interval ($N_{\mathrm{steps}}$). 

The efficacy of the algorithms will be assessed by comparing the estimated normalising constant $\widehat{Z}_{1}$ with the true value, calculated as $Z_{1} = \exp\left( \tfrac{d}{2}\log \tfrac{1}{2} - \tfrac{d}{4}\right)$. When the samples are weighted, the estimator is given by \eqref{eq:estimator of normalising constant within IS scheme}. When the samples are not weighted, the estimator is given by \eqref{eq:estimator of normalising constant outside IS scheme}, and it can be further approximated using Euler's scheme:
\begin{equation}
    \label{eq:estimator of normalising constant for unweighted samples}
    \widehat{Z}_{1} \approx \exp\left(- \tfrac{\Delta t}{N} \sum_{j = 0}^{N_{\mathrm{steps}} - 1}  \sum_{i = 1}^{N} h(X_{j\Delta t}^{i}) \right).
\end{equation}
\textbf{Setting up the algorithms.} We will evaluate five algorithms: plain KME-dynamics without modifications (KMED), Kalman-adjusted KME-dynamics (KAKMED), weighted KME-dynamics (WKMED), weighted Kalman-adjusted KME-dynamics (WKAKMED), and Sequential Importance Resampling (SIR). Note that for all KME-dynamics related algorithms, we employ the RBF kernel by default unless otherwise specified.

For KMED and KAKMED, it is assumed that the samples at each time $t \in [0,1]$ are equally weighted. However, the baseline vector fields differ: $v_{t}^0 = 0$ for KMED and $v_{t}^0 = v_{\mathrm{Kalman}}$ for KAKMED, where $v_{\mathrm{Kalman}}$ is defined in \eqref{eq:Kalman baseline}. WKMED and WKAKMED extend KMED and KAKMED using the importance sampling scheme from Section \ref{sec:importance sampling scheme}, where we compute weights for samples at each $t \in [0,1]$ to correct the corresponding measures. 

For SIR, at each $t \in [0,1]$, sampling from $X_{t+\Delta t}$ is performed by resampling from $X_{t}$ with respect to the unnormalised weight $w_{t}(X_{t}):= e^{-\Delta t h(X_{t})} \propto \frac{\pi_{t+\Delta t}(X_{t})}{\pi_{t}(X_{t})}$. With samples $(X_{t}^{i})_{i=1}^{N}$, we compute the unnormalised weight $w_{t}^{i} := w_{t}(X_{t}^{i}) = e^{-\Delta t h(X_{t}^{i})}$ for all $i = 1,\ldots, N$, then normalise the weights as $W_{t}^{i}:= \frac{w_{t}^{i}}{\sum_{i=1}^{N}w_{t}^{i}}$, and produce unbiased samples $(X_{t+\Delta t}^{i})_{i=1}^{N}$ such that $\frac{1}{N}\sum_{i=1}^{N}\delta_{X_{t+\Delta t}^{i}} = \sum_{i=1}^{N}W_{t}^{i}\delta_{X_{t}^{i}}$ via multinomial resampling, see, for example, \cite{chopin2020introduction,doucet2009tutorial}.

Below we provide more details about the parameter setup:

\textbf{Dimensionality.} In this experiment, we fix the sample size $N = 100$ and the number of time steps $N_{\mathrm{steps}} = 50$, while varying the dimensionality $d = 1,2,\ldots, 55$. We conduct comparisons among KMED, WKMED, and SIR. Specifically, we set the bandwidth ${\sigma = 5.5}$ for both KMED and WKMED. For regularisation parameters, we take $\varepsilon = 10^{-5}$ for KMED and $\varepsilon = 10^{-3}$ for WKMED. The increased regularisation for WKMED addresses its instability in higher dimensions. Additionally, we explore a variant of WKMED, where the choice of $C_t$ from \eqref{eq:ensemble covariance} is replaced by $C_t = I_{d \times d}$, to demonstrate the benefits of using the ensemble covariance as $C_t$. For SIR, due to its high instability, we conduct the experiment $100$ times for each configuration $(N, N_{\mathrm{steps}}, d)$ and report means and variances of the results.

\textbf{Sample size.} In this experiment, we set the dimensionality $d = 3$ and the number of time steps $N_{\text{steps}} = 50$, while varying the sample size $N \in [2,1000]$. We evaluate the performance of three algorithms: KMED, WKMED, and SIR. For KMED and WKMED, we select a kernel bandwidth $\sigma = 4.0$ and a regularisation parameter $\varepsilon = 10^{-5}$. The experimental procedure for SIR mirrors the approach taken in the dimensionality experiment.

\textbf{Kernel choice.} In this experiment, we fix the dimensionality $d = 3$, the sample size $N = 100$, and the number of time steps $N_{\text{steps}} = 50$. We assess the performance of KMED employing an RBF kernel against the bandwidth parameter $\sigma$ within the range $[0.5, 4.0]$. Additionally, we also check the performance of KMED with the quadratic kernel.

\textbf{Time discretisation.} In this experiment, we fix the dimensionality $d = 3$ and the sample size $N = 200$, while varying the number of time steps $N_{\mathrm{steps}}$ from 5 to 150. We conduct comparisons among four algorithms: KMED, KAKMED, WKMED, and WKAKMED. For all these algorithms, we set the kernel bandwidth $\sigma = 4.0$ and the regularisation parameter $\varepsilon = 10^{-5}$.

\begin{figure}[h]
\centering
\begin{subfigure}{0.4\textwidth}
    \centering
    \includegraphics[scale = 0.5]{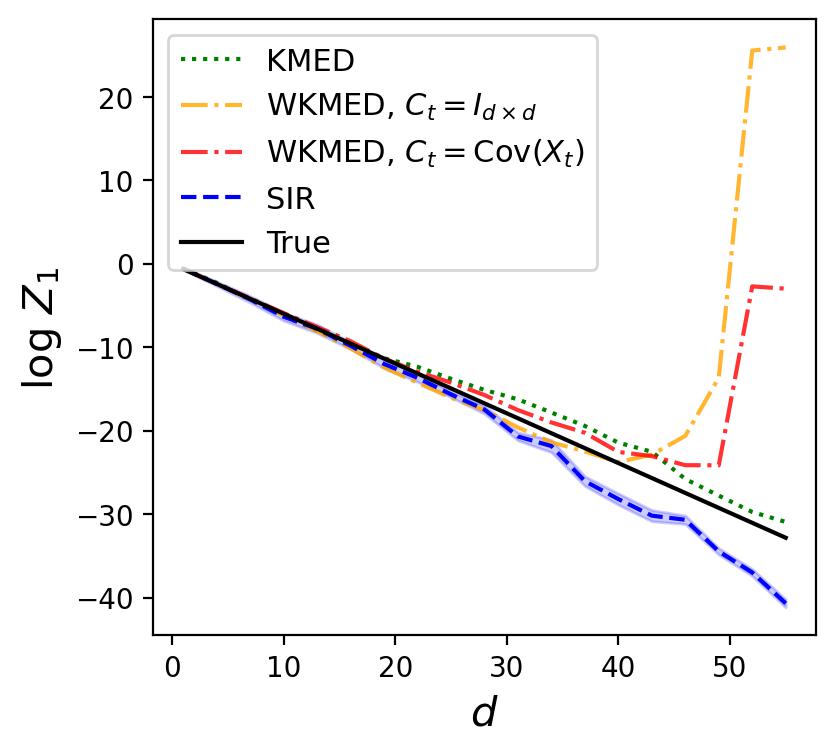}
    \caption{Benchmark on dimension $d$.}
    \label{fig:Bench_d}
\end{subfigure}
\hfill 
\begin{subfigure}{0.58\textwidth}
    \centering
    \includegraphics[scale = 0.5]{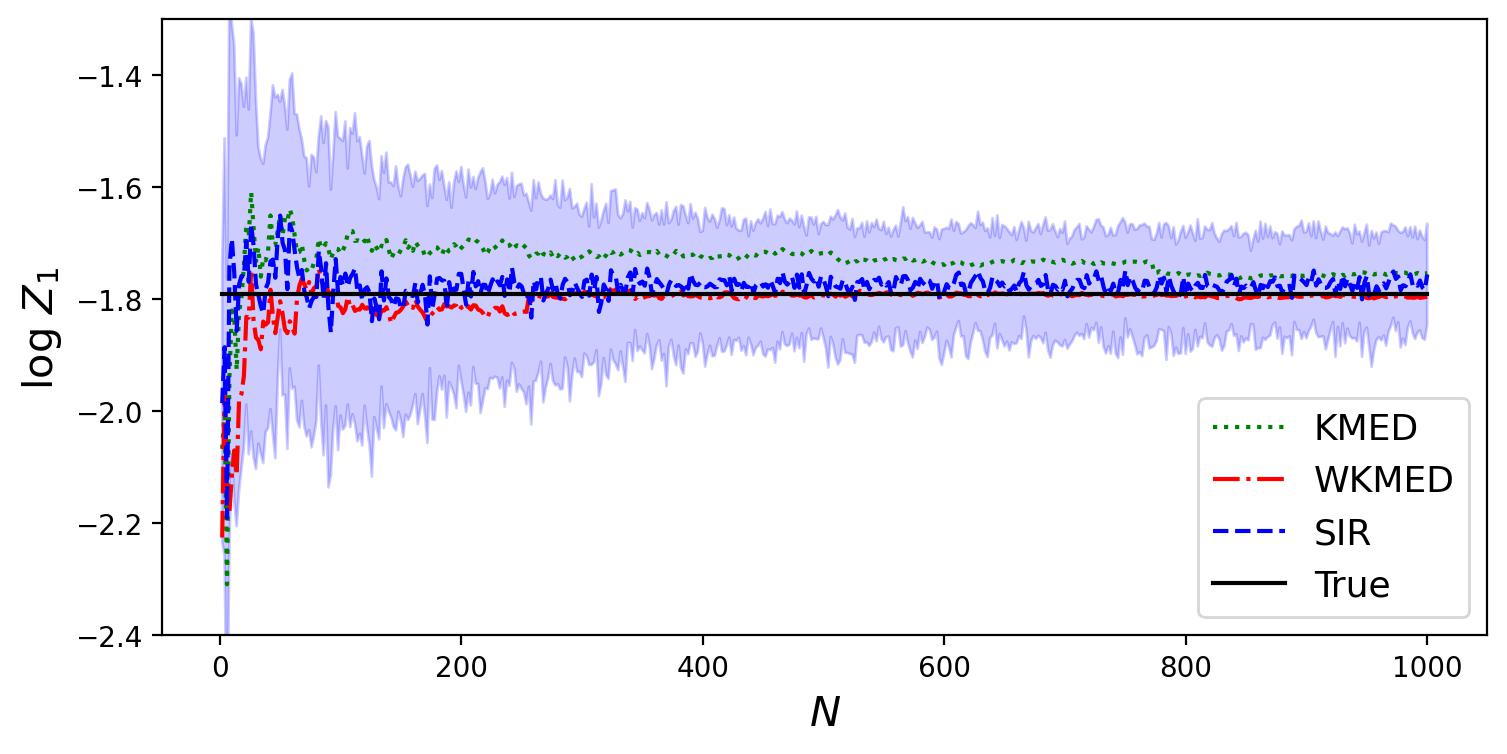}
    \caption{Benchmark on sample size $N$.}
    \label{fig:Bench_N}
\end{subfigure}
\hfill
\begin{subfigure}{0.495\textwidth}
    \centering
    \includegraphics[scale = 0.5]{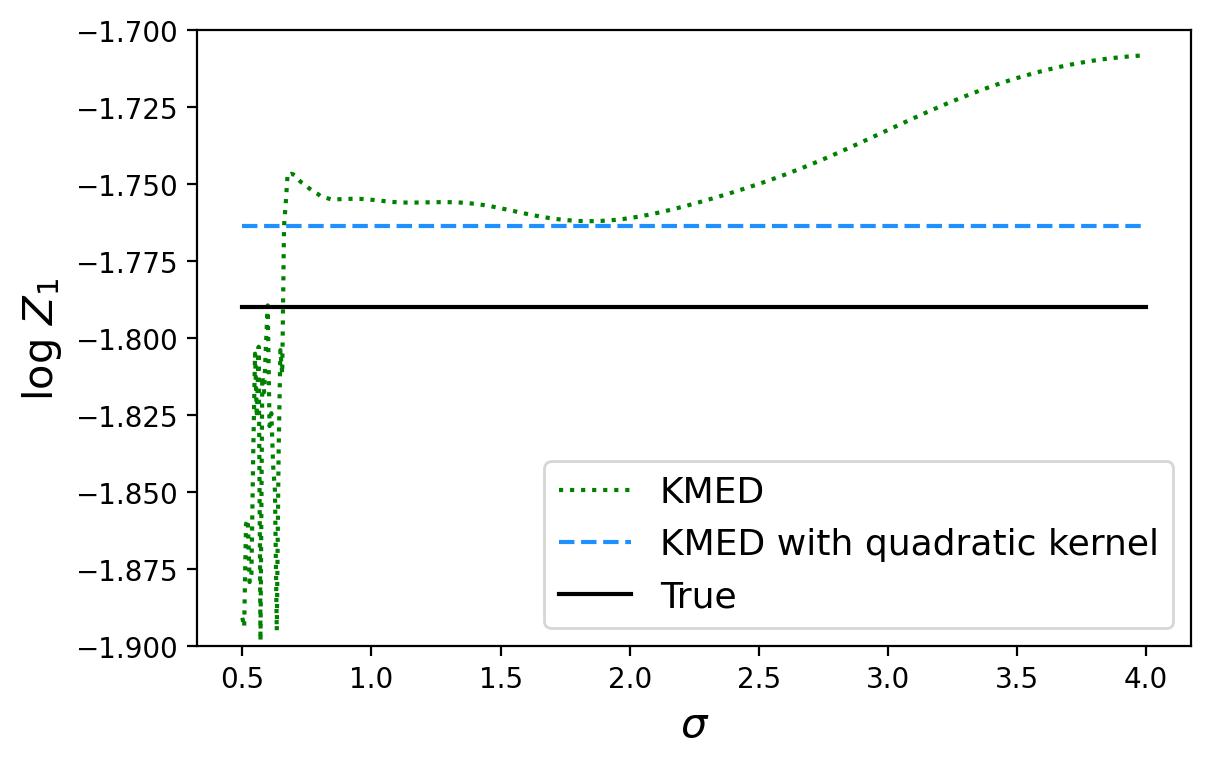}
    \caption{Benchmark on bandwidth $\sigma$.}
    \label{fig:Bench_sigma}
\end{subfigure}
\begin{subfigure}{0.495\textwidth}
    \centering
    \includegraphics[scale = 0.5]{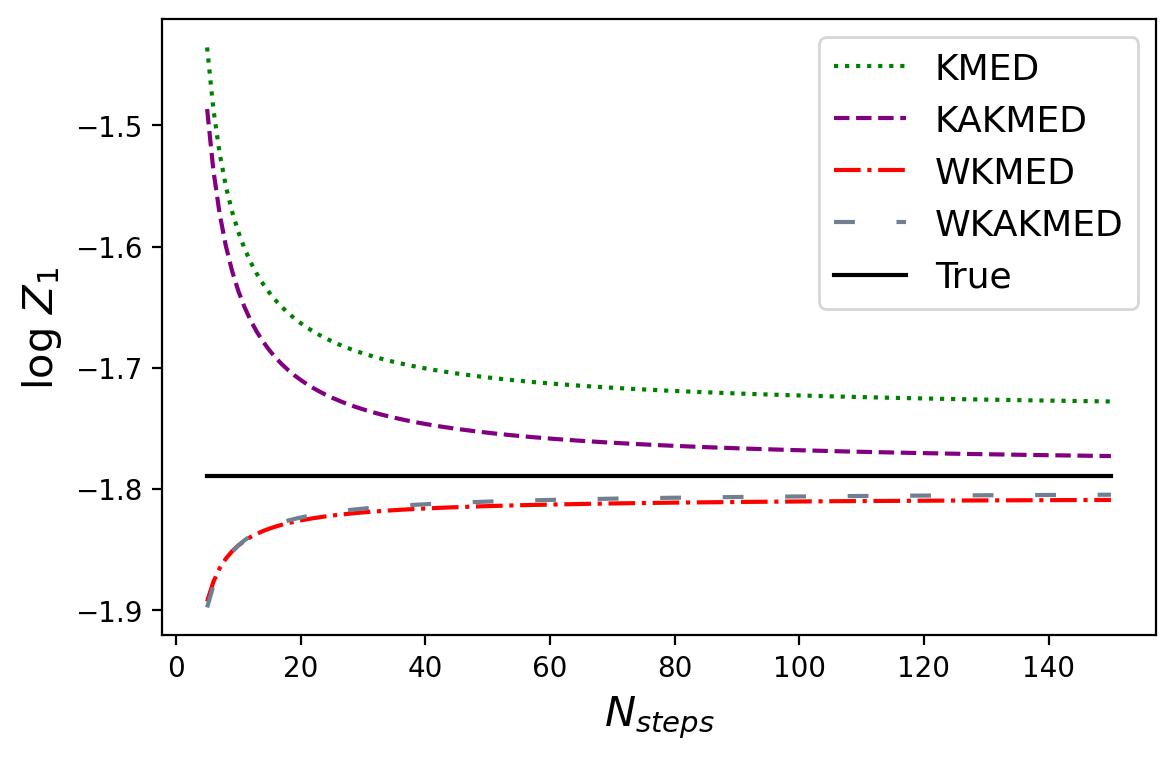}
    \caption{Benchmark on time discretisation.}
    \label{fig:Bench_time}
\end{subfigure}
\caption{We perform different sampling methods on the `Gaussian prior to Gaussian posterior' toy experiment and evaluate their performances by comparing the estimated normalising constant with the true value. The performance is evaluated against four parameters: dimensionality (top left), sample size (top right), bandwidth of the RBF kernel (bottom left), and the number of iterations for KME-dynamics (bottom right).}
\label{fig:Benchmarks}
\vspace{-0.5cm}
\end{figure}
The benchmark results are shown in Figure \ref{fig:Benchmarks}. As expected, the performance of KMED deteriorates with the increasing dimensionality of the space, but it improves with larger sample sizes and more subdivisions of the time interval $[0,1]$, as shown in Figures \ref{fig:Bench_d}, \ref{fig:Bench_N}, and \ref{fig:Bench_time}. Furthermore, Figures \ref{fig:Bench_d} and \ref{fig:Bench_N} demonstrate the superior stability and performance of KMED even in high-dimensional settings when the sample size is small. As expected, SIR (as an unbiased scheme with possibly high variance) remains the preferable option when the sample size is large and the prior and posterior intersect substantially. 

Remarkably, the enhanced performance of KAKMED, as shown in Figure \ref{fig:Bench_time}, suggests significant improvements in KMED when partial prior knowledge about the inference is available. WKMED has been numerically proven to be a robust technique for correcting the transported measure, as indicated in Figures \ref{fig:Bench_N} and \ref{fig:Bench_time}, particularly when the sample size is large. Nonetheless, WKMED suffers more from the curse of high dimensionality than KMED, as shown in Figure \ref{fig:Bench_d}. This is because the weight flow estimation, as described by \eqref{eq:Sequential computation for weight flow}, is highly sensitive to errors in score estimation (indeed, involving a  multiplication by $e^{\mathrm{error}_{\mathrm{score}}}$), where the score estimation was conducted using a similar procedure as KMED (as discussed in Section \ref{sec:Score estimation}). Due to this error-sensitivity, the choice of $C_t$ becomes crucial for KME-dynamics within the importance sampling scheme, as illustrated in Figure \ref{fig:Bench_d}. Indeed, while the improvement in KME-dynamics and score estimation by selecting $C_t$ to be as specified in \eqref{eq:ensemble covariance} may appear minimal, it significantly impacts the accuracy of weight approximation. 

\textbf{Choice of bandwidth; overfitting and underfitting.} From the perspective of the embedding \eqref{eq:abstract embedding}, it is reasonable to expect that $\mathcal{H}$ should neither be too small nor too large for optimal performance of KMED: Indeed, if $\mathcal{H}$ is too small, then the embedding $\Phi$ is not very expressive, and the accuracy in reproducing $(\pi_t)_{t \in [0,1]}$ is low (overfitting). If $\mathcal{H}$ is excessively large, we might expect fragile behaviour when the size of the ensemble is moderate (underfitting). Figure \ref{fig:Bench_sigma} shows the performance of KMED with RBF kernel against the bandwidth of the RBF kernel (we omit the part of the curve where the algorithm is unstable), recalling that the size of the corresponding RKHS decreases with increasing bandwidth \citep[Proposition 4.46]{steinwart2008support}. As expected, KMED is optimised by a specific choice of intermediate bandwidth, to prevent from overfitting or underfitting. In accordance with Section \ref{sec:quadratic Kalman}, KMED with the quadratic kernel is particularly effective in scenarios where both the prior and the posterior distributions are Gaussian.

\subsection{Data assimilation in Lorenz systems}
\label{sec:Lorenz}
In this section we compare KME-dynamics with the Ensemble Kalman Filter (EnKF) in the context of the nonlinear filtering problem for the Lorenz-63 and Lorenz-96 systems, which are studied in detail, for instance, by \citep{reich2015probabilistic}. 
Similar experiments were conducted by  \citet{stordal2021p}, who compared the EnKF to the mapping particle filter (MPF) by \citet{pulido2019sequential}, another kernel-based method that rests on Stein variational gradient descent (SVGD) introduced by \citet{liu2016stein}. Notice that SVGD ordinarily requires the scores of the target distributions; however, in the context of data assimilation, only a Monte Carlo estimator of the score function is typically available. The MPF algorithm essentially applies the SVGD method, but it uses the estimated score functions as input.

We first explain the basic framework of this experiment. The filtering is conducted using the state space model
\begin{equation}
\label{eq:HMM}
X_{n} = \mathcal{M}(X_{n-1}) + \eta_{n}, \qquad Y_{n} = X_{n} + \xi_{n},
\end{equation}
where $(X_n)_{n \in\mathbb{N}}$ represent a hidden signal, and $(Y_n)_{n \in \mathbb{N}}$ available observations. 
In \eqref{eq:HMM}, $\mathcal{M}$ represents the dynamics derived from the 4th-order Runge-Kutta method applied to either the Lorenz-63 or the Lorenz-96 system. The terms $\eta$ and $\xi$ denote i.i.d. Gaussian noise of appropriate dimensionality, with covariance matrices $Q$ and $R$, respectively. More specifically, the evolution of the Lorenz-63 system is described by the following set of ordinary differential equations (ODEs):
\begin{align}
\label{eq:Lorenz63}
    \frac{\mathrm{d}x_t}{\mathrm{d}t} = 10(y_t-x_t), \qquad
    \frac{\mathrm{d}y_t}{\mathrm{d}t} = x_t(28 - z_t) - y_t, \qquad
    \frac{\mathrm{d}z_t}{\mathrm{d}t} = x_t y_t - \frac{8}{3}z_t,
\end{align}
where $(x_t, y_t, z_t)$ represents a curve in $\mathbb{R}^3$. The evolution of the Lorenz-96 system generalises the trajectory to $(x^{1}_{t}, x^{2}_{t}, \ldots, x^{d}_{t}) \in \mathbb{R}^d$ and is defined by the following system of ODEs:
\begin{equation}
\label{eq:lorentz96}
\frac{\mathrm{d}x^{m}_{t}}{\mathrm{d}t} = (x^{m+1}_{t} - x^{m-2}_{t}) x^{m-1}_{t} - x^{m}_{t} + F,
\end{equation}
for $m = 1, 2, \ldots, d$. Here, $F$  is a constant real number.

\textbf{Creating observations.} After initialising the position $(x^{1}_{0}, x^{2}_{0}, \ldots, x^{d}_{0})$, we integrate \eqref{eq:Lorenz63} or \eqref{eq:lorentz96} using the the 4th-order Runge-Kutta method with time step $\delta t$. We set the observation time window to $\Delta t > \delta t$, and collect noisy observations
    \begin{equation}
    \label{eq:observation model}
    \beta_j = (x^{1}_{t_{j}}, x^{2}_{t_{j}}, \ldots, x^{d}_{t_{j}}) + \xi_j,
    \end{equation}
    at assimilation times $t_j = j \Delta t$, for $j=1,\ldots, N_{\text{assi}}$, in accordance with the state space model \eqref{eq:HMM}. 
    Recall that in \eqref{eq:observation model}, the perturbations are i.i.d. normally distributed, $\xi_j \sim \mathcal{N}(0,R)$.

\textbf{Filtering.} Data assimilation combines two sampling tasks, performed alternately and sequentially: the forecast step ${p(X_{n-1} \mid Y_{n-1}) \mapsto p(X_{n} \mid X_{n-1})}$, which can be simulated directly from the hidden model in \eqref{eq:HMM}; and the inference step $p(X_{n} \mid X_{n-1}) \mapsto p(X_{n} \mid Y_{n})$. We compare four methods for the inference step: Firstly, the Ensemble Kalman Filter (EnKF) is implemented as in \citet[Algorithm 7.7]{reich2015probabilistic}. Secondly, we use KME-dynamics with ${N_{\text{steps}} = 50}$ and baseline vector field $v^{0}_{t} \equiv 0$, $\forall t\in [0,1]$. Thirdly, we test Kalman Adjusted KME-dynamics (KA-KMED) by taking KME-dynamics with Kalman-Bucy velocity \eqref{eq:Kalman baseline} as the baseline $v_t^0$. We use the RBF kernel \eqref{eq:Gaussian kernel} for both KMED and KA-KMED. Finally, MPF is implemented as in \cite{pulido2019sequential}, with coordinate-wise step size adjustment based on ADAGRAD, as suggested by \cite{liu2016stein}.

The experimental details are as follows:

\textbf{Lorenz-63.} We set $\delta t = 0.01$, $\Delta t = 0.1$, and $N_{\text{assi}} = 100$. The system is initialised from $(-0.587, -0.563, 16.870)$, and the covariance of the observation error is ${R = 0.7 I_{3 \times 3}}$. The prior at $t=0$ is chosen as $\mathcal{N}((-0.587, -0.563, 16.870), 0.01 I_{3 \times 3})$. We consider ensemble sizes $N = 2, 4, 6, \ldots, 30$ and model errors $Q = qI_{3 \times 3}$, with $q$ taking the values $1.4 \Delta t$, $\frac{1.4 \Delta t}{10}$, and $\frac{1.4 \Delta t}{1000}$. We note that the filtering problem tends to exhibit more Gaussian characteristics with larger model error (as the chaotic deterministic dynamics becomes less prevalent), and stronger nonlinearity and non-Gaussianity with smaller model error (see the discussion by \citet[Appendix A.1]{stordal2021p}). For both KMED and KA-KMED, we set the bandwidth of the RBF kernel to $\sigma = 5$.
\begin{figure}
\centering
\begin{subfigure}{0.32\textwidth}
    \centering
    \includegraphics[scale = 0.47]{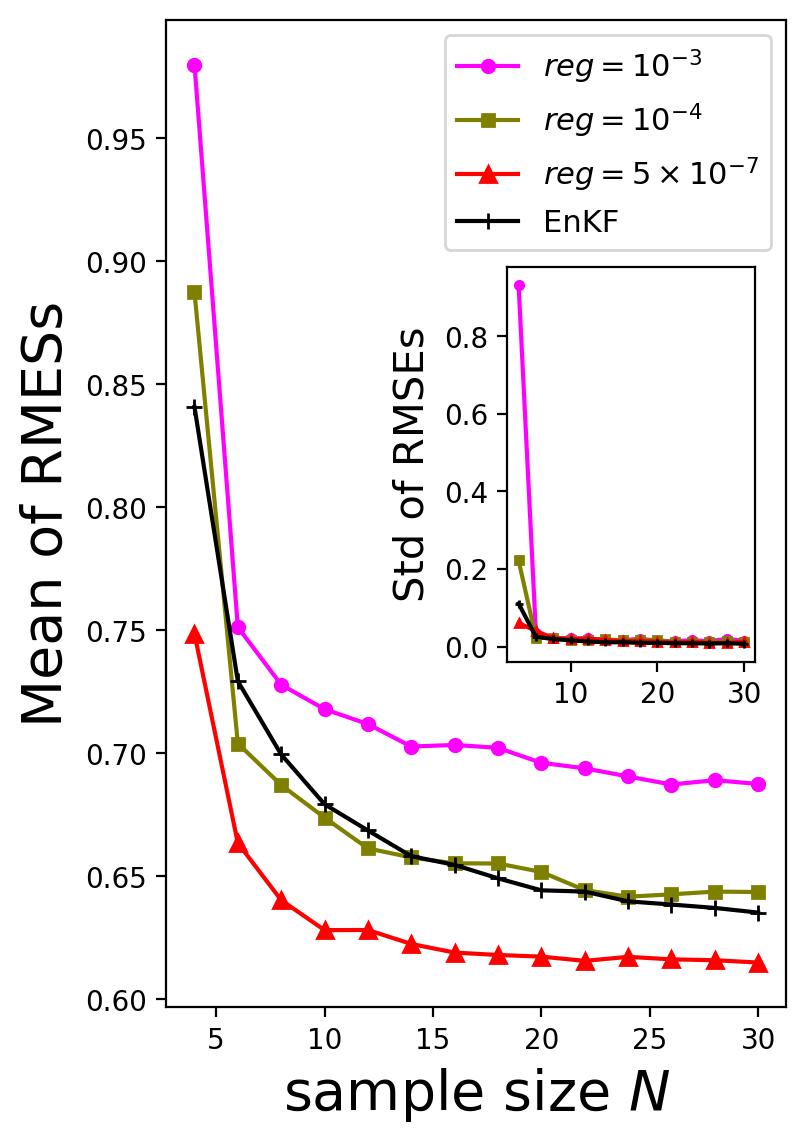}
    \caption{$q = 1.4 \Delta t$}
    \label{fig:filter_KMED_XL}
\end{subfigure}
\hfill
\begin{subfigure}{0.32\textwidth}
    \centering    \includegraphics[scale = 0.47]{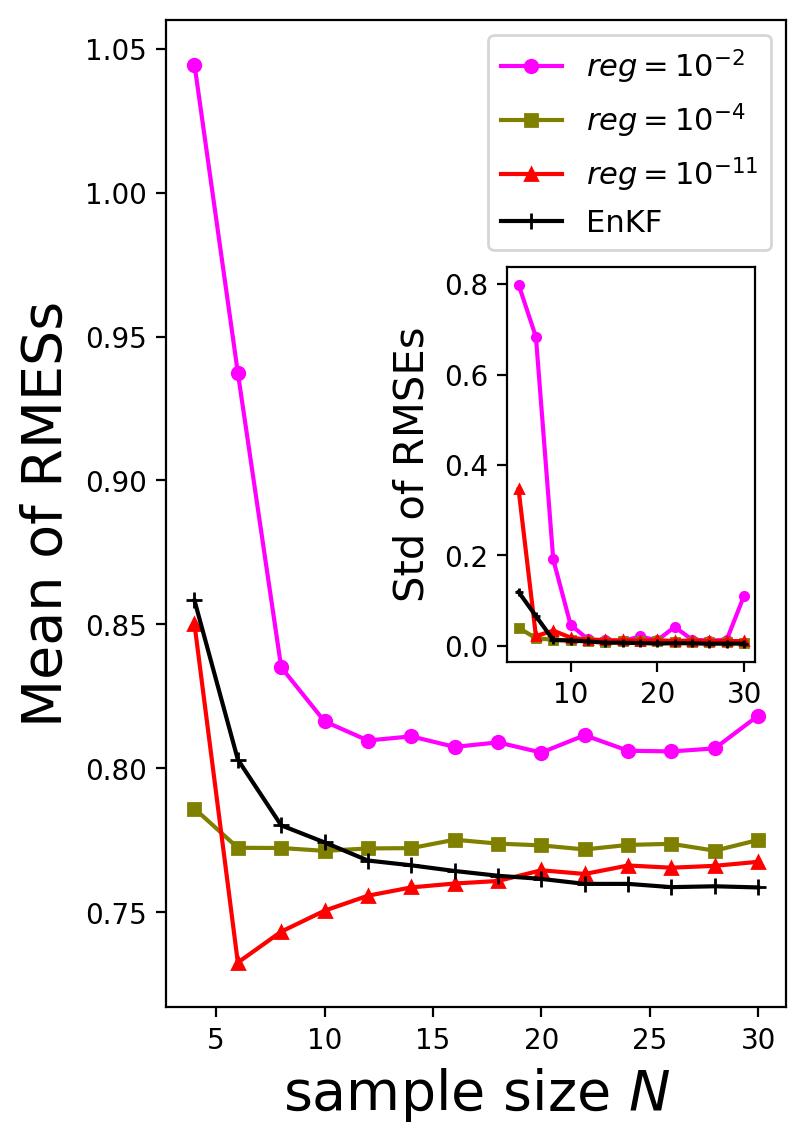}
        \caption{$q = \frac{1.4 \Delta t}{10}$}
    \label{fig:filter_KMED_L}
\end{subfigure}
\hfill
\begin{subfigure}{0.32\textwidth}
    \centering
    \includegraphics[scale = 0.47]{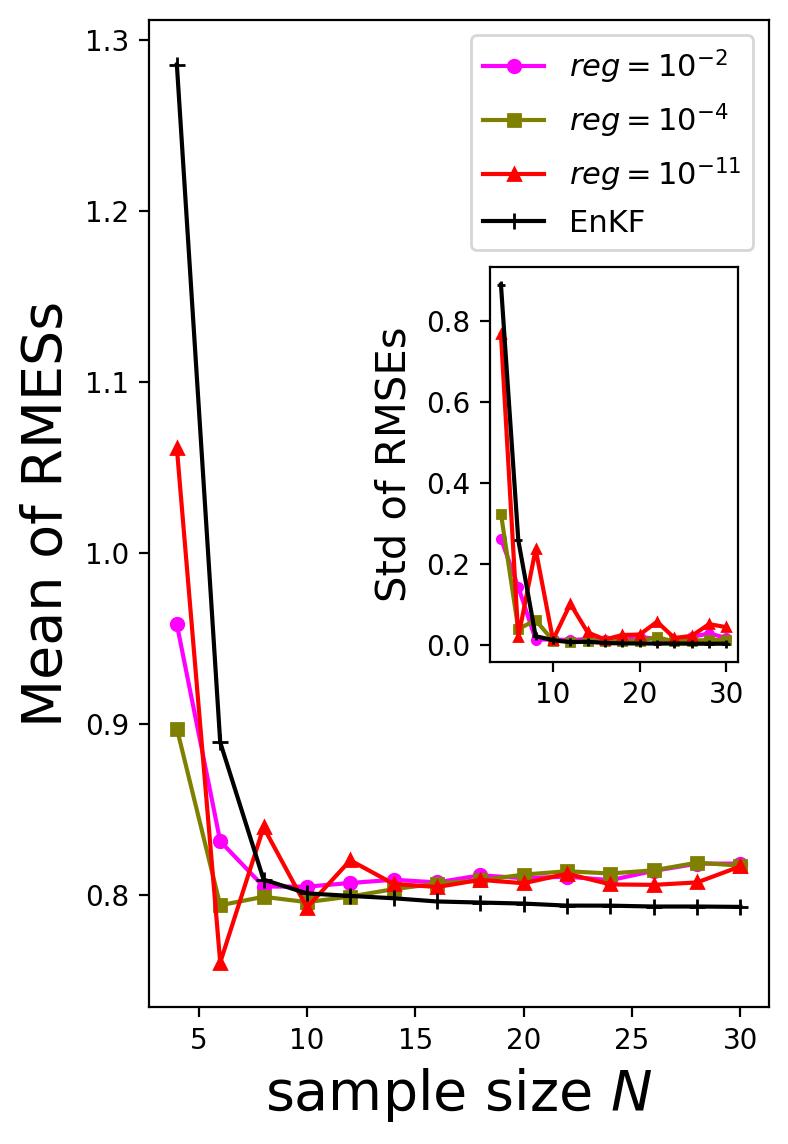}
    \caption{$q = \frac{1.4 \Delta t}{1000}$.}
    \label{fig:filter_KMED_XS}
\end{subfigure}
\caption{EnKF vs KMED for different regularisations, for large model error (left), small model error (middle),  and tiny model error (right).}
\label{fig:KMED_benchmark}
\end{figure}
\begin{figure}
\centering
\begin{subfigure}{0.32\textwidth}
    \centering
    \includegraphics[scale = 0.47]{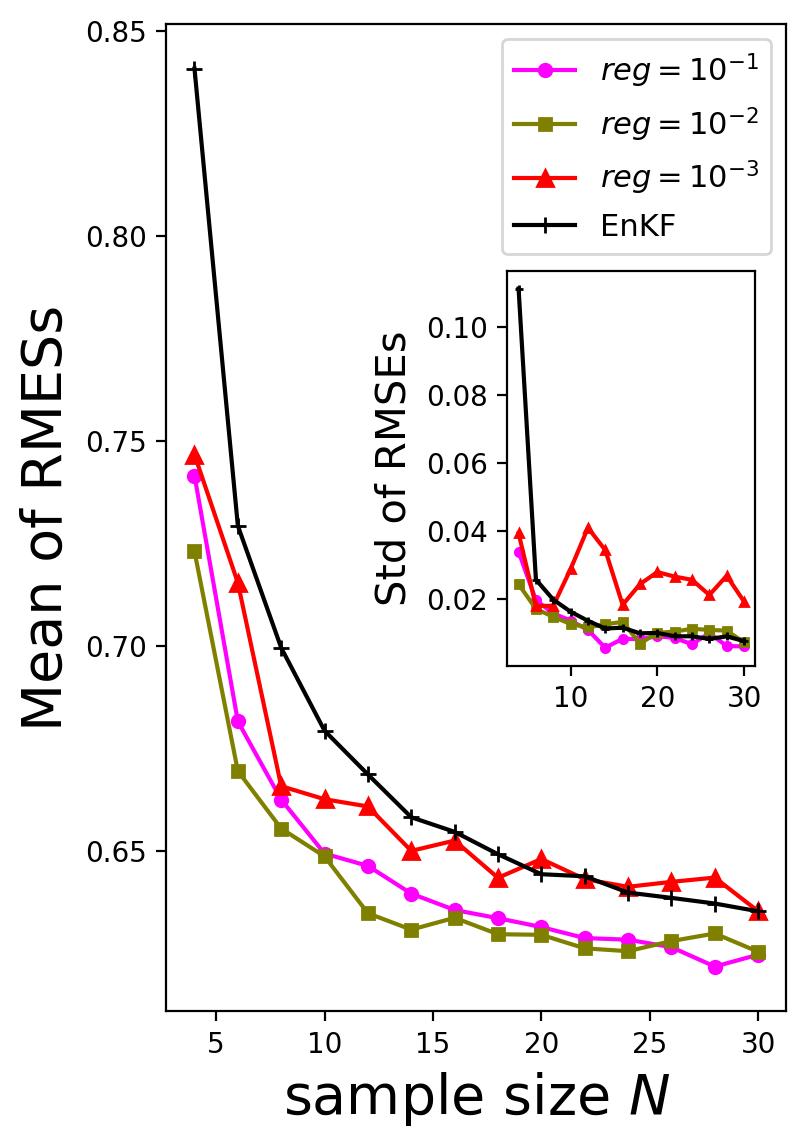}
    \caption{$q = 1.4 \Delta t$}
    \label{fig:filter_KAKMED_XL}
\end{subfigure}
\hfill
\begin{subfigure}{0.32\textwidth}
    \centering
    \includegraphics[scale = 0.47]{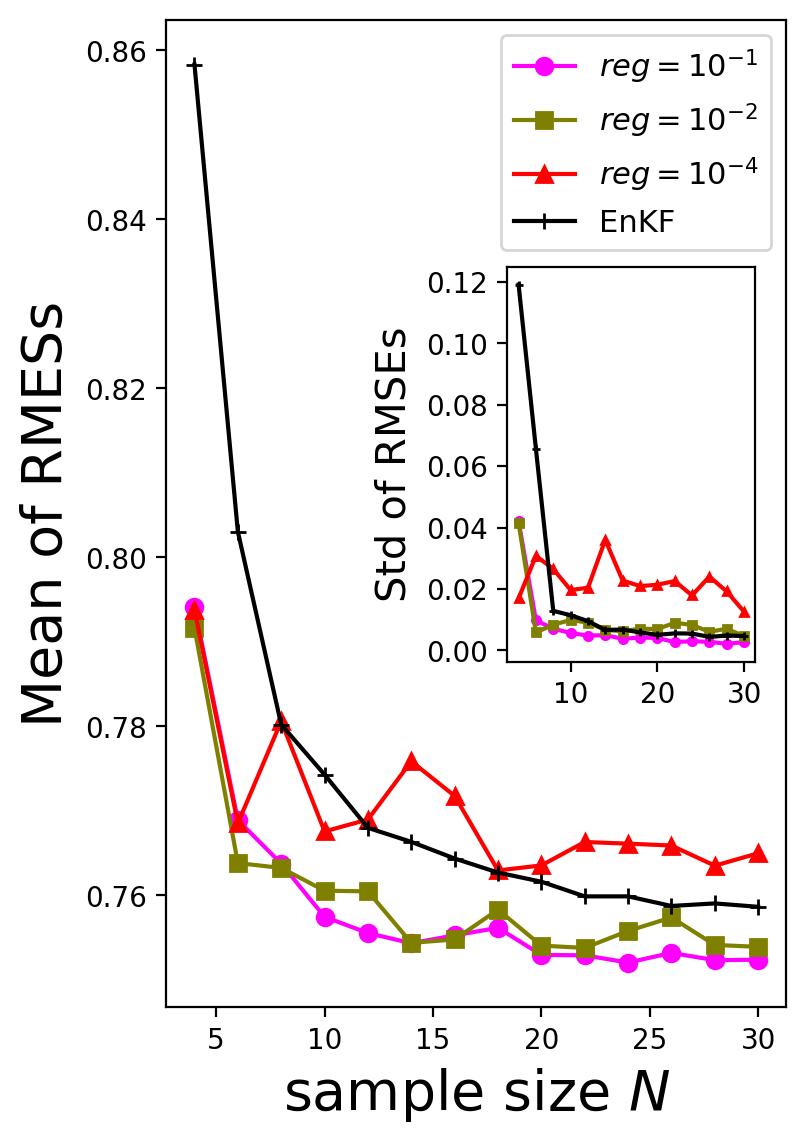}
        \caption{$q = \frac{1.4 \Delta t}{10}$}
    \label{fig:filter_KAKMED_L}
\end{subfigure}
\hfill
\begin{subfigure}{0.32\textwidth}
    \centering
    \includegraphics[scale = 0.47]{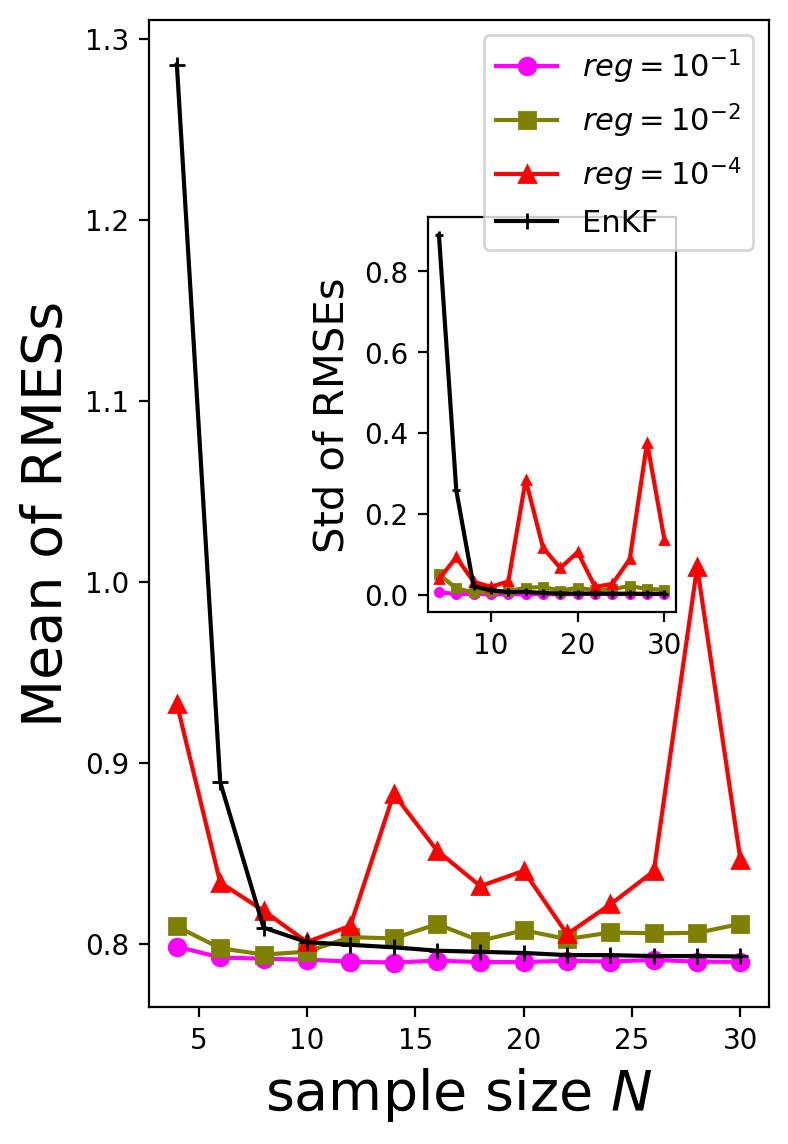}
    \caption{$q = \frac{1.4 \Delta t}{1000}$.}
    \label{fig:filter_KAKMED_XS}
\end{subfigure}
\caption{EnKF vs Kalman-adjusted KMED for different regularisations, for large model error (left), small model error (middle),  and tiny model error (right).}
\label{fig:KAKMED_benchmark}
\end{figure}

Denoting the approximate posterior samples at time $t_{j}$ as $X_{j} := \left ( x_{i,t_{j}}, y_{i,t_{j}}, z_{i,t_{j}}\right )_{i=1}^{N}$, the accuracies of the filters are measured in terms of the root mean squared errors (RMSE)
\begin{equation}
    \label{eq:RMSE}
    \text{RMSE}(t_{j}) := \sqrt{\tfrac{1}{3}\norm{\beta_{j} - \Bar{X}_{j}}},
\end{equation}
where $\Bar{X}_{j}:= \tfrac{1}{N}\sum_{i=1}^{N} \left ( x_{i,t_{j}}, y_{i,t_{j}}, z_{i,t_{j}}\right )$ are the sample means of the approximate posteriors, and $\beta_j$ are the true observations generated by \eqref{eq:observation model}. Following \cite{stordal2021p}, we discard the initial RMSEs up to $j \le 20$, and plot the time-averaged RMSEs (that is, averaged over the remaining indices $j > 20$) as a function of the ensemble size and for different regularisation strengths. The results are shown in Figure \ref{fig:KMED_benchmark} for KMED and in Figure \ref{fig:KAKMED_benchmark} for KA-KMED.

\emph{Observations:} Firstly, we note that all methods show degraded performance in terms of average RMSE as the model error decreases (see the values on the vertical axes in Figures \ref{fig:KMED_benchmark} and \ref{fig:KAKMED_benchmark}), owing to the fact that the filtering problem becomes less Gaussian and thus more difficult. 

Secondly, it is remarkable that the RMSE is not always monotonically decreasing as a function of $N$. Mainly, this phenomenon can be observed when the regularisation is fairly small, see Figures \ref{fig:filter_KMED_L}, \ref{fig:filter_KMED_XS} and \ref{fig:filter_KAKMED_XS}. Larger values of the regularisation tend to restore monotonicity, albeit at the expense of overall performance in terms of RMSE, see, for instance Figure \ref{fig:filter_KMED_L}. These observations can be explained with the nature of the linear system \eqref{eq:IPS linear_system}: For a growing number of particles, the Gram matrix $\mathbf{G}_t$ tends to become ill-conditioned, as the distance between the particles tends to be smaller. Larger values of $\varepsilon$ can mitigate this problem, but the system \eqref{eq:IPS} is further away from the idealised setting with $\varepsilon = 0$, incurring errors in the posterior approximation.

Thirdly, the Kalman adjustment is fairly helpful; in particular, Kalman-adjusted KMED can outperform the EnKF across all considered ensemble sizes and model errors when the regularisation is chosen  sufficiently large ($\varepsilon = 10^{-1}$). Notice that we have chosen the regularisations for the Kalman-adjusted scheme larger than for the plain KMED. Indeed, in the adjusted scheme, we expect the additional velocity (the first term on the right-hand side of \eqref{eq:IPS ODE}) to be relatively smaller (this is precisely the motivation for including the baseline vector field $v_t^0$), and so fluctuations in the solutions to the linear system \eqref{eq:IPS linear_system} should be suppressed commensurately.

\emph{Further illustrations and experiments:} In Figure \ref{fig:Method compare Lorentz63}, we show a comparison between Kalman-adjusted KMED (with regularisation $\varepsilon = 10^{-1}$), the EnKF and the mapping particle filter (MPF). This experiment showcases the potential of Kalman-adjusted KMED, providing more accurate filtering estimates than the mapping particle filter or the EnKF, accross all considered ensemble sizes and model errors. In particular, note that MPF is unstable for small or tiny model errors, in agreement with the results by \citet{stordal2021p}. The comparison to the MPF is instructive, since KMED and the MPF approach the task of estimating scores in different ways -- see the discussion around \eqref{eq:IBP_ KEM for score estimation} in the introduction.

To illustrate the results from Figures \ref{fig:KMED_benchmark} and \ref{fig:KAKMED_benchmark}, we compare the means of the approximate posteriors to the true observations in Figure \ref{fig:Lorentz63_example}, for  $N = 4$ ensemble members and model error $q = \frac{1.4 \Delta t}{1000}$. We only include plots for EnKF and KA-KMED, as the difference between KMED and KA-KMED is difficult to see in these plots. We observe that KA-KMED produces reliable filter estimates across the whole time interval ($400$ assimilation steps). In contrast, the EnKF looses track of the true signal and produces erroneous estimates after roughly $100$ assimilation steps.
\begin{figure}
\centering
\begin{subfigure}{0.32\textwidth}
    \centering
    \includegraphics[scale = 0.47]{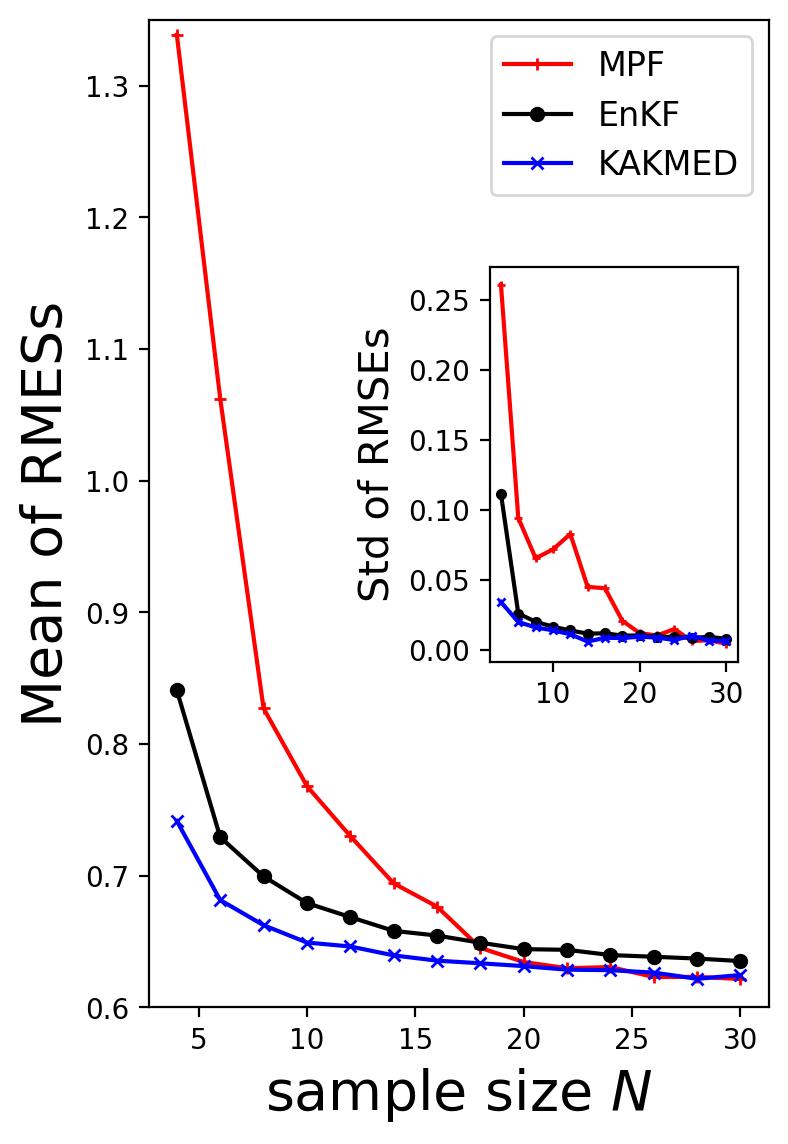}
    \caption{$q = 1.4 \Delta t$}
    \label{fig:filter_modelcompare_XL}
\end{subfigure}
\hfill
\begin{subfigure}{0.32\textwidth}
    \centering
    \includegraphics[scale = 0.47]{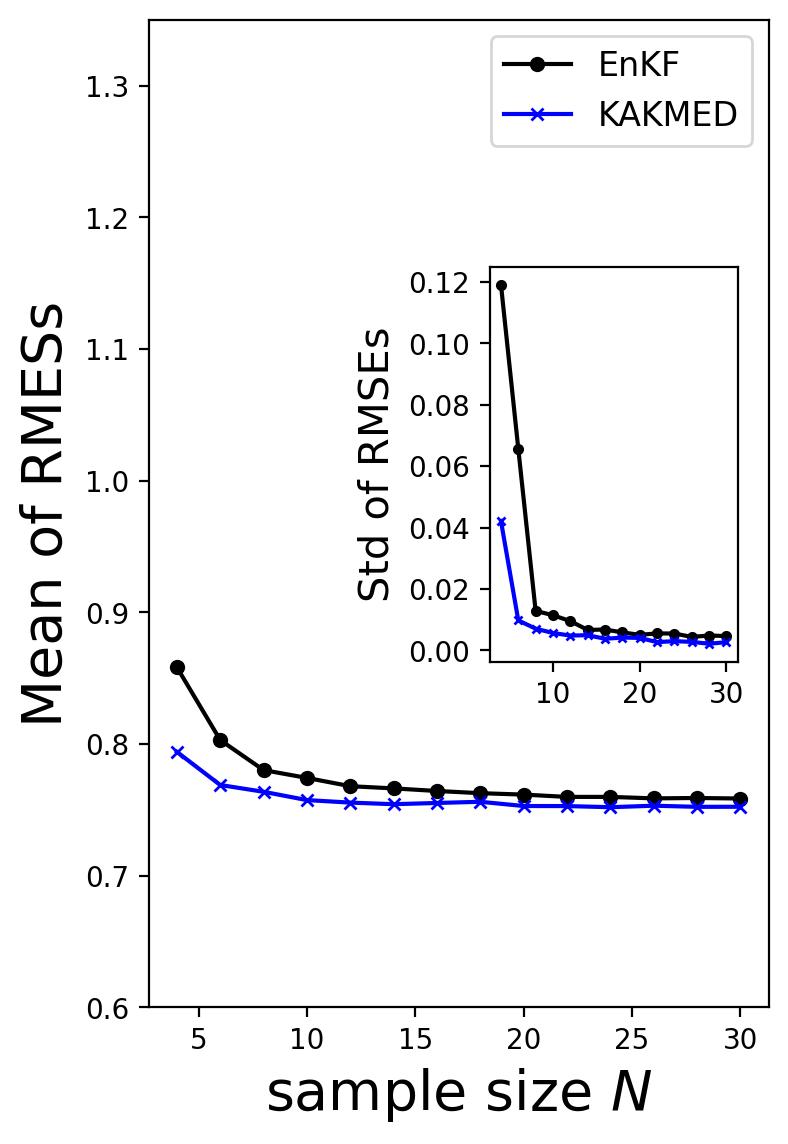}
        \caption{$q = \frac{1.4 \Delta t}{10}$}
    \label{fig:filter_modelcompare_L}
\end{subfigure}
\hfill
\begin{subfigure}{0.32\textwidth}
    \centering
    \includegraphics[scale = 0.47]{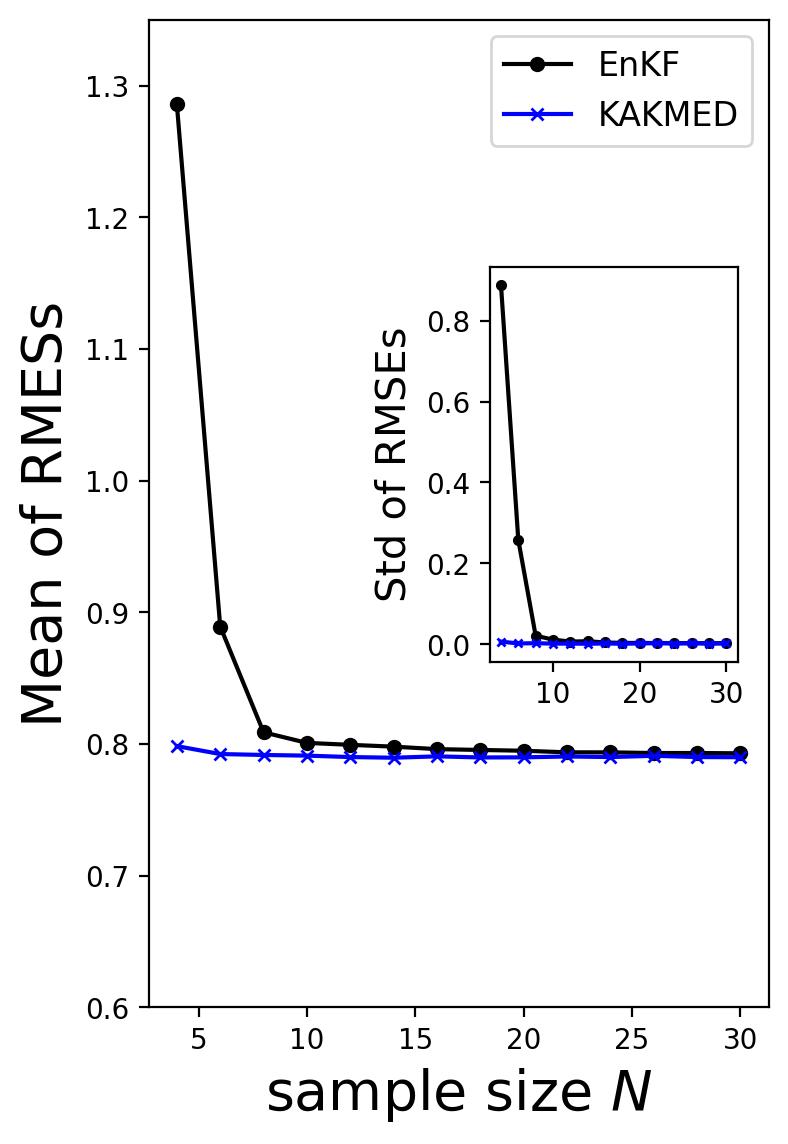}
    \caption{$q = \frac{1.4 \Delta t}{1000}$.}
    \label{fig:filter_modelcompare_XS}
\end{subfigure}
\caption{Kalman-adjusted KMED (with regularisation $\varepsilon = 10^{-1}$) vs EnKF and MPF. The MPF is implemented with 200 iterations for the gradient descent phase using the Adagrad learning rate. The MPF becomes very unstable for small (middle) and tiny (right) model errors, and the corresponding RMSEs are too large to plot.}
\label{fig:Method compare Lorentz63}
\end{figure}
\begin{remark}[Setting of the experiments]
 In our experiments, the initialisation of the Lorenz-63 system at $(-0.587, -0.563, 16.870)$ is the same as the setup used by \cite{reich2015probabilistic}, a commonly selected point to demonstrate the chaotic nature of the Lorenz-63 system. For the Runge-Kutta time step $\delta t = 0.01$ and the observation interval $\Delta t = 0.1$, our settings represent a compromise between those recommended by \cite{reich2015probabilistic} and \cite{stordal2021p}. Specifically, \cite{reich2015probabilistic} employ $\delta t = 0.001$ and $\Delta t = 0.05$ to reduce system chaos through high-precision Runge-Kutta steps, leading to more stable filtering outcomes with fewer errors. In contrast, \cite{stordal2021p} utilize $\delta t = 0.01$ and $\Delta t = 0.25$ to enhance the system's chaotic dynamics, thereby increasing the challenge of filtering and highlighting performance differences among various methods. However, the settings in \cite{stordal2021p} pose practical challenges for maintaining stability in KMED, even with a range of regularisation adjustments. Preliminary experiments indicate that KMED can be made stable in this more challenging setting by increasing the number of discretisation time steps in Algorithm \ref{alg:KME dynamics} (we used approximately 400), achieving similar superior performances.
\end{remark}

\textbf{Lorenz-96.} We set $d = 10$, $F = 8$, $\delta t = 0.001$, $\Delta t = 0.05$, and $N_{\text{assi}} = 400$. The system is initialised randomly from $(x^{1}_{0}, x^{2}_{0}, \ldots, x^{10}_{0}) \sim \mathcal{N}(0, I_{10 \times 10})$, and the covariance of the observation error is set to $R = 0.7 I_{10 \times 10}$. For the filtering step, we initiate samples at $t = 0$ from ${\mathcal{N}((x^{1}_{0}, x^{2}_{0}, \ldots, x^{10}_{0}), 0.01 I_{10 \times 10})}$.
Due to the high computational cost, we skip the detailed comparisons for different ensemble sizes and model errors, and only show the equivalent of Figure \ref{fig:Lorentz63_example} as an illustration in Figure \ref{fig:Lorentz96_example}. In this experiment, we take the ensemble size to be $N = 8$, and $Q = (1.4 \Delta t \times 10^{-9}) I_{10\times 10}$ as the model error. For KMED, we take the bandwidth of the RBF kernel as $\sigma = 3.5$ and the regularisation as $\varepsilon = 10^{-5}$. Again, the plots for KMED and KA-KMED are difficult to distinguish. Therefore, for variation, we only plot the approximate posterior means for KMED. Similarly as in the experiment shown in Figure \ref{fig:Lorentz63_example}, KMED produces accurate filter estimates, while the EnKF looses track after an initial period (approximately 200 assimilation steps).
\section{Discussion and Outlook}
\label{sec:outlook}
We have developed KME-dynamics as an algorithmic approach for general Bayesian inference, which does not require the explicit score of the target posterior, and is therefore applicable to settings in data assimilation. Based on kernel mean embeddings,  KME-dynamics captures statistical information in reproducing kernel Hilbert spaces, and moves the ensemble members in an interacting particle system accordingly. We have connected KME-dynamics with score-based generative modelling, and constructed a new kernel-based estimator for the score function. The framework of KME-dynamics seamlessly accommodates the Kalman-Bucy filter, as well as a predefined baseline dynamics. The latter modification has shown particular promise in our numerical experiment for the Lorenz-63 and Lorenz-96 systems, combining a Kalman update with a nonlinear refinement. We also extend the framework of KME-dynamics by weighting the samples, correcting numerical and statistical error. Across all our experiments, we have observed benefits from choosing a characteristic kernel, incorporating information through the KME in a lossless manner.

In future work, we plan to build on the exposed  connection between KME-dynamics, kernelised diffusion maps and Tikhonov functionals, leveraging insights from statistical learning theory to improve both the implementation of KME-dynamics as well as its theoretical understanding. It is also of interest to clarify the connection of KME-dynamics to Stein variational gradient descent \citep{liu2016stein}, and to its geometrical underpinnings \citep{nusken2023geometry,nusken2023stein}. Finally, it may be interesting to consider more general interpolations than \eqref{eq:interpolation intro}, in the spirit of \cite{syed2021parallel,Syed_2022}, or to extend KME-dynamics to applications involving optimal control, in the spirit of \cite{bevanda2024data}.

\begin{figure}
\centering
\begin{subfigure}{\textwidth}
    \centering
    \includegraphics[scale = 0.7]{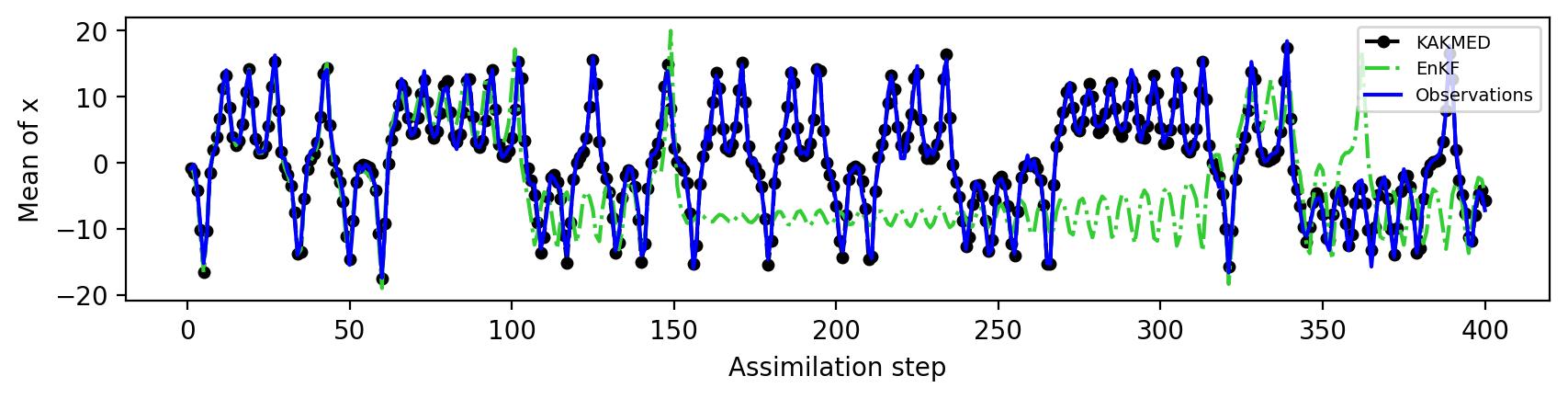}
    \label{fig:Lorentz_x}
\end{subfigure}
\hfill
\begin{subfigure}{\textwidth}
    \centering
    \includegraphics[scale = 0.7]{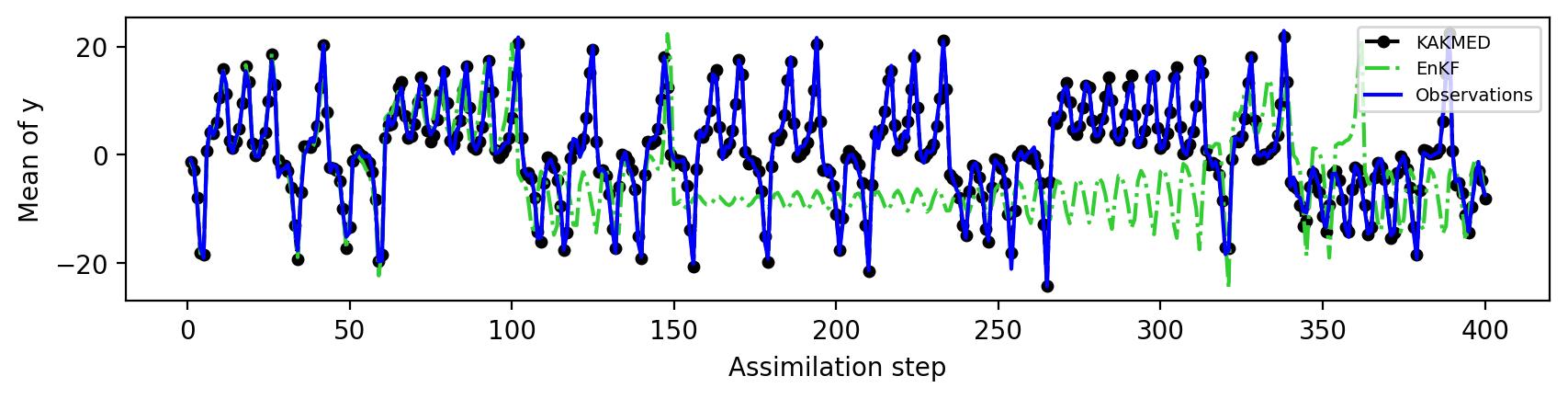}
    \label{fig:Lorentz_y}
\end{subfigure}
\hfill
\begin{subfigure}{\textwidth}
    \centering
    \includegraphics[scale = 0.7]{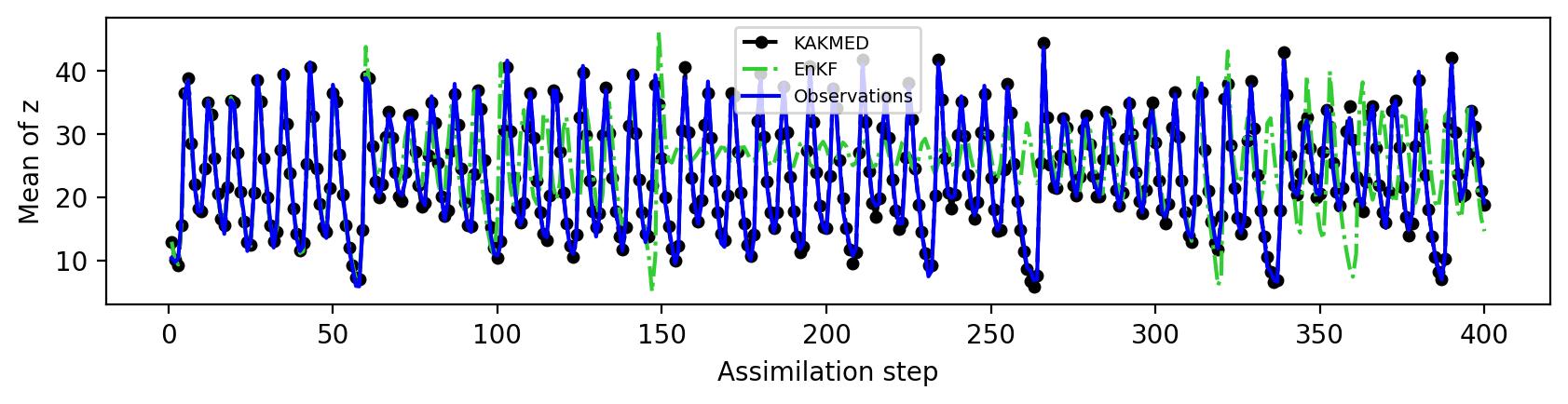}
    \label{fig:Lorentz_z}
\end{subfigure}
\caption{Lorenz-63: We compare the approximate posterior means obtained by 
Kalman adjusted-KMED and EnKF to the true observations. While Kalman-adjusted KMED is able to track the dynamics accurately, the EnKF produces erroneous filtering estimates after an initial period of approximately 100 assimilation steps.}
\label{fig:Lorentz63_example}
\end{figure}
\begin{figure}
\centering
\begin{subfigure}{\textwidth}
    \centering
    \includegraphics[scale = 0.7]{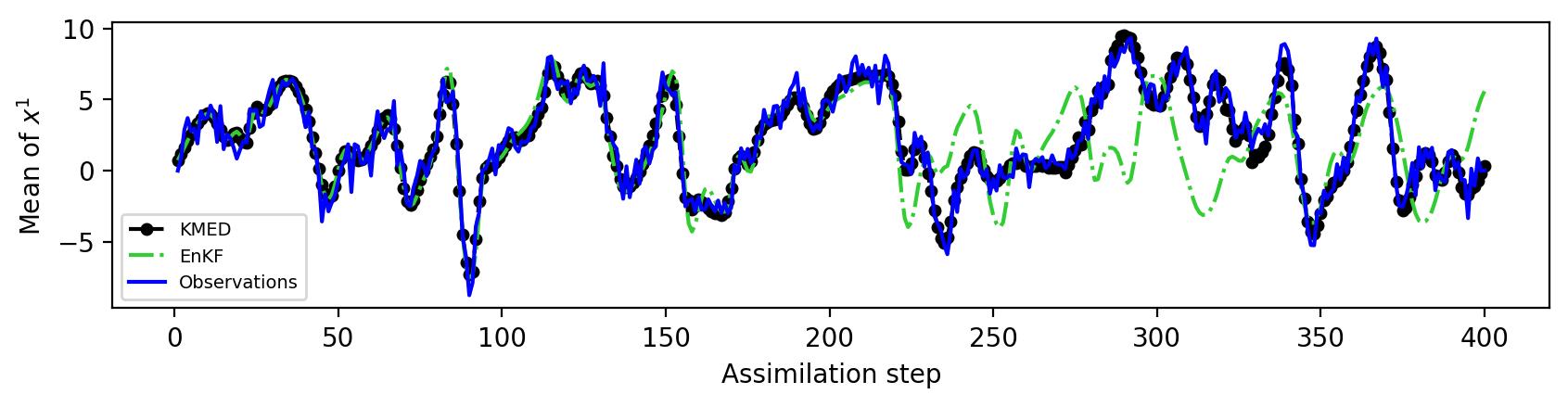}
    \label{fig:Lorentz96_x}
\end{subfigure}
\hfill
\begin{subfigure}{\textwidth}
    \centering
    \includegraphics[scale = 0.7]{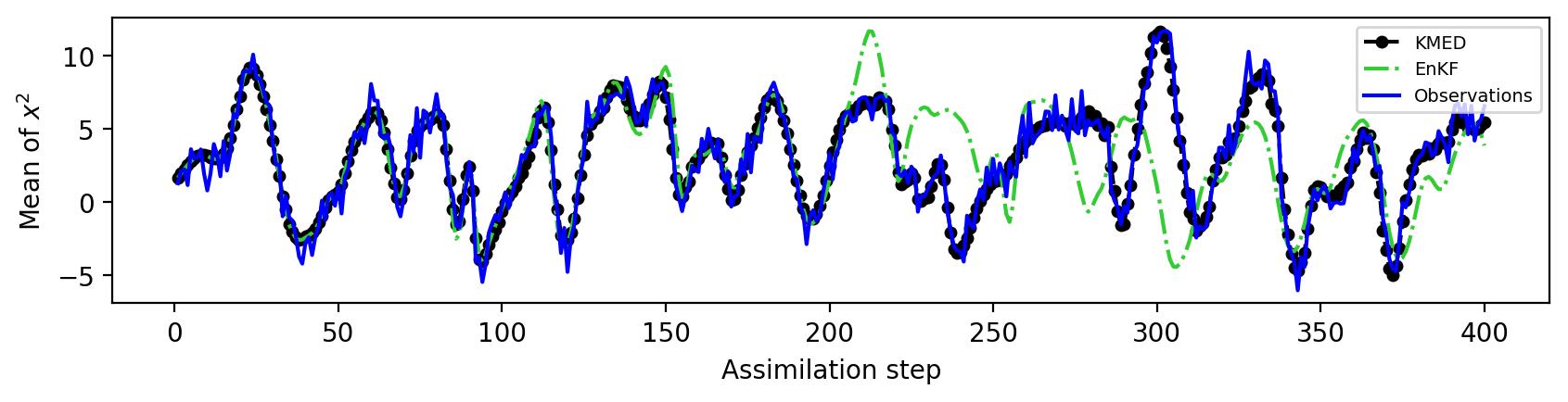}
    \label{fig:Lorentz96_y}
\end{subfigure}
\hfill
\begin{subfigure}{\textwidth}
    \centering
    \includegraphics[scale = 0.7]{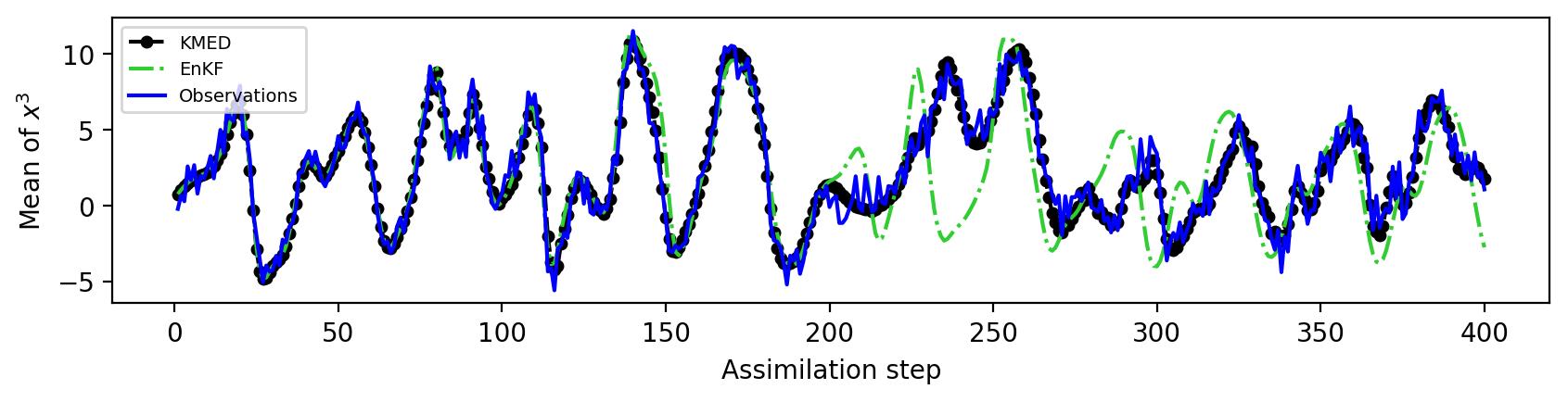}
    \label{fig:Lorentz96_z}
\end{subfigure}
\caption{Lorenz-96: We compare the approximate posterior means (first three coordinates) obtained by 
KMED and EnKF to the true observations. While  KMED is able to track the dynamics accurately, the EnKF produces erroneous filtering estimates after an initial period of approximately 200 assimilation steps.}
\label{fig:Lorentz96_example}
\end{figure}

\newpage

\appendix
\section{Proofs and calculations}
\label{app:calc}

\textbf{Proof of Lemma \ref{lem:KME ODE}}

For \eqref{eq:KME cov}, 
we begin by computing the time derivative of the normalising constant,
\begin{equation*}
\partial_t Z_t = \partial_t \int_{\mathbb{R}^d} e^{-th} \, \mathrm{d}\pi_0 = - \int_{\mathbb{R}^d} h e^{-th} \, \mathrm{d} \pi_0 = - Z_t \int_{\mathbb{R}^d} h \, \mathrm{d}\pi_t.
\end{equation*}
Using this, we see that 
\begin{equation*}
 \partial_t \pi_t = - h \pi_t   - \frac{\partial_t Z_t e^{-th} \pi_0}{Z_t^2} = -h\pi_t + \pi_t \int_{\mathbb{R}^d} h \, \mathrm{d}\pi_t.
\end{equation*}
For the kernel mean embedding, we therefore obtain
\begin{equation*}
\partial_t \Phi_k(\pi_t) = \int_{\mathbb{R}^d} k(\cdot, x) \partial_t\pi_t(\mathrm{d}x) = - \int_{\mathbb{R}^d} k(\cdot,x) \left( h(x) - \int_{\mathbb{R}^d} h \, \mathrm{d}\pi_t \right) \pi_t(\mathrm{d}x),
\end{equation*}
which can be rewritten in the form \eqref{eq:KME cov}.

For \eqref{eq:KME dyn}, notice that the continuity equation associated to \eqref{eq:ODE} reads
\begin{equation*}
\partial_t \rho_t + \nabla \cdot (\rho_t (v_t+ v_t^0)) = 0.
\end{equation*}
For the kernel mean embedding, we thus obtain
\begin{subequations}
\begin{align*}
\partial_t \Phi_k(\rho_t) & = - \int_{\mathbb{R}^d} k(\cdot, x) \nabla_x \cdot (\rho_t (v_t+ v_t^0))(x) \, \mathrm{d}x
 = \int_{\mathbb{R}^d} \nabla_x k(\cdot,x) \cdot (v_t + v_t^0)(x) \, \mathrm{d}x,
\end{align*}
\end{subequations}
as required. 

\textbf{Proof of Lemma \ref{lem:G properties}}

1.) Clearly, $G_{\rho, C}$ is symmetric, by the symmetry of $C$. To check positive definiteness, we fix $N \in \mathbb{N}$, choose $(\alpha_i)_{i=1}^N \subset \mathbb{R}$ and $(x_i)_{i=1}^N \subset \mathbb{R}^d$, and notice that
\begin{align*}
\sum_{i,j=1}^N \alpha_{i} \alpha_{j} G_{\rho,C}(x_i,x_j) & = \int_{\mathbb{R}^d} \left( \sum_{i=1}^N \alpha_i \nabla_z k(x_i,z)\right) \cdot C \left( \sum_{i=1}^N \alpha_i \nabla_z k(x_i,z)\right) \rho(\mathrm{d}z) 
\\
& = \int_{\mathbb{R}^d} \left| C^{1/2}\sum_{i=1}^N \alpha_i \nabla_z k(x_i,z) \right|^2 \rho(\mathrm{d}z) \ge 0
\end{align*}
2.) This follows from the general theory of regularised inverse problems, see, for example \citet[Section 2.2]{kirsch2011introduction}. More specifically, the integral operator 
\begin{equation*}
 B: \quad \alpha \mapsto \int_{\mathbb{R}^d} G_{\rho,C} \alpha(y) \rho(\mathrm{d}y) + \varepsilon \alpha,
\end{equation*}
mapping $L^2(\rho)$ into itself, is self-adjoint, bounded, and satisfies the coercivity estimate
\begin{equation}
\langle \alpha, B \alpha \rangle_{L^2(\rho)} \ge \varepsilon \Vert \alpha \Vert^2_{L^2(\rho)}, \qquad \alpha \in L^2(\rho).
\end{equation}
Therefore, $B$ admits a bounded inverse (by the open mapping theorem \citep[Chapter 2]{rudin1974functional}), implying the statement.

\textbf{Proof of Proposition \ref{pro:ReproducingKF}}

We will need the following lemma, see, for instance, \cite{michalowicz2009isserlis}.
\begin{lemma}[Gaussian moments]\label{le:moments_normal}
For a Gaussian random variable $X \sim \mathcal{N}(\mu,\Sigma)$ in $\mathbb{R}^d$, denote $\Omega_{ij}:= \mu_{i} \mu_{j}$. Then we have the following expressions for the moments:
\begin{subequations}
\begin{align}
    \E[X_i]  = & \mu_{i}, \qquad 
    \E[X_i X_j] = \Sigma_{ij} + \Omega_{ij}, \label{eq:secondmoment}
    \\
    \E[X_i X_j X_k]  =  &\mu_{i}\Sigma_{jk} + \mu_{j}\Sigma_{ik} +\mu_{k}\Sigma_{ij} + \Omega_{ij} \mu_{k}, \label{eq:thirdmoment}
    \\
    \E[X_i X_j X_k X_l]  = & \Sigma_{ij}\Sigma_{kl} + \Sigma_{ik}\Sigma_{jl} + \Sigma_{il}\Sigma_{jk} + \Omega_{ij}\Sigma_{kl} + \Omega_{ik}\Sigma_{jl} +
    \\
    &  + \Omega_{il}\Sigma_{jk} + \Omega_{jk}\Sigma_{il} + \Omega_{jl}\Sigma_{ik} + \Omega_{kl}\Sigma_{ij} + \Omega_{ij}\Omega_{kl}, \label{eq:fourthmoment}
\end{align}
\end{subequations}
for $i,j,k = 1,\ldots, d$.
\end{lemma}
Introducing the short-hand notation $\rho[h]:= \E_{\rho_{t}}(h) = \int_{\R ^{d}} h(a)\, \rho_{t}(\di a)$, setting $v^0_t = 0$ and fixing $y \in \mathbb{R}^d$, we first compute the right-hand side $f_{t}$ of \eqref{eq:int MF}:
\begin{align*}
    &= \int_{\R ^{d}}(x^{\top} y + 1)^{2} \left(h(x) - \rho_{t}[h] \right) \rho_{t}(\di x) = \int_{\R ^{d}}(y^{\top} x x^{\top} y + 2x^{\top} y  + 1) \left(h(x) - \rho_{t}[h] \right)\rho_{t} (\di x)\\
    &= \int_{\R ^{d}}(y^{\top} x x^{\top} y + 2x^{\top} y) \left(h(x) - \rho_{t}[h] \right)\rho_{t}(\di x)\\
    &= \int_{\R ^{d}} \left\{y^{\top} x x^{\top} y \left(h(x) - \rho_{t}[h] \right) + 2x^{\top} y \left(h(x) - \rho_{t}[h] \right)\right\} \rho_{t} (\di x) = : \circled{1} + \circled{2},
\end{align*}
where we denote $$\circled{1} :=  \int_{\R ^{d}} y^{\top} x x^{\top} y \left(h(x) - \rho_{t}[h] \right)\rho_{t} (\di x) \quad \text{and} \quad \circled{2} := \int_{\R ^{d}}2x^{\top} y \left(h(x) - \rho_{t}[h] \right) \rho_{t}(\di x).$$ 
Now we compute
\begin{align*}
    &\circled{1} = \int_{\R ^{d}}y^{\top} x x^{\top} y (h(x) - \rho_{t}[h])\rho_{t}(\di x) = \int_{\R ^{d}}y^{\top} x x^{\top} y (h(x) - \rho_{t}[h])\rho_{t}(\di x) 
    \\
    & = \int_{\R ^{d}}y^{\top} x x^{\top} y \left\{\tfrac{1}{2}(Hx - \beta)^{\top} R^{-1} (Hx - \beta) - \rho_{t}[h]\right\} \rho_{t}(\di x) 
    \\
    &= \tfrac{1}{2}\int_{\R ^{d}}y^{\top} x x^{\top} y x^{\top} H^{\top}R^{-1}Hx \rho_{t}(\di x) - \tfrac{1}{2}\int_{\R ^{d}}y^{\top} x x^{\top} y x^{\top} H^{\top}R^{-1} \beta\rho_{t}(\di x) 
    \\
    & - \tfrac{1}{2}\int_{\R ^{d}}y^{\top} x x^{\top} y \beta^{\top} R^{-1}H x \rho_{t}(\di x) + \tfrac{1}{2}\int_{\R ^{d}}y^{\top} x x^{\top} y \beta^{\top} R^{-1} \beta \rho_{t}( \di x)\\
    &- \int_{\R ^{d}}\tfrac{1}{2}(Ha - \beta)^{\top} R^{-1} (Ha - \beta) \rho_{t}(\di a) \cdot \int_{\R ^{d}}y^{\top} x x^{\top} y \rho_{t}(\di x).
\end{align*}
Denote $A: = H^{\top}R^{-1}H$, $B:= H^{\top}R^{-1}$, $M:= R^{-1}H$, $D:= R^{-1}$ and rewrite \circled{1} in index form:
\begin{align*}
    &\circled{1} = \tfrac{1}{2} {\E}_{X_{t}} \sum_{i,j,k,l}y_{i}x_{i}x_{j}y_{j} x_{k}A_{kl}x_{l} - \tfrac{1}{2} {\E}_{X_{t}} \sum_{i,j,k,l}y_{i}x_{i}x_{j}y_{j} x_{k}B_{kl}\beta_{l} - \tfrac{1}{2} {\E}_{X_{t}} \sum_{i,j,k,l}y_{i}x_{i}x_{j}y_{j} \beta_{k}M_{kl}x_{l} \\
    &+ \tfrac{1}{2} {\E}_{X_{t}} \sum_{i,j,k,l}y_{i}x_{i}x_{j}y_{j} \beta_{k}D_{kl}\beta_{l} - \frac{1}{2} \left({\E}_{X_{t}}\sum_{i,j}y_{i}x_{i}x_{j}y_{j}\right)\left({\E}_{X_{t}} \sum_{k,l}x_{k}A_{kl}x_{l}\right)\\
    &+ \tfrac{1}{2} \left({\E}_{X_{t}}\sum_{i,j}y_{i}x_{i}x_{j}y_{j}\right)\left({\E}_{X_{t}} \sum_{k,l}x_{k}B_{kl}\beta_{l}\right) + \tfrac{1}{2} \left({\E}_{X_{t}} \sum_{i,j}y_{i}x_{i}x_{j}y_{j}\right)\left({\E}_{X_{t}}\sum_{k,l}\beta_{k}M_{kl}x_{l} \right)\\
    &- \tfrac{1}{2} \left({\E}_{X_{t}} \sum_{i,j}y_{i}x_{i}x_{j}y_{j}\right)\left({\E}_{X_{t}} \sum_{k,l}\beta_{k}D_{kl}\beta_{l}\right).
\end{align*}
Now using Lemma \ref{le:moments_normal} and the linear properties of expectation and summation, we can compute \circled{1} and simplify:
\begin{align*}
    \circled{1} = & \tfrac{1}{2}\sum_{i,j,k,l}\{ y_{i}y_{j} [(A_{kl} C_{t}^{ik} C_{t}^{jl} + A_{kl} C_{t}^{il} C_{t}^{jk} + A_{kl}\Omega_{ik}C_{t}^{jl} + A_{kl}\Omega_{il}C_{t}^{jk} + A_{kl}\Omega_{jk}C_{t}^{il} + A_{kl}\Omega_{jl}C_{t}^{ik}\\
    &- (B_{kl}\beta_{l}C_{t}^{jk}\mu_{t}^{i} + B_{kl}\beta_{l}C_{t}^{ik}\mu_{t}^{j}) - (M_{kl}\beta_{k}C_{t}^{jl}\mu_{t}^{i} + M_{kl}\beta_{k}C_{t}^{il}\mu_{t}^{j})]\}
    \\
    & = \frac{1}{2}y^{\top} \{C_{t} A C_{t}^{\top} + C_{t} A^{\top} C_{t}^{\top} + \Omega A C_{t}^{\top} + \Omega A^{\top} C_{t}^{\top} + C_{t} A \Omega^{\top} + C_{t} A^{\top} \Omega^{\top}\\
    &- [\mu_{t} (C_{t} B \beta)^{\top} + (C_{t} B \beta) \mu_{t}^{\top}] - [\mu_{t} (C_{t} M^{\top} \beta)^{\top} + (C_{t} M^{\top} \beta) \mu_{t}^{\top}]\}y\\
    &= y^{\top} \{C_{t} A C_{t} +  \Omega A C_{t} + C_{t} A\Omega - [\mu_{t} (C_{t} B \beta)^{\top} + (C_{t} B \beta) \mu_{t}^{\top}]\} y,
\end{align*}
recalling that $\Omega_{ij} := \mu_{t}^{i}\mu_{t}^{j}$. Similarly, we can derive
\begin{align*}
    \circled{2} &= \sum_{i,j,k} \{y_{i}[(A_{jk}\mu_{t}^{j}C_{t}^{ik} + A_{jk}\mu_{t}^{k}C_{t}^{ij}) - B_{jk}\beta_{k}C_{t}^{ij} - M_{jk}\beta_{j}C_{t}^{ik}]\} = y^{\top} [2C_{t} A \mu_{t} - 2 C_{t} B \beta].
\end{align*}
We now consider the left-hand side of \eqref{eq:int MF} with $\varepsilon = 0$:
\begin{align*}
    LHS = & \int_{\R ^{d}} \int_{\R ^{d}} (2(x^{\top} y + 1)y) \cdot C_{t} (2(x^{\top} z + 1)z) \alpha_{t}(z) \rho_{t}(\di x) \rho_{t}(\di z)\\
    = &4 \int_{\R ^{d}} \int_{\R ^{d}}(x^{\top} y + 1) (x^{\top} z + 1) y^{\top} C_{t} z \alpha_{t}(z) \rho_{t}(\di x) \rho_{t}(\di z)\\
    = & 4 \Big({\E}_{X_{t}}{\E}_{Z_{t}} \sum_{i,j,k}x_{i}y_{i}x_{k}z_{k}y_{j}C_{t}^{jl}z_{l}\phi(z)+ {\E}_{X_{t}}{\E}_{Z_{t}} \sum_{i,j} x_{i}y_{i}y_{j}C_{t}^{jl}z_{l}\phi(z)\\
    &+ {\E}_{X_{t}}{\E}_{Z_{t}} \sum_{i,j} x_{i}z_{i}y_{j}C_{t}^{jl}z_{l}\alpha_{t}(z) +  {\E}_{X_{t}}{\E}_{Z_{t}} \sum_{i,j} y_{i}C_{t}^{il} z_{l}\alpha_{t}(z)\Big),
\end{align*}
where $Z$ is an independent identically distributed copy of $X$.
Defining $(\Psi)_{ij} := \psi_{ij} := -\E_{Z_{t}} [z_{i}z_{j} \alpha_{t}(z)]$ and $(\eta)_{i}:= \eta_{i} := -\E_{Z_{t}} [z_{i}\alpha_{t}(z)]$, we can write this as
\begin{align*}
    LHS &= -4 \left(\sum_{i,j,k}y_{i}y_{j}(C_{t}^{ik} + \Omega_{ik})C_{t}^{jl}\psi_{lk} + \sum_{i,j} y_{i}y_{j}\mu_{t}^{i}C_{t}^{jl}\eta_{l} + \sum_{i,j} y_{j} \mu_{t}^{i} C_{t}^{jl}\psi_{li} + \sum_{i} y_{i} C_{t}^{il}\eta_{l}\right)\\
    &= -4y^{\top}((C_{t} + \Omega) \Psi C_{t}  + \mu_{t} \eta^{\top}C_{t})y -4y^{\top}(C_{t} \Psi \mu_{t} + C_{t} \eta).
\end{align*}
 To solve $LHS = RHS$ for all $y$, it is equivalent to solve the following system of equations:
 \begin{align*}
    C_{t} (8\Psi + 2A) C_{t} &+ \Omega (4\Psi + 2A) C_{t} + C_{t} (4\Psi + 2A) \Omega + 8\mu_{t} \eta^{\top}C_{t} + 8C_{t} \eta \mu_{t}^{\top} \\
    &- 2\mu_{t} (C_{t} B \beta)^{\top} - 2(C_{t} B \beta) \mu_{t}^{\top} = 0,\\
    4 (C_{t}\Psi\mu_{t} + C_{t}\eta) &+ 2C_{t} A \mu_{t} - 2 C_{t} B \beta = 0.
 \end{align*}
Subtracting the second equation and the transpose of the second equation from the first equation we obtain
\[
8 C_{t} \Psi C_{t} + 2 C_{t} A C_{t} = 0,
\]
which implies $
\Psi = - \frac{1}{4}A$. 
Substituting this back into the second equation we see that
\[
\eta = \frac{1}{2}B \beta - \frac{1}{4}A \mu_{t}.
\]
Finally, we compute the velocity field on the right-hand side of \eqref{eq:MF ODE}:
\begin{align*}
    v(x) = - C_{t}\int_{\R ^{d}} \nabla_{x}k(x, z) \alpha_{t}(z) \rho_{t}(\di z) =  -C_{t}\int_{\R ^{d}}  (2(x^{\top} z + 1)z) \alpha_{t}(z) \rho_{t}(\di z).
\end{align*}
In index form this reads
\begin{align*}
    v_{j}(x) &= -2 \sum_{i,l}C_{t}^{jl} x_{i}{\E}_{Z_{t}}[z_{i}z_{l}\alpha_{t}(z)] - 2 \sum_{l} C_{t}^{jl}{\E}_{Z_{t}}[z_{l}\alpha_{t}(z)] = 2 \sum_{i,l}x_{i} C_{t}^{jl} \psi_{li} + 2 \sum_{l}C_{t}^{jl}\eta_{l},
\end{align*}
or, in matrix form:
\[
v(x) = -\frac{1}{2}C_{t} A (x+\mu_{t}) + C_{t} B \beta.
\]
Recalling that $\mu_{t} = \mu(\rho_{t})$, we substitute the values of $A$ and $B$ back and hence conclude \eqref{eq:KB ODE}.

\textbf{Proof of Lemma \ref{lem:flow of weight}}

    The total derivative of $w_{t}^{i}$ is given by
    \begin{equation}
    \label{eq:total derivative of weight}
    \frac{\mathrm{d}w_{t}^{i}}{\mathrm{d}t} = \frac{\partial w_{t}}{\partial t}(X_{t}^{i}) + \nabla w_{t}(X_{t}^{i}) \cdot \frac{\mathrm{d}X_{t}^{i}}{\mathrm{d}t}(X_{t}^{i}).
\end{equation}
We now compute the three terms on the right-hand side of \eqref{eq:total derivative of weight}, starting with the first one:
\begin{align*}
    \frac{\partial w_{t}}{\partial t} &= \frac{\rho_{t}\partial_{t}(e^{-th}\pi_{0}) - e^{-th}\pi_{0} \,\partial_{t}\rho_{t}}{\rho_{t}^{2}} = -\frac{h e^{-th} \pi_{0}}{\rho_{t}} - \frac{e^{-th} \pi_{0}}{\rho_{t}}\frac{\partial_{t}\rho_{t}}{\rho_{t}}\\
    &= -hw_{t} - w_{t} \frac{\partial_{t}\rho_{t}}{\rho_{t}} = -w_{t}(h - \nabla \log \rho_{t} \cdot u_{t} - \nabla\cdot u_{t}),
\end{align*}
where in the last step we have used the continuity equation \eqref{eq:continuity equation for flow}. Hence, 
\begin{equation}
    \label{eq:time derivative of weight}
    \frac{\partial w_{t}}{\partial t}(X_{t}^{i}) = -w_{t}(X_{t}^{i})(h(X_{t}^{i}) - \nabla \log \rho_{t}(X_{t}^{i}) \cdot u_{t}(X_{t}^{i}) - \nabla\cdot u_{t}(X_{t}^{i})).
\end{equation}

Similarly, we have
\begin{align*}
    \nabla w_{t} &= \frac{\rho_{t}\nabla(e^{-th}\pi_{0}) - e^{-th}\pi_{0} \,\nabla\rho_{t}}{\rho_{t}^{2}}\\
    &= \frac{e^{-th}\pi_{0}}{\rho_{t}}\frac{\nabla(e^{-th}\pi_{0})}{e^{-th}\pi_{0}} - \frac{e^{-th}\pi_{0}}{\rho_{t}}\frac{\nabla \rho_{t}}{\rho_{t}}= w_{t}(\nabla \log \pi_{t} - \nabla \log \rho_{t}),
\end{align*}
implying that 
\begin{equation}
    \label{eq:gradient for weight}
    \nabla w_{t}(X_{t}^{i}) = w_{t}(X_{t}^{i})(\nabla \log \pi_{t}(X_{t}^{i}) - \nabla \log \rho_{t}(X_{t}^{i})).
\end{equation}
Now plugging \eqref{eq:time derivative of weight}, \eqref{eq:gradient for weight} and \eqref{eq:ODE weighting} into \eqref{eq:total derivative of weight} we obtain
\begin{equation}
    \label{eq:flow of weight in appendix}
    \frac{\mathrm{d}w_{t}^{i}}{\mathrm{d}t} = w_{t}^{i} \left(\nabla \log \pi_{t}(X_{t}^{i}) \cdot u_{t}(X_{t}^{i}) + \nabla \cdot u_{t}(X_{t}^{i}) - h(X_{t}^{i})\right).
\end{equation}
Dividing both sides by $w_{t}^{i}$ we get to the result.

\textbf{Proof of Proposition \ref{prop:regression}}

We begin by showing that $J_{\varepsilon,t}$ is well defined on $(L^2(\rho_t))^d$. For that purpose, assume first that $v$ is smooth and compactly supported. In that case, we can write 
\begin{align}
\label{eq:MMD proof}
\mathrm{MMD}_k^2(\partial_t \rho^v_t, & \partial_t \pi_t)  = \left\Vert \int_{\mathbb{R}^d} k(\cdot, x) \partial_t \rho_t^v(\mathrm{d}x)  - \int_{\mathbb{R}^d} k(\cdot,x) \partial_t \pi_t(\mathrm{d}x) \right \Vert^2_{\mathcal{H}_k}
\\
& = \int_{\mathbb{R}^d} \int_{\mathbb{R}^d} k(x,y) \left( \partial_t \rho_t^v(\mathrm{d}x) - \partial_t \pi_t(\mathrm{d}x) \right) \left( \partial_t \rho_t^v(\mathrm{d}y) - \partial_t \pi_t(\mathrm{d}y) \right)
\nonumber
\\
& = \int_{\mathbb{R}^d} \int_{\mathbb{R}^d} k(x,y) \left( \nabla\cdot(\rho_t v)(x) - \overline{h}(x) \rho_t(x) \right) \left( \nabla\cdot(\rho_t v)(y) - \overline{h}(y) \rho_t(y) \right) \mathrm{d}x \, \mathrm{d}y,
\nonumber
\end{align}
where we have introduced the short-hand notation $\overline{h} = h - \int_{\mathbb{R}^d} h \, \mathrm{d\rho_t}$.
Using integration by parts, we then see that 
\begin{subequations}
\begin{align*}
\mathrm{MMD}_k^2&(\partial_t \rho^v_t, \partial_t \pi_t)  = \sum_{i,j=1}^d \int_{\mathbb{R}^d} \int_{\mathbb{R}^d} \frac{\partial^2 k(x,y)}{\partial x^i \partial y^j} v_i(x) v_j(y) \rho_t( \mathrm{d}x) \rho_t( \mathrm{d}y)
\\
& + 2 \int_{\mathbb{R}^d} \int_{\mathbb{R}^d} v(x) \cdot \nabla_x k(x,y)  \overline{h}(y) \rho_t(\mathrm{d}x) \rho_t(\mathrm{d}y) + \int_{\mathbb{R}^d} \int_{\mathbb{R}^d} \overline{h}(x) \overline{h}(y) \rho_t(\mathrm{d}x) \rho_t(\mathrm{d}y). 
\end{align*}
\end{subequations}
Because of the assumptions on $k$ and $h$, we conclude that $v \mapsto \mathrm{MMD}_k^2(\partial_t \rho^v_t, \partial_t \pi_t)$ is continuous for the $(L^2(\rho_t))^d$-topology, and therefore can be extended uniquely to $(L^2(\rho_t))^d$. We furthermore see that $J_{\varepsilon,t}$ is quadratic and hence strictly convex and coercive, and thus there exists a unique minimiser $v^* \in (L^2(\rho_t))^d$. 

We now use the regression interpretation suggested by \eqref{eq:regression} to connect \eqref{eq:MMD functional} to an appropriate Tikhonov functional \citep[Section 2.2]{kirsch2011introduction}. Following the discussion in Section \ref{sec:diffusion maps}, we consider the (unbounded) operator $${A = C\nabla :L^2(\rho_t) \rightarrow (L^2(\rho_t))^d}$$ and its formal adjoint $A^*: (L^2(\rho_t))^d \rightarrow L^2(\rho_t)$, given by $A^* v = -\tfrac{1}{\rho_t} \nabla \cdot (\rho_t v)$. Using the inclusion operator $i: \mathcal{H}_k \hookrightarrow L^2(\rho_t)$ and its adjoint (see Section \ref{sec:diffusion maps}), we consider the composition 
\begin{equation*}
Ai : \left(\mathcal{H}_k, \langle \cdot, \cdot \rangle_{\mathcal{H}_k} \right) \rightarrow \left((L^2(\rho_t))^d, \langle \cdot, \cdot \rangle_{L^2(\rho_t),C} \right).
\end{equation*}
Because of the boundedness and differentiability conditions on $k$, this operator is bounded, see  \citet{zhou2008derivative} or \citet[Corollary 4.36]{steinwart2008support}, and so is its adjoint $(Ai)^*$.
According to \citet[Theorem 2.11]{kirsch2011introduction}, and using the Woodbury (or `push-through') identity, it follows that $v^*$ as in \eqref{eq:regression1} minimises the Tikhonov functional 
\begin{equation}
\label{eq:Tikhonov}
\Vert i^* (A^*v - \overline{h}) \Vert^2_{\mathcal{H}_k}   + \Vert v \Vert^2_{L^2(\rho_t), C}. 
\end{equation}
To finish the proof, it is thus sufficient to show that $\Vert i^* (A^*v - \overline{h}) \Vert^2_{\mathcal{H}_k} = \mathrm{MMD}_k^2(\partial_t \rho^v_t, \partial_t \pi_t)$.
For this, notice that 
\begin{equation}
\label{eq:i star Hk}
 \Vert i^*f \Vert^2_{\mathcal{H}_k}  =  \int_{\mathbb{R}^d} \int_{\mathbb{R}^d} f(x) k(x,y) f(y) \rho_t(\mathrm{d}x) \rho_t(\mathrm{d}y), \qquad f \in L^2(\rho_t),
\end{equation}
by \citet[Theorem 4.26]{steinwart2008support}. The claim follows by comparing \eqref{eq:i star Hk} with \eqref{eq:MMD proof}, and using the fact that $A^* v = -\tfrac{1}{\rho_t} \nabla \cdot (\rho_t v)$.

\vskip 0.2in
\bibliography{ref}

\end{document}